\documentclass{amsart}%
\usepackage{amssymb}
\usepackage{amsmath}
\usepackage{amsfonts}
\usepackage[margin=1.4in]{geometry}
\usepackage[all,cmtip]{xy}
\usepackage{diagxy}
\usepackage[backref]{hyperref}%
\usepackage{graphicx}

\newtheorem{theorem}{Theorem}
\numberwithin{theorem}{section}
\theoremstyle{plain}
\newtheorem{definition}[theorem]{Definition}
\newtheorem*{acknowledgement}{Acknowledgement}

\newtheorem*{conjecture}{Conjecture}

\newtheorem{lemma}[theorem]{Lemma}
\newtheorem*{fact}{Lemma}

\newtheorem{proposition}[theorem]{Proposition}

\newtheorem{corollary}[theorem]{Corollary}
\newtheorem*{Dedication}{Dedication}
\numberwithin{equation}{section}
\theoremstyle{remark}
\newtheorem{remark}[theorem]{Remark}
\newtheorem{keypoint}[theorem]{Key Point}

\theoremstyle{definition}
\newtheorem{example}[theorem]{Example}
\newtheorem*{dangerousbend}{Dangerous Bend}
\begin{document}
\title[Geometric and analytic structures]{Geometric and analytic structures\linebreak on the higher ad\`{e}les}
\dedicatory{To Professor Fedor Bogomolov on the occasion of his 70th birthday}

\author{O. Braunling,\quad M. Groechenig,\quad J. Wolfson}

\address{Department of Mathematics, Universit\"{a}t Freiburg, Germany}
\email{oliver.braeunling@math.uni-freiburg.de}
\address{Department of Mathematics, Imperial College London, UK}
\email{m.groechenig@imperial.ac.uk}
\address{Department of Mathematics, University of Chicago, USA}
\email{wolfson@math.uchicago.edu}

\thanks{O.B.\ was supported by GK 1821 \textquotedblleft Cohomological Methods in
Geometry\textquotedblright. M.G.\ was partially supported by EPSRC Grant
No.\ EP/G06170X/1. J.W.\ was partially supported by an NSF Post-doctoral
Research Fellowship under Grant No.\ DMS-1400349. Our research was supported
in part by EPSRC Mathematics Platform grant EP/I019111/1.{\newline \newline The final version was published in: Braunling, O., Groechenig, M. and Wolfson, J. Res Math Sci (2016) 3: 22. https://doi.org/10.1186/s40687-016-0064-y; Special Collection in Celebration of the Research of Fedor Bogomolov on the Occasion of his 70th Birthday}}

\begin{abstract}
The ad\`{e}les of a scheme have local components $-$ these are topological
higher local fields. The topology plays a large role since Yekutieli showed in
1992 that there can be an abundance of inequivalent topologies on a higher
local field and no canonical way to pick one. Using the datum of a topology,
one can isolate a special class of continuous endomorphisms. Quite
differently, one can bypass topology entirely and single out special
endomorphisms (global Beilinson-Tate operators) from the geometry of the
scheme. Yekutieli's \textquotedblleft Conjecture 0.12\textquotedblright%
\ proposes that these two notions agree. We prove this.

\end{abstract}
\maketitle

The ad\`{e}les of a scheme $X$ \cite{MR565095} generalize the classical
ad\`{e}les of Chevalley and Weil. The counterpart of a prime/finite place is a
saturated flag of scheme points%
\[
\triangle:=(\eta_{0}>\cdots>\eta_{n})\qquad\eta_{i}\in X
\]
with $\eta_{i+1}$ a codimension one point of $\overline{\{\eta_{i}\}}$. The
counterpart of the local field at a prime becomes a higher local field $K $,
see Theorem \ref{intro_thm_1} below. Suppose $X$ is of finite type over a
field $k$. In dimension one, the classical case, a local field has a canonical
topology and thus comes with a canonical algebra of continuous $k$-linear
endomorphisms, call it $E_{K}$. Sadly, this collapses dramatically for
$\dim(X)\geq2$: The ad\`{e}les induce a topology on the higher local fields
$K$. But as was discovered by Yekutieli \cite{MR1213064} in 1992, this
topology is an additional datum. It cannot be recovered from knowing $K$
solely as a field. However, even if we know this topology, $K$ is no longer a
topological field or ring. So it becomes quite unclear how to define the
continuous endomorphism algebra $E_{K}$ for $\dim(X)\geq2$. Approaches are:

\begin{enumerate}
\item \textit{(\textquotedblleft Global BT operators\textquotedblright)}
Beilinson defines $E_{\triangle}^{\operatorname*{Beil}}$ using a flag
$\triangle$ in the scheme.

\item \textit{(\textquotedblleft Local BT operators\textquotedblright)}
Yekutieli defines $E_{K}^{\operatorname*{Yek}}$ for a topological higher local
field $K$.

\item \textit{(\textquotedblleft}$n$-\textit{Tate objects\textquotedblright)
}Ad\`{e}les can be viewed as an $n$-Tate object \cite{bgwTateModule}, and let
$E_{\triangle}^{\operatorname*{Tate}}$ be its endomorphism algebra in this category.
\end{enumerate}

Yekutieli has shown that if $k$ is perfect, a flag $\triangle$ as in (1) also
induces a topological higher local field structure, as in (2). So while a
priori different, this suggests the following

\begin{conjecture}
[A. Yekutieli]\footnote{\textquotedblleft Conjecture 0.12\textquotedblright%
\ of \cite{MR3317764}}Let $k$ be a perfect field. Suppose $X/k$ is a finite
type $k$-scheme of pure dimension $n$ and $\triangle:=(\eta_{0}>\cdots
>\eta_{n})$ a saturated flag of points. Then there is a canonical isomorphism%
\[
E^{\operatorname*{Yek}}\cong E^{\operatorname*{Beil}}\text{.}%
\]

\end{conjecture}

\begin{theorem}
\label{Conj012}If $X$ is reduced, the Conjecture is true. Even better,%
\[
E^{\operatorname*{Yek}}\cong E^{\operatorname*{Beil}}\cong
E^{\operatorname*{Tate}}\text{,}%
\]
i.e. all three constructions of the endomorphism algebra give canonically
isomorphic results.
\end{theorem}

See Theorem \ref{Thm_ComparisonOfCubicalAlgebras} for the precise result $-$
the above statement is simplified since a careful formulation requires some
preparations which we cannot supply in the introduction.

The theorem establishes a key merit of the $n$-Tate categories of
\cite{TateObjectsExactCats}, namely that $E^{\operatorname*{Yek}}$ and
$E^{\operatorname*{Beil}}$ become \textquotedblleft
representable\textquotedblright\ in the sense that despite the original
hand-made constructions of these algebras, they are nothing but
\textit{genuine} $\operatorname*{End}(-)$\textit{-algebras} of an exact
category.\medskip

Our principal technical ingredient elaborates on the well-known structure
theorem for the ad\`{e}les. The original version is due to A. N. Parshin
\cite{MR0419458} (in dimension $\leq2$), A. A. Beilinson \cite{MR565095}
(proof unpublished), and the first published proof due to A.\ Yekutieli
\cite{MR1213064}. The following version extends his result with regards to the
ind-pro structure of the ad\`{e}les \cite{TateObjectsExactCats}. We write
$A_{X}(\triangle,\mathcal{-})$ to denote the ad\`{e}les of the scheme $X$ for
a flag $\triangle$. Notation is as in \cite{MR565095}. In particular we write
$\triangle^{\prime}$ to denote removing the initial entry from a flag
$\triangle$.

\begin{theorem}
\label{intro_thm_1}Suppose $X$ is a Noetherian reduced excellent scheme of
pure dimension $n$ and $\triangle=\{(\eta_{0}>\cdots>\eta_{n})\}$ a saturated flag.

\begin{enumerate}
\item Then $A_{X}(\triangle,\mathcal{O}_{X})$ is a finite direct product of
$n$-local fields $\prod K_{i}$ such that each last residue field is a finite
field extension of $\kappa(\eta_{n})$. Moreover,%
\[
A_{X}(\triangle^{\prime},\mathcal{O}_{X})\overset{(\ast)}{\subseteq}%
{\textstyle\prod}
\mathcal{O}_{i}\subseteq%
{\textstyle\prod}
K_{i}=A_{X}(\triangle,\mathcal{O}_{X})\text{,}%
\]
where $\mathcal{O}_{i}$ denotes the first ring of integers of $K_{i}$ and
$(\ast)$ is a finite ring extension.

\item These sit in a canonical staircase-shaped diagram%
\[%
\bfig\node x(0,1200)[{A_{\overline{\{\eta_{0}\}}}(\triangle,\mathcal{O}_{X})}]
\node y(0,900)[{A_{\overline{\{\eta_{0}\}}}(\triangle^{\prime},\mathcal{O}%
_{X})}]
\node z(800,900)[{A_{\overline{\{\eta_{1}\}}}(\triangle^{\prime},\mathcal
{O}_{X})}]
\node w(800,600)[{A_{\overline{\{\eta_{1}\}}}(\triangle^{\prime\prime
},\mathcal{O}_{X})}]
\node u(1600,600)[{A_{\overline{\{\eta_{2}\}}}(\triangle^{\prime\prime
},\mathcal{O}_{X})}]
\node v(1600,300)[\vdots]
\arrow/{^{(}->}/[y`x;]
\arrow/{->>}/[y`z;]
\arrow/{^{(}->}/[w`z;]
\arrow/{->>}/[w`u;]
\arrow[v`u;]
\efig
\]

\item If $X$ is finite type over a field $k$, each field factor $K:=K_{i}$ in
(1) is (non-canonically) isomorphic as rings\footnote{but not necessarily
$k$-algebras!} to Laurent series,%
\[
K\overset{\sim}{\longrightarrow}\kappa((t_{1}))((t_{2}))\cdots((t_{n}))
\]
for $\kappa/k$ a finite field extension. This isomorphism can be chosen such
that it is simultaneously an isomorphism

\begin{enumerate}
\item of $n$-local fields,

\item of $n$-Tate objects with values in abelian groups,

\item (if $k$ is perfect) of $k$-algebras,

\item (if $k$ is perfect) of $n$-Tate objects with values in
finite-dimensional $k$-vector spaces,

\item (if $k$ is perfect) of topological $n$-local fields in the sense of Yekutieli.
\end{enumerate}

\item Still assume that $X$ is finite type over a field $k$. After replacing
each ring in (2), except the initial upper-left one, by a canonically defined
finite ring extension, it splits canonically as a direct product of
staircase-shaped diagrams of rings: Each factor has the shape%
\[%
\bfig\node x(0,1200)[{\kappa((t_{1}))\cdots((t_{n}))}]
\node y(0,900)[{\kappa((t_{1}))\cdots[[t_{n}]]}]
\node z(900,900)[{\kappa((t_{1}))\cdots((t_{n-1}))}]
\node w(900,600)[{\kappa((t_{1}))\cdots[[t_{n-1}]]}]
\node u(1800,600)[{\kappa((t_{1}))\cdots((t_{n-2}))}]
\node v(1800,300)[\vdots]
\arrow/{^{(}->}/[y`x;]
\arrow/{->>}/[y`z;]
\arrow/{^{(}->}/[w`z;]
\arrow/{->>}/[w`u;]
\arrow[v`u;]
\efig
\]
under any isomorphism as produced by (3).

\begin{enumerate}
\item The upward arrows are going to the field of fractions,

\item The rightward arrows correspond to passing to the residue
field.\footnote{Moreover, these maps are induced from the corresponding upward
and rightward arrows in (2), but due to the finite ring extensions interfering
here, the precise nature of this is a little too subtle to make precise in the
introduction.}
\end{enumerate}

These are continuous/admissible epics resp. monics in both Yekutieli's
category of ST modules, as well as $n$-Tate objects.

\item If $X$ is finite type over a \textsl{perfect} field $k$, then for each
field factor $K$, the notions of lattices (\`{a} la Beilinson, resp.
Yekutieli, resp. Tate) need not agree, but are pairwise final and co-final
(\textquotedblleft Sandwich property\textquotedblright).
\end{enumerate}
\end{theorem}

We refer to the main body of the text for notation and definitions. The reader
will find these results in \S \ref{sect_StructTheorems}, partially in greater
generality than stated here. See \cite[3.3.2-3.3.6]{MR1213064} for Yekutieli's
result inspiring the above. Parts (3)-(5) appear to be new results.\medskip

These results focus on the case of schemes over a field, and as we shall
explain below, are truly complicated only in the case of characteristic zero.
Note also that, since we mostly work over a base field, our considerations are
of a geometric/analytic nature, rather than an arithmetic one. Also, no
thoughts on infinite places will appear here. See \cite{MR2658047} for
ad\`{e}les directed towards arithmetic considerations.

Let us explain the relevance of (3): Yekutieli has already shown in
\cite{MR1213064} via an explicit example that in characteristic zero a random
field automorphism of an $n$-local field $K$ is frequently not continuous.
Using Yekutieli's technique in our context leads us to the following variation
of his idea:

\begin{theorem}
Suppose an $n$-local field $K$ is equipped with an $n$-Tate object structure
in $k$-vector spaces. If $\operatorname*{char}(k)=0$ and $n\geq2$, then not
every field automorphism of $K$ will preserve the $n$-Tate object structure.
\end{theorem}

See Example \ref{example_YekutieliIndPro}. Jointly with Yekutieli's original
example, this shows that in (3), the validity of property (a) does not imply
(b)-(e) being true as well.\medskip

The definitions of the endomorphism algebras $E_{\triangle}%
^{\operatorname*{Tate}},E_{K}^{\operatorname*{Yek}},E_{\triangle
}^{\operatorname*{Beil}}$ all hinge on notions of \textit{lattices}, whose
definitions we shall address later in the paper. We shall show that the
different notions of lattices used by Beilinson, Yekutieli or coming from Tate
objects are all pairwise distinct. One might state the comparison as:
\[
\text{Beilinson lattices\quad}\subsetneqq\text{\quad Yekutieli lattices\quad
}\subsetneqq\text{\quad Tate lattices,}%
\]
modulo a slight abuse of language since these types of lattices live in
different objects. See \S \ref{sect_DiffTypesOfLattices} for an example
demonstrating this. If these notions all agreed, this would have resulted in a
particularly easy proof of Theorem \ref{Conj012}.\medskip

Let us survey the relation among the central players of this paper: Let us
assume the base field $k$ is perfect.%
\begin{equation}%
{\includegraphics[
height=2.5953in,
width=3.5578in
]%
{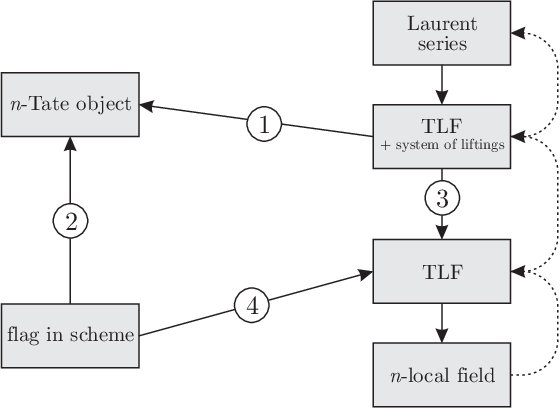}%
}
\label{lBigFigure}%
\end{equation}
The solid arrows refer to a canonical construction. Each dashed arrow
expresses that a structure can non-canonically be enriched with additional
structure. By an \textquotedblleft$n$-Tate object\textquotedblright\ we mean
an $n$-Tate object in finite-dimensional $k$-vector spaces. By an
\textquotedblleft$n$-local field\textquotedblright\ we refer to an
equicharacteristic $n$-local field with last residue field finite over $k$. By
\textquotedblleft flag in scheme\textquotedblright\ we refer to the ad\`{e}les
$A(\triangle,\mathcal{O}_{X})$, for a saturated flag $\triangle$ in a suitable
scheme $X$ of finite type over $k$. By \textquotedblleft TLF\textquotedblright%
\ we refer to a topological $n$-local field in the sense of Yekutieli. By
\textquotedblleft Laurent series\textquotedblright\ we refer to $k^{\prime
}((t_{1}))\cdots((t_{n}))$ with $k^{\prime}/k$ a finite field extension.

Arrow (1) refers to a certain construction $\sharp_{\sigma}$, established in
Theorem \ref{Theorem_TLFPlusSysOfLiftingsYekIsTate}. Arrow (2) refers to the
canonical $n$-Tate object structure of the ad\`{e}les from
\cite{TateObjectsExactCats}. The downward solid arrows on the right, in
particular Arrow (3), just refer to forgetting additional structure. Arrow (4)
refers to Yekutieli's construction of the TLF\ structure on the ad\`{e}les
\cite{MR1213064}.

\begin{dangerousbend}
It is a priori not clear that a TLF\ can be equipped with a system of liftings
inverting Arrow (3) such that we would get a commutative diagram.
\end{dangerousbend}

However, a different way to state the innovation in Theorem \ref{intro_thm_1}
is that it is possible to pick an isomorphism of $A(\triangle,\mathcal{O}%
_{X})$ with a Laurent series field such that we arrive at the same objects, no
matter which path through Figure \ref{lBigFigure} we choose. That is, no
matter through which arrows we produce an $n$-Tate object (resp. TLF), we get
the same object.

The objects in Figure \ref{lBigFigure} come with three (a priori different)
endomorphism algebras:

\begin{itemize}
\item $E^{\operatorname*{Beil}}$ of the flag of the scheme, global
Beilinson-Tate operators.

\item $E_{\sigma}^{\operatorname*{Yek}}$ of a TLF\ with a system of liftings
$\sigma$, local Beilinson-Tate operators.

\item $E^{\operatorname*{Tate}}$ the \textit{genuine endomorphisms} in the
category of $n$-Tate objects, i.e. really just a plain $\operatorname*{Hom}%
$-group. This, by the way, shows the conceptual advantage of working with
$n$-Tate categories.
\end{itemize}

A deep result of Yekutieli, quoted below as Theorem
\ref{Thm_YekIndependenceOfSystemOfLiftings}, shows that $E_{\sigma
}^{\operatorname*{Yek}}$ does not depend on $\sigma$, so we can speak of
$E^{\operatorname*{Yek}}$ of a TLF. Our paper \cite{bgwTateModule} shows that
Arrow (2) induces an isomorphism $E^{\operatorname*{Beil}}\cong
E^{\operatorname*{Tate}}$. Yekutieli's Conjecture asks whether Arrow (4)
induces an isomorphism $E^{\operatorname*{Beil}}\cong E^{\operatorname*{Yek}}%
$. We prove this in Theorem \ref{Thm_ComparisonOfCubicalAlgebras}.

In \S \ref{section_StandAloneHLFs}, we prove that Arrow (1) induces an
isomorphism $E^{\operatorname*{Yek}}\cong E^{\operatorname*{Tate}}$. This is a
result of independent interest. It touches a slightly different aspect than
Yekutieli's Conjecture since it refers to the $n$-Tate structure produced by
Arrow (1), while the conjecture is about the $n$-Tate structure of Arrow (2).
By Theorem \ref{intro_thm_1} we know that we can find a system of liftings
such that both $n$-Tate structures match, and this yields a proof of
Yekutieli's Conjecture.

\begin{acknowledgement}
We would like to thank Amnon Yekutieli and Alberto C\'{a}mara for carefully
reading an earlier version of the manuscript and their very insightful
remarks. Our category-oriented viewpoint has been shaped by Mikhail Kapranov.
Moreover, we heartily thank Alexander Beilinson, Fedor Bogomolov, and Ivan
Fesenko, whose encouragement and interest in our work was pivotal.\medskip
\newline We thank the anonymous referee for his/her very careful review, which
led to a significant improvement of the presentation and some simplification
of the arguments.
\end{acknowledgement}

\begin{Dedication}
This paper is dedicated to Fedor Bogomolov. We thank him for his interest, and
for discussions during the \textquotedblleft Symmetries and
Correspondences\textquotedblright\ Conference in July 2014, where Yekutieli
presented his conjecture. Fedor Bogomolov's almost surreal originality and
playful creativity have a tremendous impact on the maths community, but beyond
that, his papers radiate a great joy in doing mathematics, which is very
inspiring and impossible to resist. Happy birthday!
\end{Dedication}

\section{The topology problem for local
fields\label{SECT_TopologyProblemForHigherLocalFields}}

In this section we shall introduce the main players of the story. We will use
this opportunity to give a survey over many (not even all) of the approaches
to give higher local fields a topology or at least a structure replacing a
topology. This issue is surprisingly subtle and many results are scattered
over the literature.

\subsection{Na\"{\i}ve topology}

A complete discrete valuation field $K$ with the valuation $v$ comes with a
canonical topology, which we shall call the \emph{na\"{\i}ve topology},
namely: Take the sets $U_{i}:=\{x\in K\mid v(x)\geq i\}$ as an open
neighbourhood basis of the identity. This topology is highly intrinsic to the
algebraic structure.

We recall the crucial fact that a field cannot be a complete discrete
valuation field with respect to several valuations:

\begin{lemma}
[F. K.\ Schmidt]\label{lemma_CompleteDiscValuationsAreUnique}If a field $K$ is
complete with respect to a discrete valuation $v$,

\begin{enumerate}
\item then every discrete valuation on $K$ is equivalent to $v$;

\item any isomorphism of such fields stems from a unique isomorphism of their
rings of integers;

\item and is automatically continuous (in the na\"{\i}ve topology).
\end{enumerate}
\end{lemma}

See Morrow \cite[\S 1]{MorrowHLF}, who has very clearly emphasized the
importance of this uniqueness statement. A thorough study of such and related
questions can be found in the original paper of Schmidt \cite{MR1512831}.

\begin{proof}
For the sake of completeness, we give an argument, an alternative to the one
in \cite{MorrowHLF}: (1) Let $w$ be a further discrete valuation, not
equivalent to $v$, and $\pi_{w}$ a uniformizer for it. By the Approximation
Theorem \cite[Ch. I, (3.7) Prop.]{MR1915966} one can pick an element $x\in K$
so that%
\[
w(x-\pi_{w})\geq1\qquad\text{and}\qquad v(x-1)\geq1\text{.}%
\]
By the first property, $w(x)\geq1$. By the latter $x=1+a$ for some
$a\in\mathfrak{m}\mathcal{O}_{K}$ and if $l\geq2$ is any integer (such that
$l\nmid\operatorname*{char}(\mathcal{O}_{K}/\mathfrak{m})$ in case
$\mathcal{O}_{K}/\mathfrak{m}$ has positive characteristic), the series
$(1+a)^{1/l^{n}}:=\sum_{r=0}^{\infty}\tbinom{1/l^{n}}{r}a^{r}$ converges,
showing that $x$ is $l$-divisible. So $w(x)\in\mathbf{Z}$ is $l$-divisible,
forcing $w(x)=0$. This is a contradiction. (2) follows since the valuation
determines the ring of integers, (3) follows since the na\"{\i}ve topology is
defined solely in terms of the valuation.
\end{proof}

\begin{definition}
[Parshin \cite{MR0401710}, \cite{MR514485}, Kato \cite{MR517332}%
]\label{def_NLocalField}For $n\geq1$, an $n$\emph{-local field} \emph{with
last residue field} $k$ is a complete discrete valuation field $K$ such that
if $(\mathcal{O}_{1},\mathfrak{m})$ denotes its ring of integers,
$\mathcal{O}_{1}/\mathfrak{m}$ is an $(n-1)$-local field with last residue
field $k$. A $0$-local field with last residue field $k$ is just $k$ itself.
\end{definition}

Inductively unravelling this definition, every $n$-local field $K$ gives rise
to the following staircase-shaped diagram%
\begin{equation}%
\bfig\node x(0,1200)[K]
\node y(0,900)[\mathcal{O}_{1}]
\node z(300,900)[k_1]
\node w(300,600)[\mathcal{O}_{2}]
\node u(600,600)[k_2]
\node v(600,300)[\vdots,]
\arrow/{^{(}->}/[y`x;]
\arrow/{->>}/[y`z;]
\arrow/{^{(}->}/[w`z;]
\arrow/{->>}/[w`u;]
\arrow[v`u;]
\efig
\label{lda14}%
\end{equation}
where the $\mathcal{O}_{i}$ denote the respective rings of integers, and the
$k_{i}$ the residue fields. We call the integers $(\operatorname*{char}%
K,\operatorname*{char}k_{1},\ldots,\operatorname*{char}k_{n})$ the
\emph{characteristic} of $K$.

\begin{corollary}
\label{Cor_FieldIsoOfNLocalFieldsIsNLocalFieldIso}Fix $n\geq0$.

\begin{enumerate}
\item If a field $K$ possesses the structure of an $n$-local field at all, it
is unique.

\item If $K\overset{\sim}{\longrightarrow}K^{\prime}$ is a field isomorphism
of $n$-local fields, it is automatically continuous in the na\"{\i}ve topology
and induces isomorphisms of its residue fields,%
\[
k_{i}\overset{\sim}{\longrightarrow}k_{i}^{\prime}\text{,}%
\]
each also continuous in the na\"{\i}ve topology, as well as an isomorphism of
last residue fields $k\overset{\sim}{\longrightarrow}k^{\prime}$.
\end{enumerate}
\end{corollary}

\begin{proof}
This follows by induction from Lemma
\ref{lemma_CompleteDiscValuationsAreUnique}.
\end{proof}

Note that the number $n$ is not uniquely determined. An $n$-local field is
always also an $r$-local field for all $0\leq r\leq n$.

\begin{example}
If $k$ is any field, the multiple Laurent series field $k((t_{1}%
))\cdots((t_{n}))$ is an example of an $n$-local field with last residue field
$k$. It has characteristic $(0,\ldots,0)$ or $(p,\ldots,p)$ depending on
$\operatorname*{char}(k)=0$ or $p$. The field $\mathbf{Q}_{p}((t_{1}%
))\cdots((t_{n}))$ is an example of an $(n+1)$-local field with last residue
field $\mathbf{F}_{p}$. It has characteristic $(0,\ldots,0,p)$. See
\cite{MR1804915} for many more examples.
\end{example}

Let $(R,\mathfrak{m})$ be a complete Noetherian local domain and
$\mathfrak{m}$ its maximal ideal. A \emph{coefficient field} is a sub-field
$F$ so that the composition $F\hookrightarrow R\twoheadrightarrow
R/\mathfrak{m}$ is an isomorphism of fields.

\begin{proposition}
[Cohen's Structure Theorem]\label{Prop_CohenStructureTheorem}Let
$(R,\mathfrak{m})$ be a complete Noetherian local domain and $\mathfrak{m}$
its maximal ideal.

\begin{enumerate}
\item If $R$ contains a field (at all), a coefficient field exists.

\item If $\operatorname*{char}(R)=\operatorname*{char}(R/\mathfrak{m})$, a
coefficient field exists.

\item If $F$ is any coefficient field and $x_{1},\ldots,x_{r}\in\mathfrak{m}$
a system of parameters,%
\begin{align*}
F[[t_{1},\ldots,t_{r}]]  & \hookrightarrow R\\
t_{i}  & \mapsto x_{i}%
\end{align*}
is injective and $R$ is a finite module over its image. If $R$ is regular, one
can find $x_{1},\ldots,x_{r}\in\mathfrak{m}$ such that the corresponding
injection becomes an isomorphism of rings%
\[
F[[t_{1},\ldots,t_{r}]]\overset{\sim}{\longrightarrow}R\text{.}%
\]

\item (\cite[Theorem 1.1]{MR3317764}) Suppose $k$ is a perfect field and $R$ a
$k$-algebra. Then one can find a coefficient field $F$ containing $k$ and such
that $F\hookrightarrow R$ is a $k$-algebra morphism. If the residue field
$R/\mathfrak{m}$ is finite over $k$, there is only one coefficient field
having this additional property.
\end{enumerate}
\end{proposition}

This stems from Cohen's famous paper \cite{MR0016094}. Many more modern
references exist, e.g. \cite[Thm. 4.3.3]{MR2266432} for an overview, \cite[Ch.
11]{MR575344} or \cite[\S 29 and \S 30]{MR1011461} for the entire story. See
Yekutieli's paper \cite[\S 1]{MR3317764} for (4).

An immediate consequence, modulo an easy induction, is the following (simple)
excerpt of the classification theory for higher local fields:

\begin{proposition}
[Classification]\label{prop_CohenStructureThmForEquicharNLocalFields}Let $K$
be an $n$-local field with last residue field $k$ such that all fields
$K,k_{i}$ have the same characteristic. Then there exists a non-canonical
isomorphism of fields%
\[
K\simeq k((t_{1}))\cdots((t_{n}))\text{.}%
\]

\end{proposition}

If the characteristic is allowed to change, the classification of $n$-local
fields is significantly richer. We refer the reader to \cite[Ch. II,
\S 5]{MR1915966} for the structure theory of complete discrete valuation
fields, going well beyond the amount needed here. For the $n$-local field
case, see \cite{MR1804916}, \cite{MR2388492}, \cite[\S 0, Theorem]{MR1363290}
or \cite{MorrowHLF}. For our purposes here, the above version is sufficient.

\subsection{Systems of liftings}

Suppose $K$ is a complete discrete valuation field with ring of integers
$\mathcal{O}$ and residue field $\kappa:=\mathcal{O}/\mathfrak{m}$. By Cohen's
Structure Theorem, if $\operatorname*{char}(K)=\operatorname*{char}(\kappa)$,
there exists a coefficient field $F\hookrightarrow\mathcal{O}$, in other
words, the quotient map to the residue field%
\[
\mathcal{O}\twoheadrightarrow\kappa
\]
admits a section in the category of rings.

\begin{example}
\label{example_ManyCoeffFields}Such a section is usually very far from unique.
Consider $K=k(s)((t))$, a $1$-local field with last residue field $k(s)$. Take
any element in the maximal ideal $\alpha\in t\cdot k(s)[[t]]$. Then%
\[
k(s)\rightarrow k(s)((t))\text{,}\qquad s\mapsto s+\alpha
\]
defines a coefficient field. These are different whenever different $\alpha$
are chosen. Yekutieli has a much more elaborate version of this construction,
producing an enormous amount of coefficient fields for the $2$-local field
$k((s))((t))$ with $\operatorname*{char}(k)=0$. See Example
\ref{example_Yekutieli} or \cite[Ex. 2.1.22]{MR1213064}.
\end{example}

\begin{example}
Suppose $K$ is an equicharacteristic complete discrete valuation field. If and
only if the residue field is either (1) an algebraic extension of $\mathbf{Q}%
$, or (2) a perfect field of positive characteristic, then there is only one
possible choice for the coefficient field \cite[Ch. II \S 5.2-\S 5.4]%
{MR1915966}. In all other cases there will be a multitude of coefficient fields.
\end{example}

There is a straightforward extension of the concept of a coefficient field to
$n$-local fields.

\begin{definition}
\label{def_AlgSysOfLiftings}Let $K$ be an $n$-local field. An \emph{algebraic
system of liftings} $(\sigma_{1},\ldots,\sigma_{n})$ is a collection of ring
homomorphisms%
\[
\sigma_{i}:k_{i}\rightarrow\mathcal{O}_{i}%
\]
which are sections to the residue field quotient maps $\mathcal{O}%
_{i}\twoheadrightarrow k_{i}$.
\end{definition}

This concept appears for example in \cite[\S 1, p. 112]{MR726423},
\cite{MR1363290}, \cite{MR3317764}.

\begin{example}
[Madunts,\ Zhukov]\label{example_MaduntsZhukovLift}By Example
\ref{example_ManyCoeffFields}, an $n$-local field will surely have many
systems of liftings if $n\geq2$, and possibly as well if $n=1$, depending on
the last residue field. Still, if the last residue field is a finite field,
and we choose uniformizers $t_{1},\ldots,t_{n}$ for the rings of integers
$\mathcal{O}_{1},\ldots,\mathcal{O}_{n}$, Madunts and Zhukov \cite[\S 1]%
{MR1363290} isolate a distinguished, canonical, system of liftings
$h_{t_{1},\ldots,t_{n}}$ for all $n$-local fields which are either (1)
equicharacteristic $(p,\ldots,p)$ with $p>0$ some prime, or (2) mixed
characteristic $(0,p,\ldots,p)$ for some prime. This construction does not
work for example for $k((t_{1}))\cdots((t_{n}))$ with $\operatorname*{char}%
(k)=0$, or the $2$-local field $\mathbf{Q}_{p}((t))$ of characteristic
$(0,0,p)$. See \cite[\S 1.3]{MR1804916} for a survey. These liftings depend on
the choice of $t_{1},\ldots,t_{n}$.
\end{example}

\subsection{Minimal higher topology}

The na\"{\i}ve topology comes with a major drawback: Already for the multiple
Laurent series field $k((t_{1}))\cdots((t_{n}))$ the formal series notation%
\[
\sum a_{i_{1}\ldots i_{n}}t_{1}^{i_{1}}t_{2}^{i_{2}}\cdots t_{n}^{i_{n}}%
\]
of an arbitrary element is usually \textit{not} convergent in the topology
once $n\geq2$. The problem is that the topology is only made from the top
valuation, sensitive to the exponent of $t_{n}$, but gives the first residue
field $-$ when viewed as a sub-field $-$ the discrete topology. Also, the
algebraic quotient maps $\mathcal{O}_{i}\twoheadrightarrow k_{i}$ are not
topological quotient maps, i.e. they do not induce the quotient topology on
$k_{i}$. The Laurent polynomials $k[t_{1}^{\pm1},\ldots,t_{n}^{\pm}]$ are not
dense for $n\geq2$. This is a new phenomenon and complication in the case
$n\geq2$, which cannot be seen in the classical theory for $n=1$. Dealing with
this type of behaviour required some new ideas, and Parshin proposed to equip
$n$-local fields with a different topology \cite[p. 145, bottom]{MR752939}.

\begin{example}
[Parshin]\label{example_TopRingIsImpossible}There is a strong limitation to
the properties a reasonable topology on $K:=k((t_{1}))((t_{2}))$ can have, in
the shape of the following obstruction: Assume $\mathcal{T}$ is any topology
making the additive group $(K;+)$ a topological group and such that the
quotient topology induced from%
\begin{equation}
\mathcal{O}_{1}\twoheadrightarrow\mathcal{O}_{1}/\mathfrak{m}\text{,}%
\qquad\text{i.e.}\qquad k((t_{1}))[[t_{2}]]\twoheadrightarrow k((t_{1}%
))\label{lwaw1}%
\end{equation}
equips $k((t_{1}))$ precisely with the na\"{\i}ve topology. Then $K$ is
\textsl{not} a topological ring in this topology \cite[Remark 1 on p.
147]{MR752939}: Suppose it were. The map in Equation \ref{lwaw1} is continuous
by assumption, so the subsets $U+t_{2}k[[t_{2}]]$ are open in $k((t_{1}%
))[[t_{2}]]$ if and only if $U\subseteq k((t_{1}))$ is open. As multiplication
with powers of $t_{2}$ would be continuous, this enforces the following: For
$(U_{i})_{i\in\mathbf{Z}}$ a sequence of open neighbourhoods of the identity
in $k((t_{1}))$ such that $U_{i}=k((t_{1}))$ for all sufficiently large $i$,
then the sets%
\begin{equation}
V:=%
{\textstyle\sum_{i}}
U_{i}t_{2}^{i}\subseteq K\label{lwaw2}%
\end{equation}
must be open. These are finite sums of $t_{2}$-translates of sets we already
know must be open. The following figure illustrates the nature of these open
sets; the shaded range symbolizes those exponents $(i_{1},i_{2})$ whose
monomials $t_{1}^{i_{1}}t_{2}^{i_{2}}$ are allowed to carry a non-zero
coefficient:%
\[%
{\includegraphics[
height=1.0081in,
width=2.209in
]%
{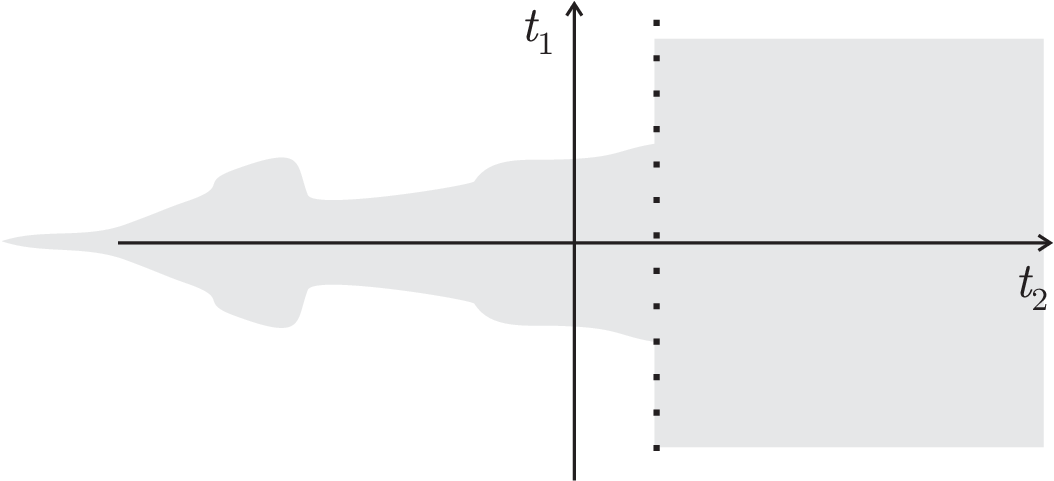}%
}
\]
This continues ad infinitum to the left; perhaps thinning out but never
terminating. The dotted line marks the index such that $U_{i}=k((t_{1}))$ for
all larger $i$. Now we observe that $V\cdot V=K$ is the entire field (under
multiplication the condition $U_{i}=k((t_{1}))$ for large $i$ compensates that
the open neighbourhoods may thin out to the left). Thus, if multiplication
$K\times K\rightarrow K$ were continuous, the pre-image of some open $U\subset
K$ would have to be open, thus contain some diagonal Cartesian open $V\times
V$, but we just saw that multiplication maps this to all of $K$. See
\cite{MR1804916}, \cite{MR1850194} for a further analysis. For example, this
observation extends to show that the multiplicative group $K^{\times}$ of an
$n$-local field cannot be a topological group for $n\geq3 $ \cite{MR1804933}.
\end{example}

It appears that the consensus of the practitioners in the field is that it is
better to have a reasonable topology than insisting on working with
topological rings, which carry an almost meaningless topology. Parshin
\cite{MR752939} then developed the theory by taking the open sets of the shape
in Equation \ref{lwaw2} as the general definition of a topology for the field
$F((t))$:\ If the additive group $(F;+)$ is equipped with a topological group
structure, generate an additive group topology on $F((t))$ from the sets
$V_{(U_{i})}$ of the shape%
\[
V_{(U_{i})}:=%
{\textstyle\sum_{i}}
U_{i}t^{i}\subseteq F((t))
\]
for $(U_{i})_{i\in\mathbf{Z}}$ open neighbourhoods of the identity in $F$ and
$U_{i}=F$ for $i$ large enough. This is explained in more detail in
\cite[\S 1]{MR1363290}, \cite{MR1804916}. Giving $k$ the discrete topology,
this inductively equips $k((t_{1}))\cdots((t_{n}))$ with a canonical topology.
We call it Parshin's\emph{\ natural topology} (there does not appear to be a
standard name in the literature; e.g. Abrashkin and his students call it the
`$P$-topology' \cite[\S 1.2]{MR2357687}). For $n\geq2$, the natural topology
has quite different opens than the na\"{\i}ve topology.

If $K$ is an equicharacteristic $n$-local field with last residue field $k$,
Prop. \ref{prop_CohenStructureThmForEquicharNLocalFields} provides an
isomorphism $\phi$ to such a multiple Laurent series field:%
\[
K\underset{\phi}{\overset{\sim}{\longrightarrow}}k((t_{1}))((t_{2}%
))\cdots((t_{n}))\text{.}%
\]
Sadly, as was discovered by Yekutieli in 1992 (see Example
\ref{example_Yekutieli} below), the induced topology usually \textsl{depends}
on the choice of the isomorphism. That means, switching to a different $\phi$
will frequently equip $K$ with a truly different topology. We shall return to
this crucial issue in \S \ref{subsect_YekutieliSTRings}.

\begin{example}
[Madunts, Zhukov]The situation is slightly better if we are in the situation
of Example \ref{example_MaduntsZhukovLift}. If $K$ is an $n$-local field,
equicharacteristic $(p,\ldots,p)$ with $p>0$, and the last residue field is
finite, Madunts and Zhukov define a topology (extending Parshin's natural
topology) based on their canonical lift $h_{t_{1},\ldots,t_{n}}$, cf. Example
\ref{example_MaduntsZhukovLift}, and in a second step prove that the topology
is independent of the choice of $t_{1},\ldots,t_{n}$ \cite[Thm. 1.3]%
{MR1363290}. This also works for $n$-local fields of characteristic
$(0,p,\ldots,p)$ and finite last residue field. Such a construction is not
available for example for $k((t_{1}))\cdots((t_{n}))$ with
$\operatorname*{char}(k)=0$. In fact, Example \ref{example_Yekutieli}, due to
A. Yekutieli, shows that no such generalization can possibly exist.
\end{example}

Before we continue this line of thought, we discuss a further development of
the natural topology:

\subsection{\label{subsect_SeqSpaces}Sequential spaces}

Working with the natural topology, at least multiplication by a \textit{fixed}
element from the left or right are continuous, and one has
\[
x_{n}\longrightarrow x\text{, }y_{n}\longrightarrow y\qquad\Longrightarrow
\qquad x_{n}\cdot y_{n}\longrightarrow x\cdot y\text{,}%
\]
i.e. the multiplication is continuous if one only tests it on sequences.
Following this lead, Fesenko modified the natural topology into a new one in
which continuity is detected by sequential continuity alone. We sketch the
implications of this:

We recall that a subset $Z\subset X$ of a topological space $X$ is called
\emph{sequentially closed} if for every sequence $(x_{n})$ with $x_{n}\in Z$,
convergent in $X$, the limit $\lim\nolimits_{n}x_{n}$ also lies in $Z$.

\begin{definition}
[Franklin]A topological space is called \emph{sequential} if a subset is
closed iff it is sequentially closed.
\end{definition}

Franklin shows that equivalently sequential spaces are those spaces which
arise as quotients of metric spaces \cite[(1.14) Corollary]{MR0180954}. The
inclusion admits a right adjoint, called \emph{sequential saturation},%
\[
\mathsf{Top}_{\operatorname{seq}}\overset{sat}{\leftrightarrows}\mathsf{Top}%
\]
between the category $\mathsf{Top}$ (resp. $\mathsf{Top}_{\operatorname{seq}}%
$) of all (resp. sequential) topological spaces.

\begin{definition}
[Fesenko]The\emph{\ saturation topology} on $k((t_{1}))\cdots((t_{n}))$ is the
sequential saturation of the natural topology. \cite{MR1850194}.
\end{definition}

This topology has many more open sets than the natural topology in general
(see \cite[(2.2) Remark]{MR1850194} for an explicit example), but a sequence
is convergent in the saturation topology if and only if it converges in the
natural topology. This is no contradiction since these topologies do not admit
countable neighbourhood bases. Example \ref{example_TopRingIsImpossible}
implies that we still cannot have a topological ring. However, we get
something like a `sequential topological ring'. But this really is a
completely different notion than a topological ring because ring objects in
sequential spaces are not compatible with ring objects in topological spaces
by the following example:

\begin{example}
[Dudley, Franklin]The categories $\mathsf{Top}_{\operatorname{seq}}$ and
$\mathsf{Top}$ have products, but they do not agree, i.e.%
\[
(X_{sat}\times_{\mathsf{Top}}Y_{sat})\neq(X\times_{\mathsf{Top}}%
Y)_{sat}\text{.}%
\]
Explicit examples were given independently by Dudley and Franklin. See
\cite[Example 1.11]{MR0180954} for the latter. We refer to \cite{MR0180954},
\cite{MR0222832} for a detailed study.
\end{example}

\begin{remark}
For $n=1$, the na\"{\i}ve, natural and saturation topology on $k((t))$ all agree.
\end{remark}

\begin{remark}
Analogously to the case of higher local fields, the ad\`{e}les of a scheme can
also be equipped with sequential topologies \cite{MR2658047}, \cite{FesRR}.
\end{remark}

\begin{remark}
A detailed exposition and elaboration on the notions of sequential groups and
rings was given by A. C\'{a}mara \cite[\S 1]{camaratopology}. He also studies
a further topological approach. In \cite{MR3161556}, \cite{MR3227342} he shows
that $n$-local fields can also be viewed as locally convex topological vector
spaces if one fixes a suitable embedding of a local field, serving as the
`field of scalars'. The interested reader should consult A. C\'{a}mara for
further information, much of which is not available in published form.
\end{remark}

\subsection{Kato's ind-pro approach\label{subsect_KatoIndProApproach}}

Kato \cite[\S 1]{MR726423} proposed that the concept of topology might in
general not be the right framework to think about continuity in higher local
fields. In the introduction to \cite{MR1804933} he proposes very clearly to
abandon the idea of topology entirely, in favour of promoting the ind-pro
structure of higher local fields, e.g. as in%
\[
k((t))=\underset{i}{\underrightarrow{\operatorname*{colim}}}\underset
{j}{\underleftarrow{\lim}}\,t^{-i}k[[t]]/t^{j}\text{,}%
\]
to the essential datum. We note that this presentation of $k((t))$ is, in the
category of linear topological vector spaces, inducing the na\"{\i}ve
topology. Thus, the ind-pro perspective is another possible starting point to
find a good generalization of continuity to higher local fields.

Instead of just working with vector spaces, such an ind-pro viewpoint makes
sense for objects in almost any category. Let $\mathcal{C}$ be an exact
category, e.g. an abelian category. Then there is a category $\mathsf{Ind}%
_{\kappa}^{a}(\mathcal{C})$ of admissible Ind-objects (of cardinality
$\leq\kappa$), e.g. encoding objects defined by an inductive system%
\[
C_{1}\hookrightarrow C_{2}\hookrightarrow C_{3}\hookrightarrow\cdots
\]
with $C_{i}\in\mathcal{C}$ and admissible monics as transition morphisms.
Additionally, more complicated defining diagrams can be allowed. A precise
definition and construction is given in\ Keller \cite[Appendix B]{MR1052551}
or in greater generality \cite[\S 3]{TateObjectsExactCats}. Following Keller's
ideas, $\mathsf{Ind}_{\kappa}^{a}(\mathcal{C})$ is again an exact category and
an analogous formalism exists for Pro-objects, $\mathsf{Pro}_{\kappa}%
^{a}(\mathcal{C})$. See also Previdi \cite{MR2872533}. We shall frequently
drop the cardinality $\kappa$ from the notation for the sake of legibility.
These categories sit in a commutative square of inclusion functors%
\[%
\bfig\Square(0,0)[\mathcal{C}`\mathsf{Ind}^{a}\mathcal{C}`\mathsf{Pro}%
^{a}\mathcal{C}`\mathsf{Ind}^{a}\mathsf{Pro}^{a}\mathcal{C}\text{.};```]
\efig
\]
One may now replace $\mathsf{Ind}^{a}\mathsf{Pro}^{a}(\mathcal{C})$ by the
smallest sub-category still containing $\mathsf{Ind}^{a}(\mathcal{C})$ and
$\mathsf{Pro}^{a}(\mathcal{C})$, but also being closed under extensions. This
is again an exact category, called the category of \emph{elementary Tate
objects}, $\mathsf{Tate}^{el}(\mathcal{C})$ \cite{MR2872533},
\cite{TateObjectsExactCats}.

\begin{example}
\label{example_QpAsIndPro}Let $\mathsf{Ab}_{fin}$ be the abelian category of
finite abelian groups. In the category of all abelian groups $\mathsf{Ab}$ we
have%
\[
\mathbf{Q}_{p}=\underset{i}{\underrightarrow{\operatorname*{colim}}}%
\underset{j}{\underleftarrow{\lim}}\frac{1}{p^{i}}\mathbf{Z}/p^{j}%
\mathbf{Z}\text{,}\qquad\text{where}\qquad\frac{1}{p^{i}}\mathbf{Z}%
/p^{j}\mathbf{Z}\in\mathsf{Ab}_{fin}\text{.}%
\]
Instead of regarding this colimit/limit inside the category $\mathsf{Ab}$, we
could read the inner limit as a diagram $J_{i}:\mathbf{N}\rightarrow
\mathsf{Ab}_{fin}$, $j\mapsto\frac{1}{p^{i}}\mathbf{Z}/p^{j}\mathbf{Z}$,
defining an object in $\mathsf{Pro}^{a}(\mathsf{Ab}_{fin})$, and using the
dependency on $i$ we get a diagram $I:\mathbf{N}\rightarrow\mathsf{Pro}%
^{a}(\mathsf{Ab}_{fin})$, $i\mapsto\lbrack(J_{i})]$ of Pro-objects.
Considering the object defined by this diagram, we get an object
$I\in\mathsf{Ind}^{a}\mathsf{Pro}^{a}(\mathsf{Ab}_{fin})$. One can easily
check that it actually lies in $\mathsf{Tate}^{el}(\mathsf{Ab}_{fin})$, see
Definition \ref{Def_TateLattice} below. One can also define a functor
$\mathsf{Tate}^{el}(\mathsf{Ab}_{fin})\rightarrow\mathsf{Ab}$ which, using
that $\mathsf{Ab}$ is complete and co-complete, evaluates the Ind-Pro-object
described by these diagrams. This yields $\mathbf{Q}_{p}\in\mathsf{Ab}$ as
before. See \cite{TateObjectsExactCats} for more background. More examples
along these lines can be found in \cite{bgwTateModule}.
\end{example}

Kapranov made the justification of Kato's idea \cite{MR1804933} very precise:

\begin{example}
[Kapranov \cite{KapranovSemiInfinite}, \cite{MR1800352}]%
\label{example_Kapranov}If $\mathcal{C}:=\mathsf{Vect}_{f}(\mathbf{F}_{q})$ is
the abelian category of finite-dimensional $\mathbf{F}_{q}$-vector spaces,
$q=p^{n}$, Kapranov proved that there is an equivalence of categories
$\mathsf{Tate}^{el}(\mathcal{C})\overset{\sim}{\rightarrow}\mathsf{LT}$, where
$\mathsf{LT}$ is the category of linearly locally compact topological
$\mathbf{F}_{q}$-vector spaces \cite{MR0007093}, \cite{KapranovSemiInfinite}.
Every equicharacteristic $1$-local field with last residue field
$\mathbf{F}_{q}$ and equipped with the na\"{\i}ve topology is an object of
$\mathsf{LT}$. One can extend this example and interpret any $1$-local field
with last residue field $\mathbf{F}_{q}$ as an object of $\mathsf{Tate}%
^{el}(\mathcal{C})$ for $\mathcal{C}$ the category of finite abelian
$p$-groups, e.g. as in Example \ref{example_QpAsIndPro}.
\end{example}

The category $\mathsf{Tate}^{el}(\mathcal{C})$ can be described as those
objects $V\in\mathsf{Ind}^{a}\mathsf{Pro}^{a}(\mathcal{C})$ which admit an
exact sequence%
\begin{equation}
L\hookrightarrow V\twoheadrightarrow V/L\label{lccab11}%
\end{equation}
so that $L\in\mathsf{Pro}^{a}(\mathcal{C})$ and $V/L\in\mathsf{Ind}%
^{a}(\mathcal{C})$.

\begin{definition}
\label{Def_TateLattice}Any $L$ appearing in such an exact sequence will be
called a \emph{(Tate) lattice} in $V$.
\end{definition}

So Tate objects are those Ind-Pro-objects admitting a lattice. A category of
this nature was first defined by Kato \cite{MR1804933} in the 1980s (the
manuscript was published only much later), but without an exact category
structure, and independently by Beilinson \cite{MR923134} for a completely
different purpose $-$ Previdi proved the equivalence between Beilinson's and
Kato's approaches \cite{MR2872533}.

\begin{remark}
It is shown in \cite[Thm. 6.7]{TateObjectsExactCats} that for idempotent
complete $\mathcal{C}$, any finite set of lattices has a common sub-lattice
and a common over-lattice. This can vaguely be interpreted as counterparts of
the statement that finite unions and intersections of opens in a topological
space should still be open.
\end{remark}

Following Kato, this suggests to replace the topologically-minded category
$\mathsf{LT}$ (of Example \ref{example_Kapranov}) by $\mathsf{Tate}%
^{el}(\mathcal{C})$, and for example a $2$-local field over $\mathbf{F}_{q}$
should be viewed as something like%
\begin{equation}
\mathbf{F}_{q}((t_{1}))((t_{2}))\quad\in\quad\mathsf{Tate}^{el}(\mathsf{Tate}%
^{el}(\mathsf{Vect}_{f}))\text{.}\label{lwwax1}%
\end{equation}
Instead of concatenating lengthy expressions, we shall call this a `$2$-Tate
object' and more generally define the following:

\begin{definition}
Let $\mathcal{C}$ be an arbitrary exact category. Define $\left.
1\text{-}\mathsf{Tate}^{el}(\mathcal{C})\right.  :=\mathsf{Tate}%
^{el}(\mathcal{C})$, and $\left.  n\text{-}\mathsf{Tate}^{el}(\mathcal{C}%
)\right.  :=\mathsf{Tate}^{el}(\,(n-1)$-$\mathsf{Tate}(\mathcal{C})\,)$ and
$\left.  n\text{-}\mathsf{Tate}(\mathcal{C})\right.  $ as the idempotent
completion of the category $\left.  n\text{-}\mathsf{Tate}^{el}(\mathcal{C}%
)\right.  $. Objects in $n$-$\mathsf{Tate}(\mathcal{C})$ will be called
$n$\emph{-Tate objects}. \cite[\S 7]{TateObjectsExactCats}
\end{definition}

The slightly complicating presence of idempotent completions in this
definition makes the categories substantially nicer to work with. See
\cite{bgwTateModule} for many instances of this effect.

\begin{example}
[Kato]\label{example_KatoIndProOfLaurentSeries}Kato \cite[\S 1]{MR726423}
equips an $n$-local field $K$ along with a fixed algebraic system of liftings,
Definition \ref{def_AlgSysOfLiftings}, with the structure of an $n$-Tate
object in finite abelian groups. The definition depends on the system of
liftings. See \cite[\S 1.2]{MR1804933} for a detailed exposition. For multiple
Laurent series we can use%
\begin{align*}
& \text{\textquotedblleft}k((t_{1}))((t_{2}))\ldots((t_{n}%
))\text{\textquotedblright}\\
& \qquad\qquad=\underset{i_{n}}{\underrightarrow{\operatorname*{colim}}%
}\underset{j_{n}}{\underleftarrow{\lim}}\cdots\underset{i_{1}}%
{\underrightarrow{\operatorname*{colim}}}\underset{j_{1}}{\underleftarrow
{\lim}}\,\frac{1}{t_{1}^{i_{1}}\cdots t_{n}^{i_{n}}}k[t_{1},\ldots
,t_{n}]/(t_{1}^{j_{1}},\ldots,t_{n}^{j_{n}})\text{.}%
\end{align*}

\end{example}

\begin{example}
[Osipov]In the case of $\mathcal{C}:=\mathsf{Vect}_{f}$ a closely related
alternative model for $n$-Tate objects are the $C_{n}$-categories of Denis
Osipov \cite{MR2314612}. There is also a variant for $\mathcal{C}%
:=\mathsf{Ab}$ or including some abelian real Lie groups, the categories
$C_{n}^{\operatorname*{fin}}$ or $C_{n}^{\operatorname*{ar}}$ of
\cite{MR2866188}.
\end{example}

Kato's approach differs quite radically from the others. Since the concept of
a topology is not used at all, it seems at first sight very unclear how one
could even formulate any sort of `comparison' between the ind-pro versus
topological viewpoint.

\subsection{Yekutieli's ST rings\label{subsect_YekutieliSTRings}}

Yekutieli's approach, first introduced in \cite{MR1213064}, uses topology
again. However, instead of just looking at fields, he directly formulates an
appropriate weakening of the concept of a topological ring for quite general
(even non-commutative) rings.

For the moment, let $k$ be any ring and it will tacitly be understood as a
topological ring with the discrete topology. Yekutieli works with his notion
of \emph{semi-topological rings} (\emph{ST\ rings}): An ST ring is a
$k$-algebra $R$ along with a $k$-linear topology on its underlying $k$-module
such that for any given $r\in R$ both one-sided multiplication maps%
\[
(r\cdot-):R\longrightarrow R\qquad\text{and}\qquad(-\cdot r):R\longrightarrow
R
\]
are continuous. We follow his notation and write $\mathsf{STRing}(k)$ for this
category. Morphisms are continuous $k$-algebra homomorphisms. See
\cite[\S 1]{MR3317764} for a review of the theory. The material is developed
in full detail in \cite[Chapter 1]{MR1213064}.

\begin{example}
[C\'{a}mara]The left- and right-continuity is also a feature of both the
natural and the saturation topology. In particular $k((t_{1}))\cdots((t_{n}))$
with the natural topology lies in $\mathsf{STRing}(k)$. By a result of
C\'{a}mara, this is no longer true for the saturation topology. In more
detail: The topology on Yekutieli's ST\ rings is always linear, i.e. admits an
open neighbourhood basis made from additive sub-groups/or sub-modules.
C\'{a}mara's theorem \cite[Theorem 2.9 and Corollary]{camaratopology} shows
that the saturation topology from \S \ref{subsect_SeqSpaces} is not a linear
topology. For a $2$-local field he shows that if one takes the topology
generated only from those saturation topology opens which are simultaneously
sub-groups, one recovers the natural topology.
\end{example}

Similarly, an ST\ module $M$ is an $R$-module along with a linear topology on
its additive group such that for any given $r\in R$ and $m\in M$ the maps%
\[
(r\cdot-):M\longrightarrow M\qquad\text{and}\qquad(-\cdot m):R\longrightarrow
M
\]
are continuous. This additive $k$-linear category is denoted by
$\mathsf{STMod}(R)$. Yekutieli already points out that this category is not
abelian. Although he does not phrase it this way, his results also imply that
the situation is not too bad either:

\begin{proposition}
\label{prop_STmodsQuasiAbelian}For any ST ring $R$, the category
$\mathsf{STMod}(R)$ is quasi-abelian in the sense of J.-P. Schneiders
\cite{MR1779315}.
\end{proposition}

\begin{proof}
Yekutieli already shows in \cite[Chapter 1]{MR1213064} that the category is
additive and has all kernels and cokernels. So one only has to check that
pushouts preserve strict monics and pullbacks preserve strict epics. These
verifications are immediate.
\end{proof}

We get a functor to ordinary modules by forgetting the topology and Yekutieli
shows \cite[\S 1.2 and Prop. 1.2.4]{MR1213064} that it has a left adjoint
\[
\mathsf{STMod}(R)\underset{\mathrm{forget}}{\overset{\mathrm{fine}%
}{\leftrightarrows}}\mathsf{Mod}(R)\text{,}%
\]
where `$\mathrm{fine}$' equips an $R$-module $M$ with the so-called \emph{fine
ST topology}, the finest linear topology such that $M$ is an ST module at all
(it exists by \cite[Lemma 1.1.1]{MR1213064}). Being a left adjoint,
`$\mathrm{fine}$' commutes with colimits.

\begin{example}
[Yekutieli]\label{example_YekutieliLaurentSTRingStructure}Yekutieli defines an
ST\ ring structure on multiple Laurent series:%
\begin{equation}
k((t_{1}))((t_{2}))\cdots((t_{n}))\in\mathsf{STRing}(k)\text{.}\label{lwawwx1}%
\end{equation}
His construction is as follows: Write it as%
\[
\underset{i_{n}}{\underrightarrow{\operatorname*{colim}}}\underset{j_{n}%
}{\underleftarrow{\lim}}\cdots\underset{i_{1}}{\underrightarrow
{\operatorname*{colim}}}\underset{j_{1}}{\underleftarrow{\lim}}\,\frac
{1}{t_{1}^{i_{1}}\cdots t_{n}^{i_{n}}}k[t_{1},\ldots,t_{n}]/(t_{1}^{j_{1}%
},\ldots,t_{n}^{j_{n}})
\]
and (1) equip the inner term with the fine ST $k$-module topology, (2) for the
limits use that the inverse limit linear topology of ST topologies is again an
ST topology \cite[Lemma 1.2.19]{MR1213064}, (3) the colimits are
localizations, equip them with the fine topology over the ring we are
localizing; this makes them ST rings again \cite[Prop. 1.2.9]{MR1213064}. See
\cite[Def. 1.17 and Def. 3.7]{MR3317764} for the details.
\end{example}

Semi-topological rings ultimately remain a very subtle working ground. On the
one hand, they behave very well with respect to many natural questions (e.g.
Yekutieli develops inner Homs, shows a type of Matlis duality, etc., see
\cite{MR1213064}, \cite{MR1352568}). On the other hand, just as for sequential
spaces, \S \ref{subsect_SeqSpaces}, harmless looking constructions can fail
badly, e.g. \cite[Remark 1.29]{MR3317764}.

\begin{example}
[Yekutieli]\label{example_Yekutieli}In \cite[Ex. 2.1.22]{MR1213064} Yekutieli
exhibited an example greatly clarifying the problem underlying the search for
a canonical topology on $n$-local fields. A detailed exposition is given in
\cite[Ex. 3.13]{MR3317764}. We sketch the construction since we shall need to
refer to some of its ingredients later: Suppose $\operatorname*{char}(k)=0$
and let $\{b_{i}\}_{i\in I}$ be a transcendence basis for $k((t_{1}))/k$ with
$b_{i}\in k[[t_{1}]]$; such exists since if $b_{i}\notin k[[t_{1}]]$, replace
it by $b_{i}^{-1}$. The index set $I$ will necessarily be infinite. Then for
any choice of elements $c_{i}\in k((t_{1}))[[t_{2}]]$, $i\in I$, Yekutieli
constructs a map%
\begin{align}
\sigma:k((t_{1}))  & \longrightarrow k((t_{1}))[[t_{2}]]\label{lwwax3}\\
b_{i}  & \longmapsto b_{i}+c_{i}t_{2}\text{,}\nonumber
\end{align}
which is a particular choice of a coefficient field (On the purely
$k$-transcendental sub-field%
\[
k(\{b_{i}\}_{i\in I})\longrightarrow k((t_{1}))[[t_{2}]]
\]
the existence of this map is clear right away. Lifting this morphism along the
algebraic extension $k((t_{1}))/k(\{b_{i}\}_{i\in I})$ is the subtle point and
hinges on $\operatorname*{char}(k)=0$ \cite{MR3317764}). We may assume
$b_{0}=t_{1}$ and $c_{0}=0$ for some index $0\in I$, so that $\sigma$ maps
$t_{1}$ to itself. Yekutieli shows that $\sigma$ lifts to a field automorphism
$\tilde{\sigma}$ of $k((t_{1}))((t_{2}))$ sending one such coefficient field
to another and $t_{2}$ to itself. Since the sub-field $k(t_{1},t_{2})$ is
element-wise fixed by $\tilde{\sigma}$, but is dense in the natural topology,
Fesenko's saturation topology and Yekutieli's ST\ topology, $\tilde{\sigma}$
will \textsl{not} be continuous unless all $c_{i}$ are zero. It follows that
if $K$ is an $n$-local field and%
\[
\phi:K\simeq k((t_{1}))\cdots((t_{n}))
\]
some field isomorphism $\phi$ from Prop.
\ref{prop_CohenStructureThmForEquicharNLocalFields}, the topology pulled back
from the right-hand side to $K$ depends on the choice of $\phi$, because we
could twist this map with arbitrary discontinuous automorphisms $\tilde
{\sigma}$.
\end{example}

\begin{example}
\label{example_YekutieliIndPro}We use this paper as an opportunity to unravel
a variation of Yekutieli's example in order to show that Kato's ind-pro
structure, as explained in Example \ref{example_KatoIndProOfLaurentSeries},
will also not be preserved by a random field automorphism. We assume at least
a passing familiarity with \cite{TateObjectsExactCats}. Recall that
$\mathsf{Vect}_{f}$ denotes the abelian category of finite-dimensional
$k$-vector spaces. Again, suppose $\operatorname*{char}(k)=0$. Consider
Yekutieli's map $\sigma$, as in Equation \ref{lwwax3}, and recall that we can
choose the $c_{i}$ quite arbitrarily. We will use this now: Pick any
surjective set-theoretic map $Q:I\twoheadrightarrow\mathbf{Z}$. Such a map
exists since the indexing set $I$ is infinite. We take%
\begin{equation}
c_{i}:=t_{1}^{Q(i)}\in k((t_{1}))[[t_{2}]]\label{lwwax7}%
\end{equation}
We write either side as a $2$-Tate object in finite-dimensional $k$-vector
spaces $\left.  2\text{-}\mathsf{Tate}(\mathsf{Vect}_{f})\right.  $, as in
Example \ref{example_KatoIndProOfLaurentSeries}. If $\tilde{\sigma}$ is
induced from a morphism of $2$-Tate objects, it is in particular an
automorphism of a $1$-Tate object (namely, a $1$-Tate object with values in
$1$-Tate objects, see Equation \ref{lwwax1}), namely of%
\[
\underset{i_{2}}{\underrightarrow{\operatorname*{colim}}}\underset{j_{2}%
}{\underleftarrow{\lim}}\,\underset{=:V_{i_{2},j_{2}}}{\underbrace{\frac
{1}{t_{2}^{i_{2}}}k((t_{1}))[[t_{2}]]/(t_{2}^{j_{2}})}}\text{.}%
\]
This in turn is true if and only if for every pair $(i_{2},j_{2}^{\prime})$
there exists a pair $(i_{2}^{\prime},j_{2})$ so that $\tilde{\sigma}$
restricts to%
\begin{equation}
\tilde{\sigma}\mid_{(i_{2},j_{2})}:V_{i_{2},j_{2}}\longrightarrow
V_{i_{2}^{\prime},j_{2}^{\prime}}\text{.}\label{lwwax2}%
\end{equation}
If this is the case, the converse translation is as follows: these
$\tilde{\sigma}\mid_{(i_{2},j_{2})}$, for each $i_{2}$ fixed and varying over
$j_{2} $, induce a morphism of Pro-diagrams (see \cite[\S 4.1, Def.
4.1]{TateObjectsExactCats} for a definition), and then varying over $i_{2}$
they induce a morphism of Tate diagrams, made from these Pro-diagrams (see
\cite[Def. 5.2]{TateObjectsExactCats} for a definition). This in turn gives
the desired morphism of Tate objects. Unravel Equation \ref{lwwax2} in the
case $i_{2}:=0$ and take any $j_{2}^{\prime}$ (we may imagine taking this
arbitrarily large, if we want) so that we have the existence of indices
$i_{2}^{\prime}$ and $j_{2}$ with
\begin{align}
\tilde{\sigma}\mid_{(0,j_{2})}:k((t_{1}))[[t_{2}]]/(t_{2}^{j_{2}})  &
\longrightarrow\frac{1}{t_{2}^{i_{2}^{\prime}}}k((t_{1}))[[t_{2}%
]]/(t_{2}^{j_{2}^{\prime}})\text{,}\nonumber\\
b_{i}  & \longmapsto b_{i}+t_{1}^{Q(i)}\cdot t_{2}\text{.}\label{lwwax5}%
\end{align}
The restriction of this morphism in the category $\mathsf{Tate}(\mathsf{Vect}%
_{f})$ to the lattice $k[[t_{1}]]$ becomes%
\begin{equation}
\tilde{\sigma}\mid_{(0,j_{2})}:k[[t_{1}]][[t_{2}]]/(t_{2}^{j_{2}%
})\longrightarrow\frac{1}{t_{2}^{i_{2}^{\prime}}}k((t_{1}))[[t_{2}%
]]/(t_{2}^{j_{2}^{\prime}})\text{.}\label{lwwax4}%
\end{equation}
But lattices are Pro-objects. Thus, by \cite[Prop. 5.8]{TateObjectsExactCats}
the morphism $\tilde{\sigma}\mid_{(0,j_{2})}$ factors through a Pro-subobject
$L$ of the right-hand side%
\[%
\bfig\node x(0,0)[{k[[t_{1}]]}]
\node y(1000,0)[{\frac{1}{t_{2}^{i_{2}^{\prime}}}k((t_{1}))[[t_{2}%
]]/(t_{2}^{j_{2}^{\prime}})}]
\node l(1000,500)[L]
\arrow/{>}/[x`l;]
\arrow/{^{(}->}/[l`y;]
\arrow/{>}/[x`y;\tilde{\sigma}]
\efig
\]
Alternatively one could use the following stronger fact: For a morphism of
Tate objects, morphisms originating from a lattice factor through a lattice in
the target \cite[Prop. 2.7 (1)]{bgwRelativeTateObjects}. Now, the Pro-system%
\begin{equation}
\left(  m\longmapsto\frac{1}{t_{2}^{i_{2}^{\prime}}}\frac{1}{t_{1}^{m}%
}k[[t_{1}]][[t_{2}]]/(t_{2}^{j_{2}^{\prime}})\right)  \qquad\text{in}%
\qquad\mathsf{Pro}^{a}(\mathsf{Vect}_{f})\label{lwwax6}%
\end{equation}
is a co-final system of lattices in the target, so in particular the image of
$\tilde{\sigma}\mid_{(0,j_{2})}\mid_{k[[t_{1}]]}$ as in Equation \ref{lwwax4}
would have to factor over some object in this system. As we could assume
$b_{i}\in k[[t_{1}]]$ for all $i\in I$ in Example \ref{example_Yekutieli} and
$Q$ is surjective, Equation \ref{lwwax5}%
\begin{align*}
k[[t_{1}]][[t_{2}]]/(t_{2}^{j_{2}})  & \longrightarrow\frac{1}{t_{2}%
^{i_{2}^{\prime}}}k((t_{1}))[[t_{2}]]/(t_{2}^{j_{2}^{\prime}})\\
b_{i}  & \longmapsto b_{i}+t_{1}^{Q(i)}\cdot t_{2}%
\end{align*}
produces a contradiction since arbitrarily negative powers of $t_{1}$ lie in
the image of this map, but each of the lattices in the system in Equation
\ref{lwwax6} only has $t_{1}$ powers with an overall lower bound on the
exponent. In other words: Even though $\tilde{\sigma}$ exists as a field
automorphism, there is no automorphism of $2$-Tate objects inducing it.
\end{example}

We summarize: A general field automorphism of the $2$-local field
$k((t_{1}))((t_{2}))$ for $\operatorname*{char}(k)=0$ need not preserve (1)
the natural or saturation topologies, (2) Yekutieli's ST\ topology, (3) or
Kato's $2$-Tate object structure.

We thank Denis Osipov for pointing out to us that those automorphisms which
preserve the $n$-Tate structure of Laurent series $k((t_{1}))\cdots((t_{n})) $
are also automatically continuous in all of the aforementioned topologies
\cite[Prop. 2.3, (i)]{MR2314612}. See also Example
\ref{example_PreserveTateImpliesPreserveTopology}.

\begin{remark}
[Characteristic $p>0$]Contrary to the usual intuition, the situation is much
simpler in positive characteristic $p>0$:

\begin{enumerate}
\item (Kato) Kato produces a canonical ind-pro structure. See \cite[\S 1.1,
Prop. 2 \& Example]{MR1804933}.

\item (Madunts, Zhukov) The paper \cite{MR1363290} constructs a canonical
topology, following Parshin.

\item (Yekutieli) Yekutieli proves that all field isomorphisms between
equicharacteristic $n$-local fields of positive characteristic $p>0$ must
automatically be continuous, i.e. isomorphisms in $\mathsf{STRing}(k)$
\cite[Prop. 2.1.21]{MR1213064}. This is based on a surprising idea using
differential operators. See \cite[Thm. 2.1.14 and Prop. 2.1.21]{MR1213064}.
\end{enumerate}

Despite these positive results, it still seems reasonable to approach the
uniqueness problem for the topology for arbitrary $n$-local fields without
using this work-around in positive characteristic.
\end{remark}

Example \ref{example_Yekutieli} and Example \ref{example_YekutieliIndPro}
suggest that looking at $n$-local fields per se, there are too many
automorphisms to make reasonable and especially canonical use of topological
concepts. As a result, Yekutieli proposes to rigidify the category of
$n$-local fields by choosing and fixing a topology on them. This will be an
extra datum. Working in this context, one can restrict one's attention to
those field automorphisms which are also continuous. This greatly cuts down
the size of the automorphism group: For an $n$-local field, we define the ring%
\[
O(K):=\mathcal{O}_{1}\times_{k_{1}}\mathcal{O}_{2}\times_{k_{2}}\cdots
\times_{k_{n-1}}\mathcal{O}_{n}\subset K\text{.}%
\]
It consists of those elements in $\mathcal{O}_{1}$ whose residual image lies
in $\mathcal{O}_{2}$ such that their residual image lies in $\mathcal{O}_{3}$
and so forth.

\begin{definition}
[Yekutieli]\label{def_TLF}Let $k$ be a perfect field. A \emph{topological }%
$n$\emph{-local field} (TLF) consists of the following data:

\begin{enumerate}
\item an $n$-local field $K$ as in Definition \ref{def_NLocalField},

\item a topology $\mathcal{T}$ on $K$ which makes it an ST\ ring,

\item a ring homomorphism $k\rightarrow O(K)$ such that the composition
$k\rightarrow O(K)\rightarrow k_{n}$ is a finite extension of fields;
\end{enumerate}

and we assume there exists a (non-canonical, not part of the datum) field
isomorphism%
\[
\phi:k((t_{1}))((t_{2}))\cdots((t_{n}))\overset{\sim}{\longrightarrow}K
\]
which is also an isomorphism in $\mathsf{STRing}(k)$, where the left-hand side
is equipped with the standard ST\ ring structure, as explained in Example
\ref{example_YekutieliLaurentSTRingStructure}.

A morphism of TLFs is a field morphism, which is simultaneously an ST ring
morphism and preserves the $k$-algebra structure given by (3).
\end{definition}

Any such isomorphism $\phi$ will be called a \emph{parametrization}. We wish
to stress that the parametrization is not part of the data. We only demand
that an isomorphism exists at all. See \cite[\S 3]{MR3317764},
\cite{MR1352568} for a detailed discussion of TLFs.

\begin{dangerousbend}
Despite the name, a `topological $n$-local field' is \textsl{not} a field
object (or even ring object) in the category $\mathsf{Top}$.
\end{dangerousbend}

\begin{example}
Since Yekutieli's Example \ref{example_Yekutieli} shows that a general field
automorphism $\phi$ will not be continuous in the ST ring topology, it implies
that it will not be a TLF automorphism.
\end{example}

\begin{remark}
All of these approaches to topologization not only apply to higher local
fields, but are also natural techniques to equip similar algebraic structures
with a topology, e.g. double loop Lie algebras $\mathfrak{g}((t_{1}%
))((t_{2}))$ \cite{MR2276855}.
\end{remark}

\section{Ad\`{e}les of schemes}

In \S \ref{SECT_TopologyProblemForHigherLocalFields} we have introduced higher
local fields and their topologies. In the present section we shall recall one
of the most natural sources producing these structures: the ad\`{e}les of a
scheme. Mimicking the classical one-dimensional theory of Chevalley and Weil,
this construction is due to Parshin in dimension two \cite{MR0419458}, and
then was extended to arbitrary dimension by Beilinson \cite{MR565095}.

\subsection{\label{subsect_AdeleSheaves}Definition of Parshin-Beilinson
ad\`{e}les}

We follow the notation of the original paper by Beilinson \cite{MR565095}. We
assume that $X$ is a Noetherian scheme. For us, any closed subset of $X$
tacitly also denotes the corresponding closed sub-scheme with the reduced
sub-scheme structure, e.g. for a point $\eta\in X$ we write $\overline
{\{\eta\}}$ to denote the reduced closed sub-scheme whose generic point is
$\eta$. For points $\eta_{0},\eta_{1}\in X$, we write $\eta_{0}>\eta_{1}$ if
$\overline{\{\eta_{0}\}}\ni\eta_{1}$, $\eta_{1}\neq\eta_{0}$. Denote by
$S\left(  X\right)  _{n}:=\{(\eta_{0}>\cdots>\eta_{n}),\eta_{i}\in X\}$ the
set of nondegenerate chains of length $n+1$. Let $K_{n}\subseteq S\left(
X\right)  _{n}$ be an arbitrary subset.

We will allow ourselves to denote the ideal sheaf of the reduced closed
sub-scheme $\overline{\{\eta\}}$ by $\eta$ as well. This allows a slightly
more lightweight notation and is particularly appropriate for affine schemes,
where the $\eta$ are essentially just prime ideals.

For any point $\eta\in X$, define $\left.  _{\eta}K\right.  :=\{(\eta
_{1}>\cdots>\eta_{n})$ s.t. $(\eta>\eta_{1}>\cdots>\eta_{n})\in K_{n}\}$, a
subset of $S\left(  X\right)  _{n-1}$. Let $\mathcal{F}$ be a
\textit{coherent} sheaf on $X$. For $n=0$ and $n\geq1$ respectively, we define
inductively%
\begin{align}
A(K_{0},\mathcal{F}):=  &
{\displaystyle\prod\nolimits_{\eta\in K_{0}}}
\underleftarrow{\lim}_{i}\,\mathcal{F}\otimes_{\mathcal{O}_{X}}\mathcal{O}%
_{X,\eta}/\eta^{i}\label{TATEMATRIX_l6}\\
A(K_{n},\mathcal{F}):=  &
{\displaystyle\prod\nolimits_{\eta\in X}}
\underleftarrow{\lim}_{i}\,A(\left.  _{\eta}K_{n}\right.  ,\mathcal{F}%
\otimes_{\mathcal{O}_{X}}\mathcal{O}_{X,\eta}/\eta^{i})\text{.}\nonumber
\end{align}
For a \textit{quasi-coherent} sheaf $\mathcal{F}$, we define%
\begin{equation}
A(K_{n},\mathcal{F}):=\underrightarrow{\operatorname*{colim}}_{\mathcal{F}%
_{j}}A(K_{n},\mathcal{F}_{j})\text{,}\label{lTX1}%
\end{equation}
where $\mathcal{F}_{j}$ runs through all coherent sub-sheaves of $\mathcal{F}%
$. As it is built successively from ind-limits and countable Mittag-Leffler
pro-limits, $A(K_{n},-)$ is an exact functor from the category of
quasi-coherent sheaves to the category of $\mathcal{O}_{X}$-module sheaves. We
state the following fact in order to provide some background, but it will not
play a big role in this paper:

\begin{theorem}
[{Beilinson \cite[\S 2]{MR565095}}]\label{lX_BeilinsonResolutionThm}For a
Noetherian scheme $X$, and a quasi-coherent sheaf $\mathcal{F}$ on $X$, there
is a functorial resolution%
\begin{equation}
0\longrightarrow\mathcal{F}\longrightarrow\mathcal{A}^{0}\longrightarrow
\mathcal{A}^{1}\longrightarrow\mathcal{A}^{2}\longrightarrow\cdots
\label{lBeilRes}%
\end{equation}
in the category of $\mathcal{O}_{X}$-module sheaves, made from the flasque
sheaves defined by $\mathcal{A}^{i}(U):=A(S\left(  U\right)  _{i}%
,\mathcal{F})$.
\end{theorem}

We will not go into further detail. See Huber \cite{MR1105583},
\cite{MR1138291} for a detailed proof (the only proof available in print, as
far as we know) as well as further background.

\begin{example}
If $X/k$ is an integral proper curve, the complex \ref{lBeilRes} for
$\mathcal{F}:=\mathcal{O}_{X}$\ becomes%
\[
0\longrightarrow\mathcal{O}_{X}\longrightarrow\underline{k(X)}\oplus
\underset{x\in U_{0}}{\left.
{\textstyle\prod}
\right.  }\widehat{\mathcal{O}_{x}}\longrightarrow\underset{x\in U_{0}%
}{\left.
{\textstyle\prod\nolimits^{\prime}}
\right.  }\widehat{\mathcal{K}_{x}}\longrightarrow0\text{,}%
\]
where $\underline{k(X)}$ is the sheaf of rational functions, $U_{0}$ is the
set of closed points in any open $U$ (read these terms as sheaves in $U$),
$\widehat{\mathcal{K}_{x}}:=\operatorname*{Frac}\widehat{\mathcal{O}_{x}}$. In
particular, we obtain $H^{i}(X,\mathcal{O}_{X})$ as the cohomology of the
global sections of this flasque resolution. Note that the global sections of
the right-most term just correspond to the classical ad\`{e}les of the curve.
Hence, the Parshin-Beilinson ad\`{e}les really extend the classical framework.
As discussed in \S \ref{SECT_TopologyProblemForHigherLocalFields} the fields
$\widehat{\mathcal{K}_{x}}$ have a well-defined intrinsic topology, just
because they are $1$-local fields. For $\dim X\geq2$, we would get higher
local fields and the question of a topology begins to play a significant role.
\end{example}

\begin{remark}
[Other ad\`{e}le theories]In this paper, whenever we speak of
\textquotedblleft ad\`{e}les\textquotedblright, we will refer to the
Parshin-Beilinson ad\`{e}les as described in this section, or the papers
\cite{MR565095}, \cite{MR1138291}. There are other notions of ad\`{e}les as
well: First of all, the Parshin-Beilinson ad\`{e}les truly generalize the
classical ad\`{e}les only in the function field case: the ad\`{e}les of a
number field feature the infinite places as a very important ingredient, and
these are not covered by the Parshin-Beilinson formalism. In a different
direction, for us a higher local field has a ring of integers in each of its
residue fields, corresponding to a valuation taking values in the integers.
However, one can also look at this story from the perspective of higher-rank
valuations, i.e. taking values in $\mathbf{Z}^{r}$ with a lexicographic
ordering. This yields further, more complicated, rings of integers, along with
corresponding notions of ad\`{e}les. See Fesenko \cite{MR2046602},
\cite{MR2658047}. Finally, instead of allowing just quasi-coherent sheaves as
coefficients, one may also allow other sheaves as coefficients. See for
example \cite{MR2489487}, \cite{MR3366858}.
\end{remark}

\subsection{Local endomorphism algebras}

We axiomatize the basic algebraic structure describing well-behaved
endomorphisms, for example of $n$-local fields, or vector spaces over
$n$-local fields. In particular, this will apply to $n$-local fields built
from the ad\`{e}les.

\begin{definition}
\label{def_BeilinsonDefinitionNFoldCubicalAlgebra}A \emph{Beilinson }%
$n$\emph{-fold cubical algebra} is

\begin{enumerate}
\item an associative unital\footnote{For some applications it can be sensible
to allow non-unital $A$ as well, but we would not have a use for this level of
generality here.} $k$-algebra $A$;

\item two-sided ideals $I_{i}^{+},I_{i}^{-}$ such that we have $I_{i}%
^{+}+I_{i}^{-}=A$ for $i=1,\ldots,n$.
\end{enumerate}
\end{definition}

This structure appears in \cite{MR565095}, but does not carry a name in loc.
cit. In all examples of relevance to us, $A$ will be non-commutative. The rest
of this section will be devoted to three rather different ways to produce
examples of this type of algebra.

\subsection{Tate categories/ind-pro approach}

\begin{theorem}
[{\cite[Theorem 1]{bgwTateModule}}]\label{Thm_TateEndosAreCubicalAlgebras}Let
$\mathcal{C}$ be an idempotent complete and split exact category. For every
object $X\in\left.  n\text{-}\mathsf{Tate}_{\aleph_{0}}^{el}(\mathcal{C}%
)\right.  $, its endomorphism algebra carries the structure of a Beilinson
$n$-fold cubical algebra, we call it%
\[
E^{\operatorname*{Tate}}(X):=\operatorname*{End}\nolimits_{\left.
n\text{-}\mathsf{Tate}_{\aleph_{0}}^{el}(\mathcal{C})\right.  }\left(
X\right)  \text{.}%
\]

\end{theorem}

In particular, we can look at finite-dimensional $k$-vector spaces, i.e.
$\mathcal{C}:=\mathsf{Vect}_{f}$, and then the\ Tate objects \`{a} la
$k((t_{1}))\cdots((t_{n}))$ in \S \ref{subsect_KatoIndProApproach}
automatically carry a cubical endomorphism algebra.\newline See
\cite{bgwTateModule} for the construction of the algebra structure and for
further background. The above result is not given in the broadest possible
formulation, e.g. even if $\mathcal{C}$ is not split exact, the ideals
$I_{i}^{+},I_{i}^{-}$ can be defined. Moreover, they even make sense in
arbitrary $\operatorname*{Hom}$-groups and not just endomorphisms. Without
split exactness, one then has to be careful with the property $I_{1}^{+}%
+I_{1}^{-}=A$ however, which may fail in general.

The introduction of \cite{bgwTateModule} provides a reasonably short survey to
what extent the above theorem can be stretched, and which seemingly plausible
generalizations turn out to be problematic.

\subsection{Yekutieli's TLF approach}

Yekutieli also constructs such an algebra, but taking a topological local
field as its input.

\begin{theorem}
[A. Yekutieli]Let $k$ be a perfect field. Let $K$ be an $n$-dimensional TLF
over $k$. Then there is a canonically defined Beilinson $n$-fold cubical
$k$-algebra%
\[
E^{\operatorname*{Yek}}(K)\subseteq\operatorname*{End}\nolimits_{k}(K)\text{,}%
\]
contained in the algebra of all $k$-linear endomorphisms.
\end{theorem}

This is \cite[Theorem 0.4]{MR3317764}. We briefly summarize what lies behind
this: Firstly, Yekutieli introduces the notion of topological \emph{systems of
liftings} $\sigma$ for TLFs \cite[Def. 3.17]{MR3317764} (actually it is easy
to define: this is an algebraic system of liftings, as in our Definition
\ref{def_AlgSysOfLiftings}, where the sections $\sigma_{i}$ have to be
ST\ morphisms. We have already seen in Example \ref{example_Yekutieli} that
this truly cuts down the possible choices). Then he gives a very explicit
definition of a Beilinson $n$-fold cubical algebra called $E_{\sigma}^{K}$ in
loc. cit., depending on this choice of liftings. The precise definition is
\cite[Def. 4.5 and 4.14]{MR3317764}, and we refer the reader to this paper for
a less dense presentation and many more details:

\begin{definition}
[Yekutieli]\label{Def_YekBeilCubAlgAndYekLattices}Let $k$ be a perfect field
and $K$ an $n$-dimensional TLF over $k$.

\begin{enumerate}
\item If $M$ is a finite $K$-module, a \emph{Yekutieli lattice} $L$ is a
finite $\mathcal{O}_{1}$-submodule of $M$ such that $K\cdot L=M$.

\item Fix any system of liftings $\sigma=(\sigma_{1},\ldots,\sigma_{n})$ in
the sense of Yekutieli \cite[Def. 3.17]{MR3317764}. For finite $K$-modules
$M_{1},M_{2}$, define%
\[
E_{\sigma}^{\operatorname*{Yek}}(M_{1},M_{2})\subseteq\operatorname*{Hom}%
\nolimits_{k}(M_{1},M_{2})
\]
to be those $k$-linear maps such that

\begin{enumerate}
\item for $n=0$ there is no further restriction, all $k$-linear maps are allowed;

\item for $n\geq1$ and all Yekutieli lattices $L_{1}\subset M_{1},L_{2}\subset
M_{2}$, there have to exist Yekutieli lattices $L_{1}^{\prime}\subset
M_{1},L_{2}^{\prime}\subset M_{2}$ such that%
\[
L_{1}^{\prime}\subseteq L_{1},\qquad L_{2}\subseteq L_{2}^{\prime},\qquad
f(L_{1}^{\prime})\subseteq L_{2},\qquad f(L_{1})\subseteq L_{2}^{\prime}%
\]
and for all such choices $L_{1},L_{1}^{\prime},L_{2},L_{2}^{\prime}$ the
induced $k$-linear homomorphism%
\begin{equation}
\overline{f}:L_{1}/L_{1}^{\prime}\rightarrow L_{2}^{\prime}/L_{2}%
\tag{$\lozenge$}\label{l_eq_reductionYek}%
\end{equation}
must lie in $E_{(\sigma_{2},\ldots,\sigma_{n})}^{\operatorname*{Yek}}%
(L_{1}/L_{1}^{\prime},L_{2}^{\prime}/L_{2})$. For this read $L_{1}%
/L_{1}^{\prime}$ and $L_{2}^{\prime}/L_{2}$ as $k_{1}$-modules via the lifting
$\sigma_{1}:k_{1}\hookrightarrow\mathcal{O}_{1}$. Yekutieli calls any such
pair $(L_{1}^{\prime},L_{2}^{\prime})$ an $f$\emph{-refinement} of
$(L_{1},L_{2})$.
\end{enumerate}

\item Define $I_{1,\sigma}^{+}(M_{1},M_{2})$ to be those $f\in E_{\sigma
}^{\operatorname*{Yek}}(M_{1},M_{2})$ such that there exists a Yekutieli
lattice $L\subset M_{2}$ with $f(M_{1})\subseteq L$. Dually, $I_{1,\sigma}%
^{-}(M_{1},M_{2})$ is made of those such that there exists a lattice $L\subset
M_{1}$ with the property $f(L)=0$.

\item For $i=2,\ldots,n$, and both \textquotedblleft$+$/$-$\textquotedblright,
we let $I_{i,\sigma}^{\pm}(M_{1},M_{2})$ consist of those $f\in E_{\sigma
}^{\operatorname*{Yek}}(M_{1},M_{2})$ such that for all lattices $L_{1}%
,L_{1}^{\prime},L_{2},L_{2}^{\prime}$ as in part (2), Equation
\ref{l_eq_reductionYek}, the condition%
\[
\overline{f}\in I_{(i-1),(\sigma_{2},\ldots,\sigma_{n})}^{\pm}(L_{1}%
/L_{1}^{\prime},L_{2}^{\prime}/L_{2})
\]
holds.

\item For any finite $K$-module $M$, these ideals equip $(E_{\sigma
}^{\operatorname*{Yek}}(M,M),I_{i,\sigma}^{\pm}(M,M))$ with the structure of a
Beilinson $n$-fold cubical algebra. Yekutieli calls elements of $E_{\sigma
}^{\operatorname*{Yek}}$ a \emph{local Beilinson-Tate operator}.
\end{enumerate}
\end{definition}

The verification that this is indeed a cubical algebra is essentially
\cite[Lemma 4.17 and 4.19]{MR3317764}.

\begin{dangerousbend}
\label{YekDangerousBend}Something is \textsl{very important} to stress in this
context: The system of liftings plays an absolutely crucial role here. The
quotients%
\[
L_{1}/L_{1}^{\prime}\qquad\text{and}\qquad L_{2}^{\prime}/L_{2}%
\]
in Equation \ref{l_eq_reductionYek} carry a canonical structure as torsion
$\mathcal{O}_{1}$-modules. There is no canonical way to turn them into modules
over the residue field $k_{1}$; the residue map%
\[
\mathcal{O}_{1}\twoheadrightarrow k_{1}%
\]
goes in the wrong direction. So we really need a section to this map, i.e. a
system of liftings. As we have seen in the Example \ref{example_Yekutieli}
(due to Yekutieli), there can be very different sections, so a priori there is
a critical dependence of $E_{\sigma}^{\operatorname*{Yek}}$ on $\sigma$.
\end{dangerousbend}

The key technical input then becomes a rather surprising observation
originating from Yekutieli \cite{MR1213064}: Every change between Yekutieli's
systems of liftings must essentially come from a continuous differential
operator, see \cite[\S 2, especially Theorem 2.8 for $M_{1}=M_{2}$]{MR3317764}
for a precise statement, and these in turn lie in $E_{\sigma}^{K}$ regardless
of the $\sigma$. This establishes the independence of the system of liftings chosen.

\begin{theorem}
[Yekutieli \cite{MR3317764}]\label{Thm_YekIndependenceOfSystemOfLiftings}The
sub-algebra $E_{\sigma}^{\operatorname*{Yek}}(M_{1},M_{2})\subseteq
\operatorname*{Hom}\nolimits_{k}(M_{1},M_{2})$ is independent of the choice of
$\sigma$, and a choice of $\sigma$ always exists.
\end{theorem}

In order to distinguish his algebra, called \textquotedblleft$E^{K}%
$\textquotedblright\ in loc. cit., from the other variants appearing in this
paper, we shall call it $E^{\operatorname*{Yek}}$ in this paper. By the above
theorem, a reference to $\sigma$ is no longer needed at all.

\begin{remark}
If one looks at the $n$-dimensional TLF $K:=k((t_{1}))\cdots((t_{n}))$ over
$k$, then a precursor of Yekutieli's algebra is Osipov's algebra
\textquotedblleft$\operatorname*{End}\nolimits_{K}$\textquotedblright\ of his
2007 paper \cite[\S 2.3]{MR2314612}. As an associative algebra, it agrees with
$E_{\sigma}^{\operatorname*{Yek}}(K,K)$ and $\sigma$ the standard lifting.
However, Osipov's definition really uses the concrete presentation of $K$ as
Laurent series, so (a priori) it does not suffice to know $K$ as a plain TLF
or $n$-local field.
\end{remark}

\subsection{Beilinson's global approach}

Now suppose $X/k$ is a reduced scheme of finite type and pure dimension $n$.
We use the notation of \S \ref{subsect_AdeleSheaves}.

\begin{definition}
\label{def_Mis1}Let $\triangle=\{(\eta_{0}>\cdots>\eta_{i})\}\subseteq
S\left(  X\right)  _{i}$ (for some $i$) be a singleton set with
$\operatorname*{codim}\nolimits_{X}\overline{\{\eta_{r}\}}=r$.

\begin{enumerate}
\item Define $\triangle^{\prime}:=\{(\eta_{1}>\cdots>\eta_{n})\}\subseteq
S\left(  X\right)  _{i-1}$, removing the initial entry.

\item Write $\mathcal{F}_{\triangle}:=A(\triangle,\mathcal{F})$ for
$\mathcal{F}$ a quasi-coherent sheaf on $X$.
\end{enumerate}
\end{definition}

The notation $M_{\triangle}$ also makes sense if $M$ is an $\mathcal{O}%
_{\eta_{0}}$-module since any such defines a quasi-coherent sheaf.

\begin{definition}
[Beilinson \cite{MR565095}]\label{Definition_AdeleOperatorIdeals}Suppose
$\triangle=\{(\eta_{0}>\cdots>\eta_{i})\}$ is given.

\begin{enumerate}
\item If $M$ is a finitely generated $\mathcal{O}_{\eta_{0}}$-module, a
\emph{Beilinson lattice} in $M$ is a finitely generated $\mathcal{O}_{\eta
_{1}}$-module $L\subseteq M$ such that $\mathcal{O}_{\eta_{0}}\cdot L=M$.

\item Let $M_{1}$ and $M_{2}$ both be finitely generated $\mathcal{O}%
_{\eta_{0}}$-modules. Define $\operatorname*{Hom}\nolimits_{\varnothing}%
(M_{1},M_{2}):=\operatorname*{Hom}\nolimits_{k}(M_{1},M_{2})$ as all
$k$-linear maps. Define $\operatorname*{Hom}\nolimits_{\triangle}(M_{1}%
,M_{2})$ to be the $k$-submodule of all those maps $f\in\operatorname*{Hom}%
\nolimits_{k}(M_{1\triangle},M_{2\triangle})$ such that for all Beilinson
lattices $L_{1}\subset M_{1},L_{2}\subset M_{2}$ there exist lattices
$L_{1}^{\prime}\subset M_{1},L_{2}^{\prime}\subset M_{2}$ such that%
\[
L_{1}^{\prime}\subseteq L_{1},\qquad L_{2}\subseteq L_{2}^{\prime},\qquad
f(L_{1\triangle^{\prime}}^{\prime})\subseteq L_{2\triangle^{\prime}},\qquad
f(L_{1\triangle^{\prime}})\subseteq L_{2\triangle^{\prime}}^{\prime}%
\]
and for all such choices $L_{1},L_{1}^{\prime},L_{2},L_{2}^{\prime}$ the
induced $k$-linear homomorphism%
\[
\overline{f}:(L_{1}/L_{1}^{\prime})_{\triangle^{\prime}}\rightarrow
(L_{2}^{\prime}/L_{2})_{\triangle^{\prime}}%
\]
lies in $\operatorname*{Hom}\nolimits_{\triangle^{\prime}}(L_{1}/L_{1}%
^{\prime},L_{2}^{\prime}/L_{2})$.

\item Define $I_{1\triangle}^{+}(M_{1},M_{2})$ to be those $f\in
\operatorname*{Hom}\nolimits_{\triangle}(M_{1},M_{2})$ such that there exists
a lattice $L\subset M_{2}$ with $f(M_{1\triangle})\subseteq L_{\triangle
^{\prime}}$. Dually, $I_{1\triangle}^{-}(M_{1},M_{2})$ is made of those such
that there exists a lattice $L\subset M_{1}$ with the property $f(L_{\triangle
^{\prime}})=0$.

\item For $i=2,\ldots,n$, and both \textquotedblleft$+$/$-$\textquotedblright,
we let $I_{i\triangle}^{\pm}(M_{1},M_{2})$ consist of those $f\in
\operatorname*{Hom}\nolimits_{\triangle}(M_{1},M_{2})$ such that for all
lattices $L_{1},L_{1}^{\prime},L_{2},L_{2}^{\prime}$ as in part (3) the
condition%
\[
\overline{f}\in I_{(i-1)\triangle^{\prime}}^{\pm}(L_{1}/L_{1}^{\prime}%
,L_{2}^{\prime}/L_{2})
\]
holds.
\end{enumerate}
\end{definition}

With these definitions in place we are ready to formulate another principal
source of algebras as in Definition
\ref{def_BeilinsonDefinitionNFoldCubicalAlgebra}:

\begin{theorem}
[{Beilinson, \cite[\S 3]{MR565095}}]%
\label{Thm_BeilFlagInSchemeGivesCubDecompAlgebra}Suppose $X/k$ is a reduced
finite type scheme of pure dimension $n$. Let $\eta_{0}>\cdots>\eta_{n}\in
S\left(  X\right)  _{n}$ be a flag with $\operatorname*{codim}\nolimits_{X}%
\overline{\{\eta_{i}\}}=i$. Then%
\[
E_{\triangle}^{\operatorname*{Beil}}:=\operatorname*{Hom}\nolimits_{\triangle
}(\mathcal{O}_{\eta_{0}},\mathcal{O}_{\eta_{0}})
\]
is an associative sub-algebra of all $k$-linear maps from $\mathcal{O}%
_{X\triangle}$ to itself. For $i=1,2,\ldots,n$, define $I_{i\triangle}^{\pm
}\subseteq E_{\triangle}^{\operatorname*{Beil}}$ by $I_{i\triangle}^{\pm
}(\mathcal{O}_{\eta_{0}},\mathcal{O}_{\eta_{0}})$. Then $(E_{\triangle
}^{\operatorname*{Beil}},(I_{i\triangle}^{\pm}))$ is a Beilinson $n$-fold
cubical algebra. We shall call its elements \emph{global Beilinson-Tate
operators}.
\end{theorem}

The structure of this definition is very close to the variant of Yekutieli.
However, some essential ingredients differ significantly: On the one hand, no
system of liftings is used, so there is no counterpart of the Dangerous Bend
in \S \ref{YekDangerousBend} and no need for a result like Yekutieli's Theorem
\ref{Thm_YekIndependenceOfSystemOfLiftings}. On the other hand, we pay the
price of using the $r$-dimensional local rings $\mathcal{O}_{\eta_{r}}$ of
$X$. Thus, we really use some data of the scheme $X$ which a stand-alone
TLF\ cannot provide.

\section{\label{section_StandAloneHLFs}Stand-alone higher local fields}

Let $k$ be a perfect field and $K$ an $n$-dimensional TLF over $k$. Then for
finite $K$-modules $V_{1},V_{2}$ we have Yekutieli's cubical algebra,
Definition \ref{Def_YekBeilCubAlgAndYekLattices}, $E^{\operatorname*{Yek}%
}(V_{1},V_{2})$. However, we could try to interpret $K$ as an $n$-Tate object
in finite-dimensional $k$-vector spaces (in some way still to discuss) so that
we also have the corresponding cubical algebra as $n$-Tate objects, Theorem
\ref{Thm_TateEndosAreCubicalAlgebras}. We will establish a comparison
result.\medskip

There will be two variations: (1) We consider the multiple Laurent series
field%
\[
K=k((t_{1}))((t_{2}))\cdots((t_{n}))\text{.}%
\]
This is canonically a TLF, Example
\ref{example_YekutieliLaurentSTRingStructure}, and simultaneously canonically
an $n$-Tate object, Example \ref{example_KatoIndProOfLaurentSeries}. In this
case both cubical algebras are defined and we shall show that they are
canonically isomorphic.

(2) We shall consider a general TLF. In this case one has to choose a
presentation as an $n$-Tate object. This makes the comparison a little more
involved, but thanks to the results of Yekutieli's paper \cite{MR3317764}, one
still arrives at an isomorphism.

\subsection{Variant 1: Multiple Laurent series fields}

Let $k$ be a field. Recall the following:

\begin{enumerate}
\item $k[[t]]$ is a principal ideal domain,

\item every non-zero ideal is of the form $(t^{n})$ for $n\geq0$,

\item every finitely generated module is (non-canonically) of the form
\[
k[[t]]^{\oplus r_{0}}\oplus\bigoplus_{i=1}^{m}k[[t]]/t^{n_{i}}\text{,}%
\]

\item the forgetful functor $\mathsf{Mod}_{f}(k[[t]])\rightarrow
\mathsf{Vect}(k)$ is exact and canonically factors through an exact functor
\[
\mathsf{Mod}_{f}(k[[t]])\rightarrow\mathsf{Pro}_{\aleph_{0}}^{a}(k)\text{,}%
\]

\item \label{itemToep}the forgetful functor $\mathsf{Vect}_{f}%
(k((t)))\rightarrow\mathsf{Vect}(k)$ is exact and canonically factors through
an exact functor
\[
T:\mathsf{Vect}_{f}(k((t)))\rightarrow\mathsf{Tate}_{\aleph_{0}}%
^{el}(k)\text{.}%
\]

\end{enumerate}

Define $K:=k((t_{1}))\cdots((t_{n}))$.

\begin{lemma}
\label{lemma_TLFforget}The forgetful functor%
\[
\mathsf{Vect}_{f}(K)\rightarrow\mathsf{Vect}(k)
\]
is exact and factors through an exact functor
\[
T:\mathsf{Vect}_{f}(K)\rightarrow n\text{-}\mathsf{Tate}_{\aleph_{0}}^{el}(k).
\]

\end{lemma}

\begin{proof}
This follows from property \ref{itemToep}, and induction on $n$.
\end{proof}

We abbreviate $V_{k}(n):=k((t_{1}))\cdots((t_{n}))$, and regard this
simultaneously as a TLF as well as an $n$-Tate object with the structure
provided in Example \ref{example_KatoIndProOfLaurentSeries}. Similarly, write
$t_{n}^{i}k((t_{1}))\cdots((t_{n-1}))[[t_{n}]]$ for the standard Yekutieli
lattices in it, regarding both as a Yekutieli lattice as well as the
Pro-object in $(n-1)$-Tate objects defined by it. Recall from Definition
\ref{Def_YekBeilCubAlgAndYekLattices} that a Yekutieli lattice in $V_{k}(n)$
is a finitely generated $k((t_{1}))\cdots((t_{n-1}))[[t_{n}]]$-submodule
$L\subset V_{k}(n)$ such that $k((t_{1}))\cdots((t_{n}))\cdot L=V_{k}(n)$.

\begin{lemma}
\label{lemma:Oclat}Every Yekutieli lattice of $V_{k}(n)$ is of the form
$t_{n}^{i}k((t_{1}))\cdots((t_{n-1}))[[t_{n}]]$. In particular, it is a free
$k((t_{1}))\cdots((t_{n-1}))[[t_{n}]]$-module of rank 1.
\end{lemma}

\begin{proof}
It suffices to assume $n=1$. For the general case, just replace the field $k$
by the field $k((t_{1}))\cdots((t_{n-1}))$ and replace the $k$-algebra
$k[[t]]$, by the $k((t_{1}))\cdots((t_{n-1}))$-algebra $k((t_{1}%
))\cdots((t_{n-1}))[[t_{n}]]$. Now, let $M\subset k((t))$ be a finitely
generated $k[[t]]$-sub-module such that $k((t))\cdot M=k((t))$. Let
$\{f_{1},\cdots,f_{m}\}$ be a set of generators for $M$ over $k[[t]]$.
Re-ordering as necessary, we can assume that ${\mathrm{ord}}_{t=0}f_{i}%
\leq{\mathrm{ord}}_{t=0}f_{i+1}$ for all $i$. Define $\ell:={\mathrm{ord}%
}_{t=0}f_{1}$. By definition, we have $M\subset t^{\ell}k[[t]]\subset k((t))$.
Conversely, because $k$ is a field, there exists a unit in $g\in
k[[t]]^{\times}$ such that $f_{1}g=t^{\ell}$. Because $t^{\ell}k[[t]]$ is a
cyclic $k[[t]]$-module generated by $t^{\ell}$, we conclude that $M\supset
t^{\ell}k[[t]]$ as well.
\end{proof}

\begin{lemma}
\label{lemma:Ocfin}Denote by $Gr^{\operatorname*{Yek}}(K)$ the partially
ordered set of Yekutieli lattices. There is a final and co-final inclusion of
partially ordered sets $Gr^{\operatorname*{Yek}}(K)\subset Gr(V_{k}(n))$,
where the latter denotes the Grassmannian of Tate lattices (i.e. the Sato
Grassmannian as defined in \cite{TateObjectsExactCats}).
\end{lemma}

\begin{proof}
The $n$-Tate object $V_{k}(n)$ is represented by the admissible Ind-diagram
\[
\cdots\hookrightarrow t_{n}^{i}V_{k}(n-1)[[t_{n}]]\hookrightarrow t_{n}%
^{i-1}V_{k}(n-1)[[t_{n}]]\hookrightarrow\cdots.
\]
We see that every Yekutieli lattice arises in this diagram. Therefore every
Yekutieli lattice is a Tate lattice of $V_{k}(n)$, i.e.
$Gr^{\operatorname*{Yek}}(K)\subset Gr(V_{k}(n))$. Further, by the definition
of Hom-sets in $n\text{-}\mathsf{Tate}_{\aleph_{0}}^{el}(k)$ (which implies
that the sub-category $\mathsf{Pro}^{a}((n-1)\text{-}\mathsf{Tate}_{\aleph
_{0}}^{el}(k))$ is left filtering), we see that every Tate lattice in
$V_{k}(n)$ factors through a Yekutieli lattice in the above diagram. Therefore
the sub-poset of Yekutieli lattices is final. It remains to show that every
Tate lattice $L$ of $V_{k}(n)$ contains a Yekutieli lattice. This will follow
from the same argument by which one shows that $\mathsf{Ind}^{a}(\mathcal{C})$
is right filtering in $\mathsf{Tate}^{el}(\mathcal{C})$ (cf. \cite[Proposition
5.10]{TateObjectsExactCats}). Denote by $\mathcal{O}_{1}(0)$ the Yekutieli
lattice $V_{k}(n-1)[[t_{n}]]\subset V_{k}(n)$. Consider the map
\[
\mathcal{O}_{1}(0)\hookrightarrow V_{k}(n)\twoheadrightarrow V_{k}%
(n)/L\text{.}%
\]
Because $\mathsf{Pro}^{a}((n-1)\text{-}\mathsf{Tate}_{\aleph_{0}}^{el}(k))$ is
left filtering in $\mathsf{n}\text{-}\mathsf{Tate}_{\aleph_{0}}^{el}(k)$,
there exists an $(n-1)$-Tate object $P$ such that the above map factors as
\[
\mathcal{O}_{1}(0)\rightarrow P\hookrightarrow V_{k}(n)/L.
\]
Further, $\mathcal{O}_{1}(0)$ is represented by the admissible Pro-diagram
\[
\cdots\mathcal{O}_{1}(0)/t_{n}^{i}\twoheadrightarrow\mathcal{O}_{1}%
(0)/t_{n}^{i-1}\twoheadrightarrow\cdots\twoheadrightarrow V_{k}(n-1)\text{.}%
\]
Therefore, by the definition of Hom-sets in $\mathsf{Pro}^{a}((n-1)\text{-}%
\mathsf{Tate}_{\aleph_{0}}^{el}(k))$ (which implies that the sub-category
$(n-1)\text{-}\mathsf{Tate}_{\aleph_{0}}^{el}(k)$ is right filtering), we see
that there exists $i$ such that the map $\mathcal{O}_{1}(0)\rightarrow P$
factors as
\[
\mathcal{O}_{1}(0)\twoheadrightarrow\mathcal{O}_{1}(0)/t_{n}^{i}\rightarrow P.
\]
By the universal property of kernels, we conclude that the Yekutieli lattice
$t_{n}^{i}\mathcal{O}_{1}(0)$ is a common Tate sub-lattice of $\mathcal{O}%
_{1}(0)$ and $L$.
\end{proof}

\begin{lemma}
\label{Lemma_LaurentYekIsTate1}For any $V_{1},V_{2}\in\mathsf{Vect}_{f}(K)$,
there is an equality of subsets of $\operatorname*{Hom}{}_{k}(V_{1},V_{2})$
\[
E^{\operatorname*{Yek}}(V_{1},V_{2})=\operatorname*{Hom}{}_{n\text{-}%
\mathsf{Tate}_{\aleph_{0}}^{el}(k)}(T(V_{1}),T(V_{2}))\text{,}%
\]
where $T$ denotes the functor of Lemma \ref{lemma_TLFforget}.
\end{lemma}

\begin{remark}
A key fact used in the statement and proof of this theorem is that the
forgetful functor $\left.  n\text{-}\mathsf{Tate}_{\aleph_{0}}^{el}\right.
(k)\rightarrow\mathsf{Vect}(k)$ is injective on Hom-sets. This is immediate
for $n=1$, and for $n>1$, it follows by induction.
\end{remark}

\begin{proof}
We prove this by induction on $n$. For $n=0$, there is nothing to show. For
the induction step, by the universal properties of direct sums, it suffices to
show the equality for $V:=V_{1}=V_{2}=k((t_{1}))\cdots((t_{n}))$.\medskip

\textit{Proof of sub-claim: }The compatibility of $E_{\sigma}%
^{\operatorname*{Yek}}(-,-)$ with direct sums in both variables is a
straightforward induction on $n$: for $n=0$, this is immediate (since we are
just considering homomorphisms of finite-dimensional vector spaces). For the
induction step, we first observe that the definition of Yekutieli lattices
implies that every lattice $L\subset V_{1}\oplus V_{2}$ is of the form
$L_{1}\oplus L_{2}$, where $L_{i}\subset V_{i}$ is a Yekutieli (the splitting
on $L$ is induced by the splitting on $V$). This, plus the induction
hypothesis, shows that%
\[
E_{\sigma}^{\operatorname*{Yek}}(W,V_{1}\oplus V_{2})\subset E_{\sigma
}^{\operatorname*{Yek}}(W,V_{1})\times E_{\sigma}^{\operatorname*{Yek}%
}(W,V_{2})
\]
and vice versa, and similarly with $W$ and $V_{1}\oplus V_{2}$ interchanged).
This finishes the proof of the sub-claim.\medskip

Note that for these $V$, $T(V):=V_{k}(n)$. We begin by showing that
$\operatorname*{End}\nolimits_{n\text{-}\mathsf{Tate}_{\aleph_{0}}^{el}%
(k)}(V_{k}(n))\subset E^{\operatorname*{Yek}}(V)$. Let $\varphi$ be an
endomorphism of $V_{k}(n)$. Let $L_{1}=t_{n}^{i_{1}}V_{k}(n-1)[[t_{n}]]$ and
$L_{2}=t_{n}^{i_{2}}V_{k}(n-1)[[t_{n}]]$ be a pair of Yekutieli lattices of
$k((t_{1}))\cdots((t_{n}))$. We begin by showing that this pair admits a
$\varphi$-refinement (see Definition \ref{Def_YekBeilCubAlgAndYekLattices}).
By the standard Ind-diagram for $V_{k}(n)$, and the definition of Hom-sets in
$n$-$\mathsf{Tate}_{\aleph_{0}}^{el}(k)$, there exists a Yekutieli lattice
$N=t_{n}^{j}V_{k}(n-1)[[t_{n}]]$ such that $\varphi(L_{1})\subset N$. Let
$i_{2}^{\prime}=\min(j,i_{2})$, and define $L_{2}^{\prime};=t_{n}%
^{i_{2}^{\prime}}V_{k}(n-1)[[t_{n}]]$. Next, consider the map $L_{1}%
\overset{\varphi}{\rightarrow}L_{2}^{\prime}/L_{2}$. The quotient
$L_{2}^{\prime}/L_{2}\cong V_{k}(n-1)[[t_{n}]]/(t_{n}^{i_{2}-i_{2}^{\prime}})$
is an elementary $(n-1)$-Tate space. By the definition of Hom-sets in
$\mathsf{Pro}^{a}((n-1)\text{-}\mathsf{Tate}_{\aleph_{0}}^{el}(k))$ (which
implies that the sub-category of $(n-1)$-Tate spaces is right filtering), the
map above factors through an admissible epic in $\mathsf{Pro}^{a}%
((n-1)\text{-}\mathsf{Tate}_{\aleph_{0}}^{el}(k))$
\[
L_{1}\twoheadrightarrow L_{1}/t_{n}^{\ell}L_{1}\overset{\overline{\varphi}%
}{\rightarrow}L_{2}^{\prime}/L_{2}\text{.}%
\]
We define $L_{1}^{\prime}=t_{n}^{i_{1}+\ell}V_{k}(n-1)[[t_{n}]]$, and observe
that $(L_{1}^{\prime},L_{2}^{\prime})$ $\varphi$-refines $(L_{1},L_{2})$.
Furthermore, because $(n-1)\text{-}\mathsf{Tate}_{\aleph_{0}}^{el}(k)$ is a
full sub-category of $\mathsf{Pro}^{a}((n-1)\text{-}\mathsf{Tate}_{\aleph_{0}%
}^{el}(k))$, the map $\overline{\varphi}$ is a map of $(n-1)$-Tate spaces. By
the inductive hypothesis, this map is an element in $E^{\operatorname*{Yek}%
}(L_{1}/t_{n}^{\ell}L_{1},L_{2}^{\prime}/L_{2})$. We conclude that
\[
\operatorname*{End}{}_{n\text{-}\mathsf{Tate}_{\aleph_{0}}^{el}(k)}%
(V_{k}(n))\subset E^{\operatorname*{Yek}}(k((t_{1}))\cdots((t_{n}))).
\]
To complete the induction step, it remains to show the reverse inclusion. Let
$\varphi\in E^{\operatorname*{Yek}}(K)$. We begin by showing that, given any
two Yekutieli lattices $L_{1}$ and $L_{2}$ such that $\varphi(L_{1})\subset
L_{2}$, then the map $L_{1}\overset{\varphi}{\rightarrow}L_{2}$ is a map of
admissible Pro-objects (in $(n-1)$-Tate spaces). By Lemma \ref{lemma:Oclat},
$L_{a}\cong t_{n}^{i_{a}}V_{k}(n-1)[[t_{n}]]$ for $a=1,2$. By the definition
of Yekutieli's $E^{\operatorname*{Yek}}$, Definition
\ref{Def_YekBeilCubAlgAndYekLattices}, for each $\ell>0$, there exists a
$\varphi$-refinement $(L_{1}^{\ell},L_{2}^{\ell})$ of the pair $(L_{1}%
,t_{n}^{\ell}L_{2})$. Without loss of generality, we can take $L_{2}^{\ell
}=L_{2}$, and we therefore obtain a square
\[%
\begin{xy} \square[L_1^\ell`t_n^\ell L_2`L_1`L_2;\varphi```\varphi] \end{xy}%
\]
By the definition of local $BT$-operators, for all $\ell\geq0$, the induced
map $L_{1}/L_{1}^{\ell}\rightarrow L_{2}/t_{n}^{\ell}L_{2}$ is a local
$BT$-operator, and thus, by induction hypothesis, a map of $(n-1)$-Tate
spaces. Because an inclusion of Yekutieli lattices is an admissible monic of
admissible Pro-objects (e.g. by Lemma \ref{lemma:Oclat}), for all $\ell\geq0$,
the map
\[
L_{1}\twoheadrightarrow L_{1}/L_{1}^{\ell}\rightarrow L_{2}/t_{n}^{\ell}L_{2}%
\]
is a map of admissible Pro-objects. Taking the limit over all $\ell$ (in
$\mathsf{Pro}^{a}((n-1)\text{-}\mathsf{Tate}_{\aleph_{0}}^{el}(k))$), we
obtain a map of admissible Pro-objects $L_{1}\rightarrow\lim_{\ell}L_{2}%
/t_{n}^{\ell}L_{2}\cong L_{2}$. The forgetful functor
\[
\mathsf{Pro}^{a}((n-1)\text{-}\mathsf{Tate}_{\aleph_{0}}^{el}(k))\rightarrow
\mathsf{Vect}(k)
\]
preserves limits (by construction, see Remark
\ref{remark_LimitPreservationInProof}). Therefore, we conclude that the map of
$k$-vector spaces underlying the map of admissible Pro-objects is equal to the
limit of the maps
\[
L_{1}\twoheadrightarrow L_{1}/L_{1}^{\ell}\rightarrow L_{2}/t_{n}^{\ell}L_{2}%
\]
but this is just $\varphi$. We have shown that $\varphi$ restricts to a map of
admissible Pro-objects on any pair of lattices $L_{1}$ and $L_{2}$ such that
$\varphi(L_{1})\subset L_{2}$. It remains to show that $\varphi$ is a map of
$n$-Tate spaces. Let $L_{\ell}=t_{n}^{\ell}V_{k}(n-1)[[t_{n}]]$. Then
$\ell\mapsto L_{\ell}$ is an admissible Ind-diagram (in $\mathsf{Pro}%
^{a}((n-1)\text{-}\mathsf{Tate}_{\aleph_{0}}^{el}(k))$) representing
$V_{k}(n)$. By inducting on $\ell$, we now construct a second admissible
Ind-diagram $\ell\mapsto L_{\ell}^{\prime}$ representing $V_{k}(n)$ such that
$\varphi$ lifts to a map of these diagrams. For the base case, by the
definition of local $BT$-operators, there exists a pair of Yekutieli lattices
$(L_{-1},L_{0}^{\prime})$ which $\varphi$-refine $(L_{0},L_{0})$. In
particular, $\varphi(L_{0})\subset L_{0}^{\prime}$. For the induction step,
suppose we have constructed an ascending chain of inclusions of Yekutieli
lattices
\[
L_{0}^{\prime}\hookrightarrow\cdots\hookrightarrow L_{n}^{\prime}%
\]
such that $\varphi(L_{i}),L_{i}\subset L_{i}^{\prime}$ for $i\leq n$. Consider
the pair of Yekutieli lattices $(L_{n+1},L_{n}^{\prime})$. Then there exists a
pair of Yekutieli lattices $(L_{a},L_{b})$ which $\varphi$-refines this pair.
Further (e.g. by Lemma \ref{lemma:Oclat}), there exists a Yekutieli lattice
$L_{n+1}^{\prime}$ which contains both $L_{b}$ and $L_{n+1}$. This completes
the induction step. Above we have shown that the maps $L_{\ell}\overset
{\varphi}{\rightarrow}L_{\ell}^{\prime}$ are maps of admissible Pro-objects
(in $(n-1)$-Tate spaces) for each $\ell$. Therefore, we conclude that
$\varphi$ lifts to a map of admissible Ind-diagrams. By construction, the
ascending chain of lattices
\[
L_{0}^{\prime}\hookrightarrow\cdots\hookrightarrow L_{\ell}^{\prime
}\hookrightarrow\cdots
\]
is final in the Grassmannian of Tate lattices $Gr(V_{k}(n))$ (because the
chain $L_{0}\hookrightarrow\cdots\hookrightarrow L_{\ell}\hookrightarrow
\cdots$ is). We conclude that $V_{k}(n)$ is the colimit of this ascending
chain, and that the map of colimits
\[
V_{k}(n)\cong\underset{\ell}{\underrightarrow{\operatorname*{colim}}}%
\,L_{\ell}\rightarrow\underset{\ell}{\underrightarrow{\operatorname*{colim}}%
}\,L_{\ell}^{\prime}\cong V_{k}(n)
\]
is a map of $n$-Tate spaces. But, this map is equal to $\varphi$ (e.g. because
the forgetful map $n$-$\mathsf{Tate}_{\aleph_{0}}^{el}(k)\rightarrow
\mathsf{Vect}(k)$ preserves colimits, by construction, cf. Remark
\ref{remark_LimitPreservationInProof}). We conclude that%
\[
E^{\operatorname*{Yek}}(k((t_{1}))\cdots((t_{n})))\subset\operatorname*{End}%
\nolimits_{n\text{-}\mathsf{Tate}_{\aleph_{0}}^{el}(k)}(V_{k}(n))\text{.}%
\]
This finishes the proof.
\end{proof}

\begin{remark}
\label{remark_LimitPreservationInProof}Let us provide some details on the
preservation of (co-)limits: Suppose $\mathcal{D}$ is a complete and
co-complete category. For any exact category $\mathcal{C}$, $\mathsf{Pro}%
^{a}\mathcal{C}$ is a full sub-category of the category of right-exact
co-sheaves on $\mathcal{C}$ \cite{TateObjectsExactCats}. As such, any functor
$\mathcal{C}\rightarrow\mathcal{D}$ extends uniquely to a limit-preserving
functor $\mathsf{Pro}^{a}\mathcal{C}\rightarrow\mathcal{D}$. We emphasize that
this limit preservation refers to the category of Pro-objects, i.e. it makes
no statements about limits taken inside of $\mathcal{C}$ (taking limits in
$\mathcal{C}$ or $\mathsf{Pro}^{a}\mathcal{C}$ usually yields different
outcomes). Similarly, any functor $\mathcal{C}\rightarrow\mathcal{D}$ extends
uniquely to a colimit-preserving functor $\mathsf{Ind}^{a}\mathcal{C}%
\rightarrow\mathcal{D}$. By the evaluation of limits and colimits, we have
functors%
\[
\left.  (n-1)\text{-}\mathsf{Tate}(k)\right.  \rightarrow\mathsf{Vect}%
(k)\text{,}%
\]
and these canonically induce limit-preserving functors%
\[
\mathsf{Pro}^{a}(\left.  (n-1)\text{-}\mathsf{Tate}(k)\right.  )\rightarrow
\mathsf{Vect}(k)
\]
and colimit-preserving functors%
\[
\left.  n\text{-}\mathsf{Tate}(k)\right.  \rightarrow\mathsf{Ind}%
^{a}\mathsf{Pro}^{a}(\left.  (n-1)\text{-}\mathsf{Tate}(k)\right.
)\rightarrow\mathsf{Vect}(k)\text{.}%
\]

\end{remark}

\begin{lemma}
\label{Lemma_LaurentYekIsTate2}For any $V_{1},V_{2}\in\mathsf{Vect}_{f}(K)$,
the equality
\[
E^{\operatorname*{Yek}}(V_{1},V_{2})=\operatorname*{Hom}\nolimits_{n\text{-}%
\mathsf{Tate}_{\aleph_{0}}^{el}(k)}(T(V_{1}),T(V_{2}))
\]
of Lemma \ref{Lemma_LaurentYekIsTate1} restricts to an equality of two-sided
ideals
\[
I_{i,\operatorname*{Yek}}^{\pm}(V_{1},V_{2})=I_{i,\operatorname*{Tate}}^{\pm
}(T(V_{1}),T(V_{2}))\text{.}%
\]
for $1\leq i\leq n$.
\end{lemma}

\begin{proof}
We prove this by induction on $n$. For $n=0$, there is nothing to show.
Because every Yekutieli lattice of $V$ induces a Tate lattice of $V_{k}(n)$,
knowing any conditions defining $I_{i}^{\pm}$ for all Tate lattices, implies
it for all Yekutieli lattices. Thus, we immediately get%
\[
I_{i,\operatorname*{Yek}}^{\pm}(V_{1},V_{2})\supseteq
I_{i,\operatorname*{Tate}}^{\pm}(T(V_{1}),T(V_{2}))
\]
The converse direction is a bit more involved. Not every Tate lattice is a
Yekutieli lattice, but with the help of Lemma \ref{lemma:Ocfin}\ we shall
reduce checking conditions for Tate lattices to Yekutieli lattices. Suppose we
want to check whether $\varphi\in I_{i,\operatorname*{Tate}}^{\pm}%
(T(V_{1}),T(V_{2}))$ holds. For $i=1$, Lemma \ref{lemma:Ocfin} implies that
having image contained in a Yekutieli lattice is the same as having image
contained in a Tate lattice, and analogously for kernels. Thus, to deal with
$i=2,\ldots,n$ we only need to confirm that this argument survives
refinements: We know that if $L_{1}\subset T(V_{1})$, $L_{2}\subset T(V_{2})$
are Tate lattices and we pick Tate lattices $L_{1}^{\prime}\subset T(V_{1})$,
$L_{2}^{\prime}\subset T(V_{2})$ such that%
\[
L_{1}^{\prime}\subseteq L_{1},\qquad L_{2}\subseteq L_{2}^{\prime},\qquad
f(L_{1}^{\prime})\subseteq L_{2},\qquad f(L_{1})\subseteq L_{2}^{\prime
}\text{,}%
\]
we have the $f$-refinement%
\[
\overline{f}:L_{1}/L_{1}^{\prime}\rightarrow L_{2}^{\prime}/L_{2}\text{.}%
\]
We need to show that $\overline{f}\in I_{i-1,\operatorname*{Tate}}^{\pm}%
(L_{1}/L_{1}^{\prime},L_{2}^{\prime}/L_{2})$, just assuming this holds
whenever all of the above lattices are also Yekutieli lattices. So let
$L_{1},L_{2}$ be Tate lattices for which we want to check the defining
property. By Lemma \ref{lemma:Ocfin}, there exist Yekutieli lattices
$N_{2,a}\subset L_{2}$ and $N_{1,b}\supset L_{1}$. Also, by Lemma
\ref{lemma:Ocfin}, we can choose a $\varphi$-refinement $(N_{1,a},L_{2})$ of
$(L_{1}^{\prime},N_{2,a})$ with $N_{1,a}$ a Yekutieli lattice, and we can also
choose a $\varphi$-refinement $(L_{1},N_{2,b})$ of $(N_{1,b},L_{2}^{\prime})$
with $N_{2,b}$ a Yekutieli lattice. These refinements define a commuting
diagram
\[%
\xymatrix{ N_{1,b}/N_{1,a} \ar[r] & N_{2,b}/N_{2,a} \\ L_1/N_{2,a} \ar
[u] \ar[r] \ar[d] & L_2'/N_{2,a}\ar[u] \ar[d]\\ L_1/L_1' \ar[r]^{\overline
{\varphi}} & L_2'/L_2. }%
\]
By assumption, the top horizontal map is in $I_{i-1,\operatorname*{Yek}}^{\pm
}(N_{1,b}/N_{1,a},N_{2,b}/N_{2,a})$. Further, the upper vertical arrows are
admissible monics, while the lower vertical arrows are admissible epics. In
particular, all the vertical maps split, so we have a commuting diagram
\[%
\xymatrix{ N_{1,b}/N_{1,a} \ar[r] & N_{2,b}/N_{2,a} \ar[d] \\ L_1/L_1' \ar
[u] \ar[r] & L_2'/L_2. }%
\]
in which the top map is in $I_{i-1,\operatorname*{Yek}}^{\pm}(N_{1,b}%
/N_{1,a},N_{2,b}/N_{2,a})$. Because this is a categorical ideal \cite[Lemma
4.16 (2)]{MR3317764}, we conclude that the bottom map is in
$I_{i-1,\operatorname*{Tate}}^{\pm}(L_{1}/L_{1}^{\prime},L_{2}^{\prime}%
/L_{2})$ as claimed.
\end{proof}

Of course combining Lemma \ref{Lemma_LaurentYekIsTate1} with Lemma
\ref{Lemma_LaurentYekIsTate2} implies:

\begin{theorem}
\label{Theorem_LaurentYekIsTate}The functor%
\[
T:\mathsf{Vect}_{f}(K)\rightarrow\left.  n\text{-}\mathsf{Tate}_{\aleph_{0}%
}^{el}\right.  (k)
\]
induces canonical isomorphisms%
\[
E^{\operatorname*{Yek}}(V_{1},V_{2})\cong\operatorname*{Hom}\nolimits_{\left.
n\text{-}\mathsf{Tate}_{\aleph_{0}}^{el}(k)\right.  }(T(V_{1}),T(V_{2}))
\]
so that for $V_{1}=V_{2}$ this becomes an isomorphism of Beilinson cubical algebras.
\end{theorem}

This finishes the comparison.

\begin{example}
[Osipov, Yekutieli]\label{example_PreserveTateImpliesPreserveTopology}%
Yekutieli has shown that elements in $E^{\operatorname*{Yek}}(V_{1},V_{2})$
are morphisms of ST\ modules, i.e. they are continuous in the ST\ topology
\cite[Thm. 4.24]{MR3317764}. However, he also proved that
$E^{\operatorname*{Yek}}(V_{1},V_{2})$ is strictly smaller than the algebra of
all ST\ module homomorphisms for $n\geq2$ \cite[Example 4.12 and
following]{MR3317764}. This generalizes an observation due to Osipov, who had
established the corresponding statements for Laurent series with Parshin's
natural topology \cite[\S 2.3]{MR2314612}.
\end{example}

\subsection{Variant: TLFs}

Instead of working with an explicit model like $k((t_{1}))\cdots((t_{n}))$ we
can also work with a general TLF. Firstly, recall that this forces us to
assume that the base field $k$ is perfect. Even though we cannot associate an
$n$-Tate vector space over $k$ to a TLF directly, we can do so using
Yekutieli's concept of a system of liftings:

\begin{definition}
\label{def_TLFTateObject}Let $k$ be a perfect field. Moreover, let $K$ be an
$n$-dimensional TLF over $k$ and $\sigma=(\sigma_{1},\ldots,\sigma_{n})$ a
system of liftings in the sense of Yekutieli. Suppose $V$ is a
finite-dimensional $K$-vector space.

\begin{enumerate}
\item If $n=0$, $K=k$ and every finite-dimensional $k$-vector space is
literally a $0$-Tate object over $\mathsf{Vect}_{f}(k)$.

\item If $n\geq1$, the ring of integers $\mathcal{O}_{1}:=\mathcal{O}_{1}(K)$
is a (not finitely generated) $k_{1}(K)$-module. Let $b_{1},\ldots,b_{r}$ be
any $K$-basis of $V$ and $\mathcal{O}_{1}\otimes\left\{  b_{1},\ldots
,b_{r}\right\}  $ its $\mathcal{O}_{1}$-span inside $V$. We can partially
order all such bases by the inclusion relation among their $\mathcal{O}_{1}%
$-spans. Note that each
\[
\left(  \mathcal{O}_{1}\otimes\left\{  b_{1},\ldots,b_{r}\right\}  \right)
/\mathfrak{m}_{1}^{m}%
\]
is a finite torsion $\mathcal{O}_{1}$-module and thus a finite-dimensional
$k_{1}(K)$-vector space by the lifting $\sigma_{1}$.

\item Thus, if we assume that each finite-dimensional vector space $V$ over
the $(n-1)$-dimensional TLF $k_{1}(K)$ along with the system of liftings
$(\sigma_{2},\ldots,\sigma_{n})$ comes with a fixed model, denoted $V^{\sharp
}$, as an $(n-1)$-Tate object in $k$-vector spaces,%
\begin{equation}
\underset{b_{1},\ldots,b_{r}}{\underrightarrow{\operatorname*{colim}}%
}\underset{m}{\underleftarrow{\lim}}\left(  \left(  \mathcal{O}_{1}%
\otimes\left\{  b_{1},\ldots,b_{r}\right\}  \right)  /\mathfrak{m}_{1}%
^{m}\right)  ^{\sharp}\label{lcca1}%
\end{equation}
defines an $n$-Tate object in $k$-vector spaces.

\item Inductively, this associates a canonical $n$-Tate object to each
finite-dimensional $K$-vector space (but depending on the chosen system of liftings).
\end{enumerate}
\end{definition}

It is easy to check that the colimit over the bases $b_{1},\ldots,b_{r}$ is
filtering.\medskip

The technical result as well as the key idea underlying the proof of the
following is entirely due to Yekutieli:

\begin{theorem}
\label{Theorem_TLFPlusSysOfLiftingsYekIsTate}Let $k$ be a perfect field and
$K$ an $n$-dimensional TLF over $k$.

\begin{enumerate}
\item For any system of liftings $\sigma$, the construction in Definition
\ref{def_TLFTateObject} gives rise to a functor \textquotedblleft%
$\sharp_{\sigma}$\textquotedblright%
\[
\mathsf{Vect}_{f}(K)\overset{\sharp_{\sigma}}{\longrightarrow}\left.
n\text{-}\mathsf{Tate}^{el}(\mathsf{Vect}_{f}(k))\right.  \overset
{\operatorname{eval}}{\longrightarrow}\mathsf{Vect}(k)
\]
so that the composition agrees with the forgetful functor to $k$-vector spaces
as in Lemma \ref{lemma_TLFforget}.

\item For any $V_{1},V_{2}\in\mathsf{Vect}_{f}(K)$, the functor $\sharp
_{\sigma}$ induces an isomorphism%
\[
E^{\operatorname*{Yek}}(V_{1},V_{2})\overset{\sim}{\longrightarrow
}\operatorname*{Hom}\nolimits_{n\text{-}\mathsf{Tate}^{el}}(\sharp_{\sigma
}V_{1},\sharp_{\sigma}V_{2})\text{.}%
\]

\item For any two systems of liftings $\sigma,\sigma^{\prime}$, there exists
an $n$-Tate automorphism $e_{\sigma,\sigma^{\prime}}$ such that $\sharp
_{\sigma^{\prime}}=e_{\sigma,\sigma^{\prime}}\circ\sharp_{\sigma}$.

\item For any $V_{1},V_{2}$, the image of $\operatorname*{Hom}%
\nolimits_{n\text{-}\mathsf{Tate}^{el}}(\sharp_{\sigma}V_{1},\sharp_{\sigma
}V_{2})$ under `$\operatorname{eval}$' is independent of the choice of
$\sigma$, and agrees with $E^{\operatorname*{Yek}}(V_{1},V_{2})$.
\end{enumerate}
\end{theorem}

The interesting aspect of (3) is the existence of a \textit{canonical}
isomorphism. The existence of an abundance of rather random isomorphisms is
clear from the outset.

\begin{proof}
(1) and (2):\ The proof is basically a repetition of everything we have done
with $k((t_{1}))\cdots((t_{n}))$ in this section. The argument works basically
verbatim. Replace each $k((t_{1}))\cdots((t_{n}))$ by $K$, each $(-)[[t]]$ by
the respective ring of integers $\mathcal{O}$, and each power $t^{i}$ by
$\mathfrak{m}^{i}$ with $\mathfrak{m}$ the respective maximal ideal. The only
slight change is that in Equation \ref{lcca1} we take the colimit over all
bases $b_{1},\ldots,b_{r}$ in Lemma \ref{lemma:Ocfin}. Part (3) is deep in
principle, but easy for us since we can rely on the theory set up in
\cite{MR3317764}. In Definition \ref{def_TLFTateObject}, part (2), we can read
the finite $\mathcal{O}_{1}$-module%
\[
\left(  \mathcal{O}_{1}\otimes\left\{  b_{1},\ldots,b_{r}\right\}  \right)
/\mathfrak{m}_{1}^{m}%
\]
as a $k_{1}(K)$-vector space either by the lifting $\sigma$ or $\sigma
^{\prime}$. The assumptions of \cite[Theorem 2.8, (2)]{MR3317764} are
satisfied; the above is a finite $\mathcal{O}_{1}$-module and it is a precise
Artinian local ring by \cite[Lemma 3.14]{MR3317764}. By Yekutieli's theorem,
loc. cit., the identity automorphism on the module transforms the two
$k_{1}(K)$-vector space structures of $\sigma$ and $\sigma^{\prime}$ via
$\operatorname*{GL}\nolimits_{(-)}(\mathcal{D}_{K/k}^{\operatorname*{cont}})$
and by \cite[Lemma 4.11]{MR3317764} this lies in Yekutieli's $E^{K}=E_{\sigma
}^{K}$, i.e. our $E^{\operatorname*{Yek}}(K)$ (at this point in Yekutieli's
paper it has not yet been proven that this is independent of $\sigma$, but of
course we may already use this here). Finally, by part (2) this is nothing but
an automorphism as an $n$-Tate object, giving the desired $e_{\sigma
,\sigma^{\prime}}$. Part (4) follows from (3): The images just differ by an
inner automorphism, but that means that they are the same.
\end{proof}

\section{\label{sect_StructTheorems}Structure theorems}

\subsection{Structure of the ad\`{e}les}

In order to proceed, we shall need a few structural results about the
structure of the local ad\`{e}les. The following result

\begin{itemize}
\item is classical (and nearly trivial) in dimension one,

\item is due to Parshin in dimension two \cite{MR0419458},

\item is due to Beilinson in general \cite{MR565095}, but the proof remained unpublished,

\item and the first proof in print is due to Yekutieli \cite[\S 3,
3.3.2-3.3.6]{MR1213064}.
\end{itemize}

We shall give a self-contained proof in this paper $-$ needless to say,
following similar ideas than those used by Yekutieli $-$ but a number of steps
are done a bit differently and we strengthen parts of the results, especially
in view of Kato's ind-pro perspective (\S \ref{subsect_KatoIndProApproach}).

The following section relies on a number of standard facts from commutative
algebra. For the convenience of the reader, we will cite them from the
Appendix \S \ref{section_Appendix_ResultsFromComAlg}, where we have collected
the relevant material.\newline

\begin{definition}
A \emph{saturated flag} $\triangle$ in $X$ is a singleton set $\triangle
=\{(\eta_{0}>\cdots>\eta_{r})\}\subseteq S\left(  X\right)  _{r}$ such that
$\operatorname*{codim}_{X}\overline{\{\eta_{i}\}}=i$.
\end{definition}

Whenever we need to relate ad\`{e}les between different schemes, in order to
be sure what we mean, we write $A_{X}(-,\mathcal{-})$ to denote ad\`{e}les of
a scheme $X$. Note that flags $\eta_{0}>\cdots>\eta_{r}$ in $X$ also make
sense as flags for closed sub-schemes if all their entries are contained in them.

\begin{theorem}
[Structure Theorem]\label{TATE_StructureOfLocalAdelesProp}Suppose $X$ is a
Noetherian reduced excellent scheme of pure dimension $n$ and $\triangle
=\{(\eta_{0}>\cdots>\eta_{r})\}$ a saturated flag for some $r$.

\begin{enumerate}
\item \label{marker_adele_struct_one_partone}Then $A_{X}(\triangle
,\mathcal{O}_{X})$ is a finite direct product of $r$-local fields $\prod
K_{i}$ such that each last residue field is a finite field extension of
$\kappa(\eta_{r})$, the rational function field of $\overline{\{\eta_{r}%
\}}\subseteq X$. Moreover,%
\begin{equation}
A_{X}(\triangle^{\prime},\mathcal{O}_{X})\overset{(\ast)}{\subseteq}%
{\textstyle\prod}
\mathcal{O}_{i}\subseteq%
{\textstyle\prod}
K_{i}=A_{X}(\triangle,\mathcal{O}_{X})\text{,}\label{lCOA_D_3}%
\end{equation}
where $\mathcal{O}_{i}$ denotes the first ring of integers of $K_{i}$ and
$(\ast)$ is the normalization, a finite ring extension.

\item \label{marker_adele_struct_one_parttwo}If we regard $\triangle^{\prime}$
as a flag in the closed sub-scheme $\overline{\{\eta_{1}\}}$ instead, the
corresponding decomposition of Equation \ref{lCOA_D_3} exists for
$A_{\overline{\{\eta_{1}\}}}(\triangle^{\prime},\mathcal{O}_{\overline
{\{\eta_{1}\}}})$ as well, say%
\begin{equation}%
{\textstyle\prod}
k_{j}=A_{\overline{\{\eta_{1}\}}}(\triangle^{\prime},\mathcal{O}%
_{\overline{\{\eta_{1}\}}})\label{lCOA_D_3b}%
\end{equation}
(with a possibly different number of factors), and the residue fields of the
$\mathcal{O}_{i}$ in Equation \ref{lCOA_D_3} are finite extensions of these
field factors. Here to each $k_{j}$ correspond $\geq1$ factors in Equation
\ref{lCOA_D_3}.

\item If $X$ is of finite type over a field $k$, then each $K_{i}$ is
non-canonically ring isomorphic to $k^{\prime}((t_{1}))\cdots((t_{r}))$ for
$k^{\prime}/\kappa(\eta_{r})$ a finite field extension. If $k$ is perfect, it
can be promoted to a $k$-algebra isomorphism.

\item For a quasi-coherent sheaf $\mathcal{F}$, $A(\triangle,\mathcal{F}%
)\cong\mathcal{F}\otimes_{\mathcal{O}_{X}}A(\triangle,\mathcal{O}_{X})$.
\end{enumerate}
\end{theorem}

In claim (2) we state that for each field factor $k_{j}$ in Equation
\ref{lCOA_D_3b} there may be several field factors $K_{i}$ in Equation
\ref{lCOA_D_3}, but at least one, corresponding to it. In a concrete case such
a branching pattern may for example look like
\begin{equation}%
\raisebox{0.0061in}{\parbox[b]{1.2296in}{\begin{center}
\includegraphics[
height=0.6113in,
width=1.2296in
]%
{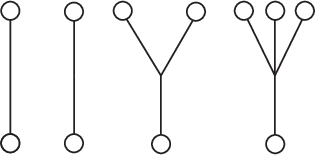}%
\\
{}%
\end{center}}}
\label{lDiag_Branching_1}%
\end{equation}
where the dots in the bottom row represent the field factors $k_{j}$, and the
dots of the top row the higher local fields $K_{i}$ corresponding the them,
that is: for each such factor the top ring of integers $\mathcal{O}%
_{i}\subseteq K_{i}$ has a finite field extension of $k_{j}$ as its respective
residue field.\medskip

We devote the entire section to the proof, split up into several
pieces.\medskip

Unravelling the inductive definition from Equation \ref{TATEMATRIX_l6} yields
the formula%
\begin{equation}
A_{X}(\triangle,\mathcal{F})=\underset{i_{0}}{\underleftarrow{\lim}}%
\underset{\eta_{0}}{\underrightarrow{\operatorname*{colim}}}\cdots
\underset{i_{r}}{\underleftarrow{\lim}}\underset{\eta_{r}}{\underrightarrow
{\operatorname*{colim}}}\,\mathcal{F}\otimes\mathcal{O}_{\left\langle \eta
_{r}\right\rangle }/\eta_{r}^{i_{r}}\underset{\mathcal{O}_{X}}{\otimes}%
\cdots\underset{\mathcal{O}_{X}}{\otimes}\mathcal{O}_{\left\langle \eta
_{0}\right\rangle }/\eta_{0}^{i_{0}}\text{,}\label{TATEMATRIX_l29}%
\end{equation}
where we have allowed ourselves the use of the following viewpoint/shorthands:

\begin{itemize}
\item As already the inner-most colimit corresponds to the localization at
$\eta_{r}$ (i.e. taking the stalk), we can henceforth work with rings and
modules instead of the scheme and its coherent sheaves. More precisely, we can
do this computation in $\mathcal{O}_{\eta_{r}}$-modules.

\item We (temporarily) use the notation%
\[
\mathcal{O}_{\eta_{a}}=\underset{\eta_{a}}{\underrightarrow
{\operatorname*{colim}}}\,\mathcal{O}_{\left\langle \eta_{a}\right\rangle }%
\]
for the system of finitely generated $\mathcal{O}_{\eta_{r}}$-submodules
$\mathcal{O}_{\left\langle \eta_{a}\right\rangle }\subseteq\mathcal{O}%
_{\eta_{a}}$.

\item We write $\eta_{i}$ not just for the scheme point $\eta_{i}$, but also
for its prime ideal $-$ under the transition to look at the stalk rather than
working with sheaves, the ideal sheaf of the reduced closed sub-scheme
$\overline{\{\eta_{i}\}}$ corresponds to a prime ideal.
\end{itemize}

Equation \ref{TATEMATRIX_l29}, the commutativity of tensor products with
colimits, and Lemma \ref{comalg_lemma_reminder}.\ref{comalg_fact_D1} settles
Theorem \ref{TATE_StructureOfLocalAdelesProp}, (4).\medskip

To proceed, let us consider the iterated limit/colimit%
\begin{equation}
\underset{i_{0}}{\underleftarrow{\lim}}\underset{\eta_{0}}{\underrightarrow
{\operatorname*{colim}}}\cdots\underset{i_{j-1}}{\underleftarrow{\lim}%
}\underset{\eta_{j-1}}{\underrightarrow{\operatorname*{colim}}}\,A_{j}%
\underset{\mathcal{O}_{\eta_{j}}}{\otimes}\mathcal{O}_{\left\langle \eta
_{j-1}\right\rangle }/\eta_{j-1}^{i_{r-1}}\underset{\mathcal{O}_{X}}{\otimes
}\cdots\underset{\mathcal{O}_{X}}{\otimes}\mathcal{O}_{\left\langle \eta
_{0}\right\rangle }/\eta_{0}^{i_{0}}\text{,}\label{TATEMATRIX_l15}%
\end{equation}
where $A_{j}$ is an $\mathcal{O}_{\eta_{j}}$-module yet to be defined.

\begin{example}
We had just seen that $A(\triangle,\mathcal{F})$ is of this shape for $j:=r$
and $A_{r}:=\mathcal{F}\otimes\widehat{\mathcal{O}_{\eta_{r}}}$.
\end{example}

As colimits commute with tensor products, we may rewrite the above expression
as%
\begin{align*}
& \cong\underset{i_{0}}{\underleftarrow{\lim}}\underset{\eta_{0}%
}{\underrightarrow{\operatorname*{colim}}}\cdots\underset{i_{j-1}%
}{\underleftarrow{\lim}}\,A_{j}\underset{\mathcal{O}_{\eta_{j}}}{\otimes
}(\underset{\eta_{j-1}}{\underrightarrow{\operatorname*{colim}}}%
\mathcal{O}_{\left\langle \eta_{r-1}\right\rangle })/\eta_{j-1}^{i_{r-1}%
}\underset{\mathcal{O}_{X}}{\otimes}\cdots\underset{\mathcal{O}_{X}}{\otimes
}\mathcal{O}_{\left\langle \eta_{0}\right\rangle }/\eta_{0}^{i_{0}}\\
& \cong\cdots\underset{\eta_{j-2}}{\underrightarrow{\operatorname*{colim}}%
}\,(\underset{i_{j-1}}{\underleftarrow{\lim}}\,A_{j}[(\mathcal{O}_{\eta_{j}%
}-\eta_{j-1})^{-1}]/\eta_{j-1}^{i_{j-1}})\underset{\mathcal{O}_{X}}{\otimes
}\cdots\underset{\mathcal{O}_{X}}{\otimes}\mathcal{O}_{\left\langle \eta
_{0}\right\rangle }/\eta_{0}^{i_{0}}%
\end{align*}
(as the colimit is just the localization $\mathcal{O}_{\eta_{r-1}}$ and then
use Lemma \ref{comalg_lemma_reminder}.\ref{comalg_fact_D2}). Then we have
recovered the shape of Equation \ref{TATEMATRIX_l15} for $j-1$. Hence,
inductively, $A_{X}(\triangle,\mathcal{F})=A_{0}$. Thus, Theorem
\ref{TATE_StructureOfLocalAdelesProp} is essentially a result on the structure
of $A_{0}$ for the special case $\mathcal{F}:=\mathcal{O}_{X}$.

\begin{definition}
\label{def_AjRingsShorthand}For the sake of an induction, we shall give the
following auxiliary rings a name:%
\begin{equation}
A_{j-1}:=\underset{i_{j-1}}{\underleftarrow{\lim}}A_{j}[(\mathcal{O}_{\eta
_{j}}-\eta_{j-1})^{-1}]/\eta_{j-1}^{i_{j-1}}\text{.}\label{TATEMATRIX_l17}%
\end{equation}
Equivalently, $A_{j}:=A(\eta_{j}>\cdots>\eta_{r},\mathcal{O}_{X})$ for $0\leq
j\leq r$.
\end{definition}

We now argue inductively along $j$:

\begin{lemma}
\label{TATE_StructureOfLocalAdelesProp_KeyLemma}Assume for some $j$ we have
shown the following:

\begin{enumerate}
\item $A_{j}$ is a faithfully flat Noetherian $\mathcal{O}_{\eta_{j}}$-algebra
of dimension $j$.

\item The maximal ideals of $A_{j}$ are precisely the primes minimal over
$\eta_{j}A_{j}$.

\item $A_{j}$ is a finite product of reduced $j$-dimensional local rings, each
complete with respect to its maximal ideal.
\end{enumerate}

Then the analogous statements for $j-1$ are true.
\end{lemma}

(We apologize to the reader for this slightly redundant formulation, but we
also intend the numbering as a guide along the steps in the proof.)

Beginning with $j:=r$ we had set $A_{r}:=\widehat{\mathcal{O}_{\eta_{r}}}$. It
is clear that all properties are satisfied since $\dim\widehat{\mathcal{O}%
_{\eta_{r}}}=\dim\mathcal{O}_{\eta_{r}}=\operatorname*{codim}_{X}\eta_{r}=r$.

\begin{proof}
\textit{(Step 1)} By construction $A_{j-1}$ is an $\eta_{j-1}A_{j-1}$-adically
complete Noetherian ring. $A_{j}$ is an $\mathcal{O}_{\eta_{j}}$-algebra
(property 1 for $A_{j}$), so by the universal property of localization we have%
\[%
\begin{array}
[c]{ccc}%
A_{j} & \longrightarrow & A_{j}[(\mathcal{O}_{\eta_{j}}-\eta_{j-1})^{-1}]\\
\uparrow &  & \uparrow\\
\mathcal{O}_{\eta_{j}} & \longrightarrow & (\mathcal{O}_{\eta_{j}}%
)_{\eta_{j-1}}\text{,}%
\end{array}
\]
but $(\mathcal{O}_{\eta_{j}})_{\eta_{j-1}}=\mathcal{O}_{\eta_{j-1}}$. So
$A_{j}[(\mathcal{O}_{\eta_{j}}-\eta_{j-1})^{-1}]$ and its $\eta_{j-1}$-adic
completion are $\mathcal{O}_{\eta_{j-1}}$-algebras. \textit{(Step 2: Maximal
ideals under localization)} Next, we determine the maximal ideals
$\mathfrak{m}_{i}$ of $A_{j-1}$: By Lemma \ref{comalg_lemma_reminder}%
.\ref{comalg_fact_D3}%
\[
\eta_{j-1}A_{j-1}\subseteq\operatorname*{rad}A_{j-1}:=%
{\textstyle\bigcap}
\mathfrak{m}_{i}\text{,}%
\]
i.e. they are in bijective correspondence with the maximal ideals of
$A_{j-1}/\eta_{j-1}A_{j-1}\cong A_{j}[(\mathcal{O}_{\eta_{j}}-\eta_{j-1}%
)^{-1}]/\eta_{j-1}$. The primes of the localization $A_{j}[(\mathcal{O}%
_{\eta_{j}}-\eta_{j-1})^{-1}]$ correspond bijectively to those primes
$P\subset A_{j}$ such that $P\cap(\mathcal{O}_{\eta_{j}}-\eta_{j-1}%
)=\varnothing$. By induction (properties 1 \& 2 for $A_{j}$) we know that the
maximal ideals in $A_{j}$ are the (finitely many) primes which are minimal
over $\eta_{j}A_{j}$. Moreover, $A_{j}$ is faithfully flat over $\mathcal{O}%
_{\eta_{j}}$, so by Lemma \ref{comalg_lemma_reminder}.\ref{comalg_fact_D11}
the primes $P$ minimal over $\eta_{j}A_{j}$ are those minimal with the
property $P\cap\mathcal{O}_{\eta_{j}}=\eta_{j}$. Hence, for them
$P\cap(\mathcal{O}_{\eta_{j}}-\eta_{j-1})=\eta_{j}-\eta_{j-1}\neq\varnothing$;
they all disappear in the localization. Thus, the maximal ideals of
$A_{j}[(\mathcal{O}_{\eta_{j}}-\eta_{j-1})^{-1}]$ correspond to primes in
$A_{j}$ having at least coheight $1$. This enforces that $A_{j}[(\mathcal{O}%
_{\eta_{j}}-\eta_{j-1})^{-1}]/\eta_{j-1}$ is zero-dimensional. Hence, the
maximal ideals $P$ of%
\begin{equation}
A_{j}[(\mathcal{O}_{\eta_{j}}-\eta_{j-1})^{-1}]/\eta_{j-1}\cong A_{j-1}%
/\eta_{j-1}A_{j-1}\label{lCOA_D_2}%
\end{equation}
are exactly the minimal primes of it, i.e. they are primes minimal over
$\eta_{j-1}A_{j-1}$ in $A_{j-1}$ (proving property 2 for $A_{j-1}$).
\textit{(Step 3: Faithful flatness)} $A_{j-1}$ is clearly flat over
$\mathcal{O}_{\eta_{j-1}}$ since it arises from repeated localization and
completion from $\mathcal{O}_{\eta_{j-1}}$ and both operations are flat.
Moreover, again by faithful flatness of $A_{j}$ over $\mathcal{O}_{\eta_{j}}$,
$\eta_{j-1}A_{j}\cap\mathcal{O}_{\eta_{j}}=\eta_{j-1}$, hence $\eta_{j-1}%
A_{j}\cap(\mathcal{O}_{\eta_{j}}-\eta_{j-1})=\eta_{j-1}-\eta_{j-1}%
=\varnothing$; so the ring in Equation \ref{lCOA_D_2} is not the zero ring. By
Lemma \ref{comalg_lemma_reminder}.\ref{comalg_fact_D6} this shows that
$A_{j-1}$ is even a faithfully flat $\mathcal{O}_{\eta_{j-1}}$-algebra
(proving property 1 for $A_{j-1}$). \textit{(Step 4: Reducedness)} Next, we
claim that $A_{j-1}$ is reduced. Both localization and completion (with
respect to arbitrary ideals) are regular morphisms by Lemma
\ref{comalg_lemma_reminder}.\ref{comalg_fact_D16}. Thus, the composition is
regular. It is also faithfully flat, so by faithfully flat ascent, Lemma
\ref{comalg_lemma_reminder}.\ref{comalg_fact_D17}, $A_{j-1}$ is reduced. In
completely the same fashion, $A_{j-1}/\eta_{j-1}A_{j-1}$ arises from iterated
localizations and completions from $\widehat{\mathcal{O}_{\eta_{r}}/\eta
_{j-1}}$. As $\eta_{j-1}$ is prime, $\mathcal{O}_{\eta_{r}}/\eta_{j-1}$ is a
domain and thus $\widehat{\mathcal{O}_{\eta_{r}}/\eta_{j-1}}$ is at least
reduced. Hence, the same argument implies that $A_{j-1}/\eta_{j-1}A_{j-1}$ is
reduced. Since we know now that $A_{j-1}/\eta_{j-1}A_{j-1}$ is reduced and
zero-dimensional, Lemma \ref{comalg_lemma_reminder}.\ref{comalg_fact_D4}
implies that we have%
\begin{equation}
A_{j-1}/\eta_{j-1}A_{j-1}\cong%
{\textstyle\prod\nolimits_{\mathfrak{m}}}
[A_{j}[(\mathcal{O}_{\eta_{j}}-\eta_{j-1})^{-1}]/\eta_{j-1}]_{\mathfrak{m}%
}\text{,}\label{TATEMATRIX_l14}%
\end{equation}
where $\mathfrak{m}$ runs through the finitely many (automatically minimal)
primes in $A_{j-1}/\eta_{j-1}A_{j-1}$. The localizations of the right-hand
side are reduced zero-dimensional local rings, i.e. by Lemma
\ref{comalg_lemma_reminder}.\ref{comalg_fact_D7} they must be fields. We
obtain a complete system of pairwise orthogonal idempotents $\overline{e_{1}%
},\ldots,\overline{e_{\ell}}\in A_{j-1}/\eta_{j-1}A_{j-1}$ giving the
decomposition of Equation \ref{TATEMATRIX_l14}. Using Lemma
\ref{comalg_lemma_reminder}.\ref{comalg_fact_D8} these idempotents lift
uniquely to a complete system of pairwise orthogonal idempotents $e_{1}%
,\ldots,e_{\ell}$ in $A_{j-1}$. Hence,%
\[
A_{j-1}\cong%
{\textstyle\prod\nolimits_{\mathfrak{m}}}
e_{i}A_{j-1}\text{.}%
\]
Hence, $A_{j-1}$ is a finite product of reduced $(j-1)$-dimensional local
rings (proving property 3 for $A_{j-1}$).
\end{proof}

After this preparation we are ready to establish the rest of Theorem
\ref{TATE_StructureOfLocalAdelesProp}.

\begin{proof}
[Proof of Thm. \ref{TATE_StructureOfLocalAdelesProp}]Recall that
$A_{X}(\triangle,\mathcal{O}_{X})=A_{0}$. From Lemma
\ref{TATE_StructureOfLocalAdelesProp_KeyLemma}, property 3, for $A_{0}$ it
follows that $A_{X}(\triangle,\mathcal{O}_{X})$ is a finite product of fields.
We may unwind $A_{X}(\triangle^{\prime},\mathcal{O}_{X})$ entirely analogously
as in Equation \ref{TATEMATRIX_l29} and obtain $A_{X}(\triangle^{\prime
},\mathcal{O}_{X})=A_{1}$ and thus (by the very definition of $A_{0} $,
Equation \ref{TATEMATRIX_l17})%
\begin{align*}
A_{X}(\triangle,\mathcal{O}_{X})=A_{0}  & =\underleftarrow{\lim}_{i_{0}}%
A_{1}[(\mathcal{O}_{\eta_{1}}-\eta_{0})^{-1}]/\eta_{0}^{i_{0}}\\
& =\underleftarrow{\lim}_{i_{0}}A_{X}(\triangle^{\prime},\mathcal{O}%
_{X})[(\mathcal{O}_{\eta_{1}}-\eta_{0})^{-1}]/\eta_{0}^{i_{0}}\text{.}%
\end{align*}
By Lemma \ref{TATE_StructureOfLocalAdelesProp_KeyLemma} the ring $A_{1}$ is a
finite product of one-dimensional reduced complete local rings. Denote by
$Q_{i}$ the minimal primes of $A_{1}$. Being reduced, the first arrow in%
\begin{align*}
A_{X}(\triangle^{\prime},\mathcal{O}_{X})=A_{1}  & \hookrightarrow%
{\textstyle\prod\nolimits_{i}}
A_{1}/Q_{i}\\
& \hookrightarrow%
{\textstyle\prod\nolimits_{i}}
A_{1}/Q_{i}[(\mathcal{O}_{\eta_{1}}-\eta_{0})^{-1}]\\
& \hookrightarrow\underleftarrow{\lim}_{i_{0}}%
{\textstyle\prod\nolimits_{i}}
A_{1}/Q_{i}[(\mathcal{O}_{\eta_{1}}-\eta_{0})^{-1}]/\eta_{0}^{i_{0}}=%
{\textstyle\prod\nolimits_{i}}
A_{X}(\triangle,\mathcal{O}_{X})/Q_{i}%
\end{align*}
is injective. The injectivity of the third follows from being Noetherian.
Consider the normalization of $A_{X}(\triangle^{\prime},\mathcal{O}_{X})$ in
$A_{X}(\triangle,\mathcal{O}_{X})$. By Lemma \ref{comalg_lemma_reminder}%
.\ref{comalg_fact_D9} the normalization arises as the product of the integral
closures $N_{i}$ of each $A_{X}(\triangle^{\prime},\mathcal{O}_{X})/Q_{i}$ in
the respective field of fractions $A_{X}(\triangle,\mathcal{O}_{X})/Q_{i} $.
Each of these is a finite extension since complete local rings are always
excellent; in particular, the entire normalization is a finite ring extension.
Moreover, $A_{X}(\triangle^{\prime},\mathcal{O}_{X})/Q_{i}$ is complete local
and has a unique minimal prime, so by Lemma \ref{comalg_lemma_reminder}%
.\ref{comalg_fact_D18} there is also just a single maximal ideal in its
normalization $N_{i}$, i.e. $N_{i}$ is local, too. We obtain
\[
A_{X}(\triangle^{\prime},\mathcal{O}_{X})\hookrightarrow%
{\textstyle\prod\nolimits_{i}}
A_{1}/Q_{i}\hookrightarrow%
{\textstyle\prod\nolimits_{i}}
N_{i}\hookrightarrow%
{\textstyle\prod\nolimits_{i}}
A_{X}(\triangle,\mathcal{O}_{X})/Q_{i}=A_{X}(\triangle,\mathcal{O}%
_{X})\text{.}%
\]
Each $N_{i}$ is a one-dimensional \textit{normal} complete local ring. Such a
local ring is a discrete valuation ring by Lemma \ref{comalg_lemma_reminder}%
.\ref{comalg_fact_D14}. Hence, $A_{X}(\triangle,\mathcal{O}_{X})$ is a finite
product of complete discrete valuation fields, $N_{i}$ are their respective
rings of integers. Under the normalization each local ring of $A_{X}%
(\triangle^{\prime},\mathcal{O}_{X})$ gets extended to a semi-local ring,
leading to a branching into some $g\geq1$ maximal ideals over it, and thus to
a branching like (for example)%
\[%
\raisebox{0.0061in}{\parbox[b]{1.2296in}{\begin{center}
\includegraphics[
height=0.6113in,
width=1.2296in
]%
{finringsplitting.eps}%
\\
{}%
\end{center}}}
\]
once we look at all local rings together: dots in the upper row represent
maximal ideals of the normalizations, i.e. factors $N_{i}$. Dots in the lower
row represent maximal ideals of $A_{X}(\triangle^{\prime},\mathcal{O}_{X})$,
so by Lemma \ref{TATE_StructureOfLocalAdelesProp} equivalently minimal primes
of $A_{X}(\triangle^{\prime},\mathcal{O}_{X})/\eta_{1}$. The respective
residue fields $\kappa_{i}:=N_{i}/\mathfrak{m}_{i}$ also follow to be finite
ring extensions of $(A_{X}(\triangle^{\prime},\mathcal{O}_{X})/Q_{i})/\eta
_{1}$. By direct inspection one sees that $A_{X}(\triangle^{\prime
},\mathcal{O}_{X})/\eta_{1}$ can be identified with $A_{\overline{\{\eta
_{1}\}}}(\triangle^{\prime},\mathcal{O}_{\overline{\{\eta_{1}\}}})$, i.e.
identified with $A(\triangle^{\prime},\mathcal{O}_{X})$, but taking
$X:=\overline{\{\eta_{1}\}}$ as the scheme and reading $\triangle^{\prime}$ as
an element of $S(\overline{\{\eta_{1}\}})_{r-1}$ instead of $S(X)_{r-1}$.
Therefore, by induction on the dimension of $X$, in the Figure above the lower
row dots equivalently correspond canonically to the factors $k_{j}$; and the
upper row dots to the $\kappa_{i}$. Moreover, again by induction, the ring
$A_{X}(\triangle^{\prime},\mathcal{O}_{X})/\eta_{1}$ is a finite product of
$(r-1)$-local fields in a canonical fashion, and the $\kappa_{i}$ finite field
extensions thereof. Going all the way down, by induction on $r$, this shows
that the last residue fields are finite extensions of%
\[
A_{\overline{\{\eta_{r}\}}}(\{\eta_{r}\},\mathcal{O}_{\overline{\{\eta_{r}\}}%
})=\underleftarrow{\lim}_{i}\mathcal{O}_{\overline{\{\eta_{r}\}},\eta_{r}%
}/\eta_{r}^{i}=\kappa(\eta_{r})\text{.}%
\]
directly from the definition of the ad\`{e}les, Equation \ref{TATEMATRIX_l6}.
This establishes part (2) of the claim.\newline Each $\kappa_{i}$ is (a finite
extension of $-$ and thus itself) a complete discrete valuation field whose
residue field is $(r-1)$-local. Thus, each $F_{i}$ is an $r$-local field. This
establishes part (1) of the theorem. Finally, if all the fields in this
induction are $k$-algebras, each complete discrete valuation ring $R_{i}$ is
equicharacteristic, so by Cohen's Structure Theorem, Prop.
\ref{Prop_CohenStructureTheorem}, there is a non-canonical isomorphism
$\simeq\kappa_{i}[[t]]$. Hence, $F_{i}\simeq\kappa_{i}((t))$ and inductively
this shows that $r$-local fields are multiple Laurent series fields, proving
part (3) of the theorem. If $k$ is perfect, pick each coefficient field such
that it is additionally a sub-$k$-algebra. Part (4) is just the sheaf version
of Lemma \ref{comalg_lemma_reminder}.\ref{comalg_fact_D1}.
\end{proof}

We can easily extract the higher local field structure of the local ad\`{e}les
from the previous result. Recall that we write $A_{Z}(-,\mathcal{-})$ to
denote ad\`{e}les of a scheme $Z$.

\begin{theorem}
[Structure Theorem II]\label{COA_StructureTheorem2}Suppose $X$ is a purely
$n$-dimensional reduced Noetherian excellent scheme and $\triangle=\{(\eta
_{0}>\cdots>\eta_{r})\} $ a saturated flag. Then we get a diagram%
\begin{equation}%
\bfig\node x(0,1200)[{A_{\overline{\{\eta_{0}\}}}(\triangle,\mathcal{O}_{X})}]
\node y(0,800)[{A_{\overline{\{\eta_{0}\}}}(\triangle^{\prime},\mathcal{O}%
_{X})}]
\node z(1000,800)[{A_{\overline{\{\eta_{1}\}}}(\triangle^{\prime},\mathcal
{O}_{X})}]
\node w(1000,400)[{A_{\overline{\{\eta_{1}\}}}(\triangle^{\prime\prime
},\mathcal{O}_{X})}]
\node u(2000,400)[{A_{\overline{\{\eta_{2}\}}}(\triangle^{\prime\prime
},\mathcal{O}_{X})}]
\node v(2000,0)[\vdots]
\arrow/{^{(}->}/[y`x;]
\arrow/{->>}/[y`z;]
\arrow/{^{(}->}/[w`z;]
\arrow/{->>}/[w`u;]
\arrow[v`u;]
\efig
\label{lDiagStruct2}%
\end{equation}
where

\begin{enumerate}
\item the upward arrows are precisely the inclusions of Theorem
\ref{TATE_StructureOfLocalAdelesProp} (part
\ref{marker_adele_struct_one_partone}), Equation \ref{lCOA_D_3};

\item the rightward arrows are taking the quotient of $A_{\overline{\{\eta
_{i}\}}}(\triangle^{\prime\cdots\prime},\mathcal{O}_{X})$ by $\eta_{i+1} $;

\item After replacing each ring in Diagram \ref{lDiagStruct2}, except the
initial upper-left one, by a canonically defined finite ring extension, it
splits canonically as a direct product of staircase-shaped diagrams of rings:
Each factor has the shape%
\[%
\bfig\node x(0,1200)[{\kappa((t_{1}))\cdots((t_{n}))}]
\node y(0,900)[{\kappa((t_{1}))\cdots[[t_{n}]]}]
\node z(900,900)[{\kappa((t_{1}))\cdots((t_{n-1}))}]
\node w(900,600)[{\kappa((t_{1}))\cdots[[t_{n-1}]]}]
\node u(1800,600)[{\kappa((t_{1}))\cdots((t_{n-2}))}]
\node v(1800,300)[\vdots]
\arrow/{^{(}->}/[y`x;]
\arrow/{->>}/[y`z;]
\arrow/{^{(}->}/[w`z;]
\arrow/{->>}/[w`u;]
\arrow[v`u;]
\efig
\]
In particular, each object in it is a direct factor of a finite extension of
the corresponding entry in Diagram \ref{lDiagStruct2}.

\begin{enumerate}
\item The upward arrows are going to the field of fractions,

\item The rightward arrows correspond to passing to the residue field.
\end{enumerate}

\item These factors are indexed uniquely by the field factors of the
upper-left entry $A_{X}(\triangle,\mathcal{O}_{X})=%
{\textstyle\prod}
K_{i}$. Each field factor $k_{j}$ of $A_{\overline{\{\eta_{i}\}}}%
(\triangle^{\prime\cdots\prime},\mathcal{O}_{X})$ in any row of Diagram
\ref{lDiagStruct2} corresponds to $\geq1$ field factors in the row above, such
that the respective residue field is a finite field extension of the chosen
$k_{j}$.
\end{enumerate}
\end{theorem}

An elaboration: As we already know, each $A_{\overline{\{\eta_{i}\}}%
}(\triangle^{\prime\cdots\prime},\mathcal{O}_{X})$ decomposes as a finite
direct product of fields. In particular, in Diagram \ref{lDiagStruct2} we get
such a decomposition in every single row (and of the two terms in each row, we
refer to the one following after \textquotedblleft$\twoheadrightarrow
$\textquotedblright), and there is a matching between the field factors of the
individual rows. For each field factor $k_{j}$ of a row, there are $\geq1$
field factors in the row above it, such that the respective residue field is
finite over the given $k_{j}$. If we follow the graphical representation of
this branching behaviour as in Diagram \ref{lDiag_Branching_1}, we get a
simple description of the entire branching behaviour from the top row all to
the bottom row: If we begin with the field factors of the upper-left entry
$A_{X}(\triangle,\mathcal{O}_{X})=%
{\textstyle\prod}
K_{i}$, the matching to the indexing of the field factors of $A_{\overline
{\{\eta_{i}\}}}(\triangle^{\prime\cdots\prime},\mathcal{O}_{X})$ in the rows
below is obtained by following the downward paths top-to-bottom in the tree
graph obtained by concatenating the branching diagrams (like Diagram
\ref{lDiag_Branching_1}) on each level, e.g. as in
\[%
{\includegraphics[
height=1.2531in,
width=1.5056in
]%
{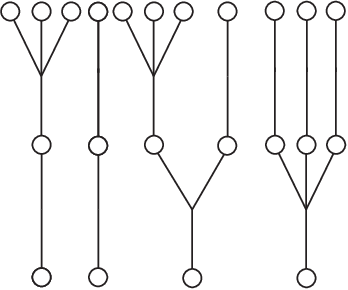}%
}
\]

\begin{proof}
The first step (both logically as well as visually in the diagram)%
\[%
\bfig\node x(0,1200)[{A_{\overline{\{\eta_{0}\}}}(\triangle,\mathcal{O}_{X})}]
\node y(0,800)[{A_{\overline{\{\eta_{0}\}}}(\triangle^{\prime},\mathcal{O}%
_{X})}]
\node z(1000,800)[{A_{\overline{\{\eta_{1}\}}}(\triangle^{\prime},\mathcal
{O}_{X})}]
\arrow/{^{(}->}/[y`x;]
\arrow/{->>}/[y`z;]
\efig
\]
is literally just Theorem \ref{TATE_StructureOfLocalAdelesProp} applied to the
scheme $X:=\overline{\{\eta_{0}\}}$ and the flag $\triangle$. To continue to
the next step, just inductively apply Theorem
\ref{TATE_StructureOfLocalAdelesProp} to $X:=\overline{\{\eta_{i}\}}$ instead
and note that the $i$-fold truncated flag of sub-schemes can be viewed as a
flag of sub-schemes in this smaller scheme as well.
\end{proof}

\begin{definition}
\label{def_ShorthandColimits}For a point (or ideal) $\eta$, we shall write%
\[
\underset{f\notin\eta}{\underrightarrow{\operatorname*{colim}}}\,\mathcal{O}%
\left\langle f^{-\infty}\right\rangle
\]
to denote the colimit over all coherent sub-sheaves (or finitely generated
sub-modules) of the localization $\mathcal{O}_{\eta}$.
\end{definition}

\begin{lemma}
\label{lemma_indproreformulation}Suppose $X$ is a purely $n$-dimensional
reduced Noetherian excellent scheme and $\triangle=\{(\eta_{0}>\cdots>\eta
_{r})\}$ a saturated flag. Suppose $\mathcal{F}$ is a coherent sheaf. Then the
following $\mathcal{O}_{\eta_{r}}$-modules are pairwise canonically isomorphic
for all $j=1,\ldots,r$:

\begin{enumerate}
\item $\mathcal{F}_{\triangle}\underset{\mathrm{def}}{=}A(\eta_{0}>\cdots
>\eta_{r},\mathcal{F})$.$\quad$(this intentionally does not depend on $j$)

\item $\underset{f_{0}\notin\eta_{0}}{\underrightarrow{\operatorname*{colim}}%
}\underset{i_{1}\geq1}{\underleftarrow{\lim}}\cdots\underset{f_{j-1}\notin
\eta_{j-1}}{\underrightarrow{\operatorname*{colim}}}\underset{i_{j}\geq
1}{\underleftarrow{\lim}}\,A\left(  \eta_{j+1}>\cdots>\eta_{r},\frac
{\mathcal{F}_{\eta_{j}}\otimes\mathcal{O}\left\langle f_{0}^{-\infty
}\right\rangle \otimes\cdots\otimes\mathcal{O}\left\langle f_{j-1}^{-\infty
}\right\rangle }{\eta_{1}^{i_{1}}+\cdots+\eta_{j}^{i_{j}}}\right)  $,\newline
where the denominator tacitly is to be understood as $(\eta_{1}^{i_{1}}%
+\cdots+\eta_{j}^{i_{j}})\cdot($numerator$)$.

\item $\underset{f_{0}\notin\eta_{0}}{\underrightarrow{\operatorname*{colim}}%
}\underset{i_{1}\geq1}{\underleftarrow{\lim}}\cdots\underset{f_{j-1}\notin
\eta_{j-1}}{\underrightarrow{\operatorname*{colim}}}\underset{i_{j}\geq
1}{\underleftarrow{\lim}}\underset{f_{j}\notin\eta_{j}}{\underrightarrow
{\operatorname*{colim}}}\,A\left(  \eta_{j+1}>\cdots>\eta_{r},\frac
{\mathcal{F}\otimes\mathcal{O}\left\langle f_{0}^{-\infty}\right\rangle
\otimes\cdots\otimes\mathcal{O}\left\langle f_{j}^{-\infty}\right\rangle
}{\eta_{1}^{i_{1}}+\cdots+\eta_{j}^{i_{j}}}\right)  $,\newline where the
denominator tacitly is to be understood as $(\eta_{1}^{i_{1}}+\cdots+\eta
_{j}^{i_{j}})\cdot($numerator$)$.

\item $\underset{L_{1}}{\underrightarrow{\operatorname*{colim}}}%
\underset{L_{1}^{\prime}}{\underleftarrow{\lim}}\cdots\underset{L_{j}%
}{\underrightarrow{\operatorname*{colim}}}\underset{L_{j}^{\prime}%
}{\underleftarrow{\lim}}\,A\left(  \eta_{j+1}>\cdots>\eta_{r},\frac{L_{j}%
}{L_{j}^{\prime}}\right)  $,\newline where for all $\ell=1,\ldots,j$ the
$L_{\ell}$ run through all finitely generated $\mathcal{O}_{\eta_{\ell}}%
$-submodules of%
\[
\frac{L_{\ell-1}}{L_{\ell-1}^{\prime}}\text{ (in case }\ell>1\text{)}%
\qquad\text{or}\qquad\mathcal{F}_{\eta_{0}}\text{ (in case }\ell=1\text{)}%
\]
in ascending order; and the $L_{\ell}^{\prime}\subseteq L_{\ell}$ run through
all full rank finitely generated $\mathcal{O}_{\eta_{\ell}}$-submodules of
$L_{\ell}$ in descending order.
\end{enumerate}
\end{lemma}

Statement (1) intentionally does not depend on the choice of $j$. We merely
use the numbering of the above statement as a guideline through the steps of
the proof. Overall, we are just collecting a large number of different ways to
express the same object.

\begin{proof}
First of all, recall that%
\[
A(\eta_{0}>\cdots>\eta_{r},\mathcal{F})=A(\eta_{0}>\cdots>\eta_{r}%
,\mathcal{O}_{X})\otimes_{\mathcal{O}_{X}}\mathcal{F}\text{,}%
\]
and we see that it suffices to prove the claim for $\mathcal{F}:=\mathcal{O}%
_{X}$. The isomorphy of the objects in (2) and (3) is clear from the
definition since $\mathcal{F}_{\eta_{j}}$ will generally only be a
quasi-coherent sheaf, see Equation \ref{lTX1}. Next, we demonstrate the
isomorphism between (2) and (4) for any fixed $j$: Suppose we are given
$\ell\geq1$. Define for any $\mathcal{O}\left\langle f_{\ell-1}^{-\infty
}\right\rangle $ in the $\ell$-th colimit and $i_{\ell}\geq1$ in the $\ell$-th
limit%
\begin{align}
&  L_{\ell}:=\mathcal{O}_{\eta_{\ell}}\text{-span of }\mathcal{O}\left\langle
f_{0}^{-\infty},\ldots,f_{\ell-1}^{-\infty}\right\rangle \text{\quad}%
\subseteq\text{ }\frac{L_{\ell-1}}{L_{\ell-1}^{\prime}}\text{ (if }%
\ell>1\text{) or }\mathcal{O}_{\eta_{0}}\text{ (if }\ell=1\text{)}%
\label{lwla1}\\
&  L_{\ell}^{\prime}:=\mathcal{O}_{\eta_{\ell}}\text{-span of }\eta_{1}%
^{i_{1}}+\cdots+\eta_{\ell}^{i_{\ell}}\text{\quad}\subseteq\text{ }%
\frac{L_{\ell-1}}{L_{\ell-1}^{\prime}}\text{ (if }\ell>1\text{) or
}\mathcal{O}_{\eta_{0}}\text{ (if }\ell=1\text{).}\nonumber
\end{align}
As $\mathcal{O}\left\langle f_{\ell-1}^{-\infty}\right\rangle $ is a coherent
sheaf by construction, cf. Definition \ref{def_ShorthandColimits}, $L_{\ell}$
is a finitely generated $\mathcal{O}_{\eta_{\ell}}$-module. The same is true
for $L_{\ell}^{\prime}$ and we clearly have $L_{\ell}^{\prime}\subseteq
L_{\ell}$. This shows that there is a morphism between the indexing sets of
the limits/colimits in (2) to the indexing sets of the $L_{\ell},L_{\ell
}^{\prime}$ in (4). Moreover, we unravel by induction%
\begin{align*}
&  \frac{L_{\ell}}{L_{\ell}^{\prime}}=\frac{\mathcal{O}_{\eta_{\ell}}%
\cdot\mathcal{O}\left\langle f_{0}^{-\infty}\right\rangle \otimes\cdots
\otimes\mathcal{O}\left\langle f_{\ell-1}^{-\infty}\right\rangle }{\eta
_{1}^{i_{1}}+\cdots+\eta_{\ell}^{i_{\ell}}}\\
&  \qquad\text{(a quotient of sub-spaces of }\frac{L_{\ell-1}}{L_{\ell
-1}^{\prime}}\text{ for }\ell>1\text{, or }\mathcal{O}_{\eta_{0}}\text{ if
}\ell=1\text{).}%
\end{align*}
We see that $A\left(  \eta_{j+1}>\cdots>\eta_{r},L_{\ell}/L_{\ell}^{\prime
}\right)  $ agrees with the $A(-,-)$ appearing in formulation (2). Summarized,
the ind-pro limits of (2) define a sub-system of the ind-pro limits in (4),
running over the same objects as in (2). Next, note that for all finitely
generated $\mathcal{O}_{\eta_{\ell}}$-submodules of $\frac{L_{\ell-1}}%
{L_{\ell-1}^{\prime}}$ or $L_{\ell}$ we can lift generators from sub-quotients
to rational functions, allowing us to form a co-final system within the
ind-pro limits of (2). This implies that (4) is canonically isomorphic to (2).
Now, prove the full claim by induction on $j$: We verify (1)$\cong$(2) in the
special case $j=1$ by hand. Now assume (3) for any given $j$. Then by
unwinding the definition of $A(\eta_{j+1}>\cdots>\eta_{r},-)$ we literally
obtain (2) for $j+1$. Since we already have proven (3)$\cong$(2) for all $j$,
this sets up the entire induction along $j $, establishing our claim.
\end{proof}

This result has a particularly nice consequence for flags of the maximal
possible length:

\begin{corollary}
\label{cor_Obvious}Suppose $X$ is a purely $n$-dimensional reduced Noetherian
excellent scheme and $\triangle=\{(\eta_{0}>\cdots>\eta_{n})\}$ a saturated
flag. Suppose $\mathcal{F}$ is a coherent sheaf. Then%
\[
A(\eta_{0}>\cdots>\eta_{n},\mathcal{F})=\underset{L_{1}}{\underrightarrow
{\operatorname*{colim}}}\underset{L_{1}^{\prime}}{\underleftarrow{\lim}}%
\cdots\underset{L_{n}}{\underrightarrow{\operatorname*{colim}}}\underset
{L_{n}^{\prime}}{\underleftarrow{\lim}}{}\frac{L_{n}}{L_{n}^{\prime}}\text{,}%
\]
where for all $\ell=1,\ldots,n$, the $L_{\ell}$ run through all Beilinson
lattices (for the flag $\eta_{\ell-1}>\cdots>\eta_{n}$) in%
\[
\frac{L_{\ell-1}}{L_{\ell-1}^{\prime}}\text{ (in case }\ell>1\text{)}%
\qquad\text{or}\qquad\mathcal{F}_{\eta_{0}}\text{ (in case }\ell=1\text{)}%
\]
in ascending order; and the $L_{\ell}^{\prime}\subseteq L_{\ell}$ run through
all contained Beilinson lattices in descending order.
\end{corollary}

\begin{proof}
Just apply Lemma \ref{lemma_indproreformulation} in the special case $r=n$.
\end{proof}

In the formulation of the following lemma we shall employ the notation
$\widehat{(-)}$, which refers to omission here and not to completion or the like.

\begin{lemma}
\label{lemma_indproreform2}Suppose $X$ is a purely $n$-dimensional reduced
scheme of finite type over a field $k$ and $\triangle=\{(\eta_{0}>\cdots
>\eta_{n})\}$ a saturated flag.

\begin{enumerate}
\item Assume we are given finitely generated $\mathcal{O}_{\eta_{0}}$-modules
$M_{1},M_{2}$. Then a $k$-vector space morphism%
\[
f\in\operatorname*{Hom}\nolimits_{k}(M_{1\triangle},M_{2\triangle})
\]
is an element of $\operatorname*{Hom}\nolimits_{\triangle}(M_{1},M_{2})$ if
and only if

\begin{enumerate}
\item one can provide a final and co-final collection of Beilinson lattices
$L_{\ell}^{\prime}\subseteq L_{\ell}$ of $M_{1}$, and $N_{\ell}\subseteq
N_{\ell}^{\prime}$ of $M_{2}$ (in either case for $\ell=1,\ldots,n$) as in
Corollary \ref{cor_Obvious}, such that

\item there exists a compatible system of $k$-vector space morphisms%
\[
\frac{L_{n}}{L_{n}^{\prime}}\rightarrow\frac{N_{n}}{N_{n}^{\prime}}%
\]
inducing the map $f$ in the iterated Ind- and Pro-diagrams%
\begin{align*}
f:M_{1\triangle} &  \rightarrow M_{2\triangle}\\
\underset{L_{1}}{\underrightarrow{\operatorname*{colim}}}\underset
{L_{1}^{\prime}}{\underleftarrow{\lim}}\cdots\underset{L_{n}}{\underrightarrow
{\operatorname*{colim}}}\underset{L_{n}^{\prime}}{\underleftarrow{\lim}}%
\frac{L_{n}}{L_{n}^{\prime}} &  \rightarrow\underset{N_{1}}{\underrightarrow
{\operatorname*{colim}}}\underset{N_{1}^{\prime}}{\underleftarrow{\lim}}%
\cdots\underset{N_{n}}{\underrightarrow{\operatorname*{colim}}}\underset
{N_{n}^{\prime}}{\underleftarrow{\lim}}\frac{N_{n}}{N_{n}^{\prime}}\text{.}%
\end{align*}

\end{enumerate}

\item Suppose $f\in\operatorname*{Hom}\nolimits_{\triangle}(M_{1},M_{2})$.
Then $f\in I_{i\triangle}^{+}(M_{1},M_{2})$ if and only if $f$ admits a
factorization of the shape%
\[
\underset{L_{1}}{\underrightarrow{\operatorname*{colim}}}\underset
{L_{1}^{\prime}}{\underleftarrow{\lim}}\cdots\underset{L_{n}}{\underrightarrow
{\operatorname*{colim}}}\underset{L_{n}^{\prime}}{\underleftarrow{\lim}}%
\frac{L_{n}}{L_{n}^{\prime}}\rightarrow\underset{N_{1}}{\underrightarrow
{\operatorname*{colim}}}\underset{N_{1}^{\prime}}{\underleftarrow{\lim}}%
\cdots\underset{N_{i}}{\underrightarrow{\widehat{\operatorname*{colim}}}%
}\cdots\underset{N_{n}}{\underrightarrow{\operatorname*{colim}}}%
\underset{N_{n}^{\prime}}{\underleftarrow{\lim}}\frac{N_{n}}{N_{n}^{\prime}%
}\text{,}%
\]
i.e. instead of a colimit running over all $N_{i}$, it factors through a fixed
$N_{i}$ (depending only on $N_{1}$, $N_{1}^{\prime}$, $\ldots$, $N_{i-1}$,
$N_{i-1}^{\prime}$).

\item Similarly, $f\in I_{i\triangle}^{-}(M_{1},M_{2})$ holds if and only if
$f$ admits a factorization of the shape%
\[
\underset{L_{1}}{\underrightarrow{\operatorname*{colim}}}\underset
{L_{1}^{\prime}}{\underleftarrow{\lim}}\cdots\underset{L_{i}}{\underleftarrow
{\widehat{\lim}}}\cdots\underset{L_{n}}{\underrightarrow{\operatorname*{colim}%
}}\underset{L_{n}^{\prime}}{\underleftarrow{\lim}}\frac{L_{n}}{L_{n}^{\prime}%
}\rightarrow\underset{N_{1}}{\underrightarrow{\operatorname*{colim}}}%
\underset{N_{1}^{\prime}}{\underleftarrow{\lim}}\cdots\underset{N_{n}%
}{\underrightarrow{\operatorname*{colim}}}\underset{N_{n}^{\prime}%
}{\underleftarrow{\lim}}\frac{N_{n}}{N_{n}^{\prime}}\text{,}%
\]
i.e. instead of having the limit run over all $L_{i}$, it vanishes on a fixed
$L_{i}$ (depending only on $L_{1}$, $L_{1}^{\prime}$, $\ldots$, $L_{i-1}$,
$L_{i-1}^{\prime}$).
\end{enumerate}
\end{lemma}

\begin{proof}
In view of Cor. \ref{cor_Obvious}, this follows rather straightforwardly from
Beilinson's Definition \ref{Definition_AdeleOperatorIdeals}. For (1): Once
$f\in\operatorname*{Hom}\nolimits_{\triangle}(M_{1},M_{2})$ holds true for a
$k$-linear map $f$, Definition \ref{Definition_AdeleOperatorIdeals} allows us
to produce many such factorizations; firstly over%
\[
\left(  \frac{L_{1}}{L_{1}^{\prime}}\right)  _{\triangle^{\prime}}%
\rightarrow\left(  \frac{N_{1}}{N_{1}^{\prime}}\right)  _{\triangle^{\prime}%
}\text{,}%
\]
(for any prescribed $L_{1}$ and $N_{1}^{\prime}$) and then inductively further
down the flag $\triangle$. Conversely, given such factorizations, they clearly
define a $k$-linear map and the condition of Definition
\ref{Definition_AdeleOperatorIdeals} follows from the map being of this shape.
(2) and (3) follow just from unravelling Beilinson's definition in view of
Cor. \ref{cor_Obvious} and the fact that all $L_{\ell},L_{\ell}^{\prime}$ (for
all $\ell=1,\ldots,n$) are Beilinson lattices.
\end{proof}

\begin{proposition}
For $\triangle=\{(\eta_{0}>\cdots>\eta_{n})\}$ and $\mathcal{F}$ a coherent
sheaf, the presentation of Corollary \ref{cor_Obvious},%
\[
\mathcal{F}_{\triangle}=\underset{L_{1}}{\underrightarrow
{\operatorname*{colim}}}\underset{L_{1}^{\prime}}{\underleftarrow{\lim}}%
\cdots\underset{L_{n}}{\underrightarrow{\operatorname*{colim}}}\underset
{L_{n}^{\prime}}{\underleftarrow{\lim}}{}\frac{L_{n}}{L_{n}^{\prime}}\text{,}%
\]
also equips $\mathcal{F}_{\triangle}$ with a canonical structure as an
$n$-Tate object in ST $k$-modules (with their exact structure, Prop.
\ref{prop_STmodsQuasiAbelian}). Or, executing the colimits and limits, as an
ST $k$-module itself.
\end{proposition}

\begin{proof}
We only need to know that the transition maps of the Ind- and Pro-diagrams are
admissible monics and epics. This was already shown by Yekutieli, albeit in a
slightly different language \cite[Lemma 4.3, (2) and (4)]{MR3317764}. For the
second claim, we only need to know that the respective limits and colimits
exist in ST\ modules; this is \cite[Lemma 4.3, (3) and (6)]{MR3317764}.
\end{proof}

\begin{theorem}
[Structure Theorem III]\label{Thm_COA_StructureTheorem3}Suppose $X$ is a
purely $n$-dimensional reduced scheme of finite type over a field $k$ and
$\triangle=\{(\eta_{0}>\cdots>\eta_{n})\}$ a saturated flag. Then each direct
summand of the upper-left object in Diagram \ref{lDiagStruct2} of Theorem
\ref{COA_StructureTheorem2} carries a canonical structure

\begin{enumerate}
\item of $n$-local fields,

\item of objects in $\left.  n\text{-}\mathsf{Tate}(\mathsf{Ab})\right.  $,
i.e. with values in abelian groups,

\item of objects in $\left.  n\text{-}\mathsf{Tate}(\mathsf{Vect}_{f})\right.
$, i.e. with values in finite-dimensional $k$-vector spaces,

\item of $k$-algebras,

\item (if $k$ is perfect) of topological $n$-local fields in the sense of Yekutieli,
\end{enumerate}

and one can find (non-canonically) a simultaneous field and $\left.
n\text{-}\mathsf{Tate}(\mathsf{Ab})\right.  $ isomorphism to a multiple
Laurent series field%
\[%
\bfig\node x(0,1200)[{\kappa((t_{1}))\cdots((t_{n}))}]
\node y(0,900)[{\kappa((t_{1}))\cdots[[t_{n}]]}]
\node z(900,900)[{\kappa((t_{1}))\cdots((t_{n-1}))}]
\node w(900,600)[{\kappa((t_{1}))\cdots[[t_{n-1}]]}]
\node u(1800,600)[{\kappa((t_{1}))\cdots((t_{n-2}))}]
\node v(1800,300)[\vdots]
\arrow/{^{(}->}/[y`x;]
\arrow/{->>}/[y`z;]
\arrow/{^{(}->}/[w`z;]
\arrow/{->>}/[w`u;]
\arrow[v`u;]
\efig
\]
with its standard field and $\left.  n\text{-}\mathsf{Tate}(\mathsf{Ab}%
)\right.  $ structure. Here $\kappa/k$ is a finite field extension.
\end{theorem}

If one is happy with plain field isomorphisms without extra structure, this is
of course part of the original results of Parshin and Beilinson. The
construction and very definition of the canonical TLF structure/ST module
structure is due to Yekutieli \cite{MR1213064}, \cite{MR3317764}. However, we
know from Example \ref{example_Yekutieli}, going back to Yekutieli's work,
that a general field isomorphism will not preserve this structure, and from
its variation Example \ref{example_YekutieliIndPro} that it would also not
preserve the $n$-Tate structure.

\subsection{Proof of Theorem \ref{Thm_COA_StructureTheorem3}}

We shall devote this entire subsection to the proof of Theorem
\ref{Thm_COA_StructureTheorem3}.

\subsubsection{Step 0: Preamble on our usage of Tate categories}

The argument will deal with objects which may simultaneously be regarded as
objects in the category of rings, ST modules and/or Tate objects over a base
category. There is a slight change with regards to what categories we work in
precisely, depending on whether $k$ is perfect or not. We make this case
distinction here, and it is valid for the rest of the section: Specifically,

\begin{itemize}
\item if $k$ is perfect, we work in the categories of $k$-algebras, ST modules
and Tate objects of finite-dimensional $k$-vector spaces, and as a shorthand
write%
\[
\left.  n\text{-}\mathsf{Tate}\right.  :=\left.  n\text{-}\mathsf{Tate}%
(\mathsf{Vect}_{f})\right.  \text{.}%
\]

\item If $k$ is not perfect, we work in the categories of rings and Tate
objects of all abelian groups. We use the shorthand%
\[
\left.  n\text{-}\mathsf{Tate}\right.  :=\left.  n\text{-}\mathsf{Tate}%
(\mathsf{Ab})\right.  \text{.}%
\]
In this case, simply ignore all statements about $k$-algebra structures,
$k$-vector space structures or ST\ module structures in the proof below.
\end{itemize}

\subsubsection{Step 1: Definition of auxiliary rings}

Suppose we are in the situation of the assumptions of the theorem.

\begin{definition}
For $j=0,1,\ldots,n$, we define a ring%
\begin{equation}
C_{j}:=\underset{i_{j}\geq1}{\underleftarrow{\lim}}\underset{f_{j}\notin
\eta_{j}}{\underrightarrow{\operatorname*{colim}}}\cdots\underset{i_{n}\geq
1}{\underleftarrow{\lim}}\frac{\mathcal{O}_{\eta_{n}}\otimes\mathcal{O}%
\left\langle f_{j}^{-\infty}\right\rangle \otimes\mathcal{O}\left\langle
f_{j+1}^{-\infty}\right\rangle \otimes\cdots\otimes\mathcal{O}\left\langle
f_{n-1}^{-\infty}\right\rangle }{\eta_{j-1}+\eta_{j}^{i_{j}}+\cdots+\eta
_{n}^{i_{n}}}\text{.}\label{lVC_1}%
\end{equation}
We denote by $q$ the quotient map%
\[
q:C_{j}\twoheadrightarrow C_{j}/\eta_{j}\text{.}%
\]

\end{definition}

We can equip $C_{j}$ and $q$ with a lot more structure than just being a
$k$-algebra and a $k$-algebra morphism: They also carry natural structures as

\begin{enumerate}
\item (Tate objects) Reading the limits and colimits in Equation \ref{lVC_1}
as diagrams, the definition describes an object in $\mathsf{Pro}^{a}(\left.
(n-j)\text{-}\mathsf{Tate}\right.  )$. In this category the definition of $q$
also makes sense, and it is an admissible epic, since it is the natural
mapping from a Pro-diagram to one of its entries.

\item (as ST\ modules) Equation \ref{lVC_1} also defines an object in
Yekutieli's category of ST\ modules. Equip the inner term with the fine ST
module structure. (Much like in Example
\ref{example_YekutieliLaurentSTRingStructure}) each limit is equipped with its
limit topology, resulting again in an ST\ module \cite[Lemma 1.2.19]%
{MR1213064}, and equip the colimits, which are localizations, with the fine
topology over the ring we are localizing (or equivalently with the colimit
topology \cite[Cor. 1.2.6]{MR1213064}); this makes them ST rings again. Then
$q$ is an admissible epic in ST modules and (equivalently) induces the
quotient topology by \cite[Lemma 4.3]{MR3317764}.
\end{enumerate}

We return to regarding $C_{j}$ as a ring, and study its properties:

\begin{lemma}
We have the following ring-theoretic properties:

\begin{enumerate}
\item $C_{j}$ is a one-dimensional $\eta_{j}$-adically complete semi-local
$k$-algebra with Jacobson radical $\eta_{j}$.

\item $C_{j}/\eta_{j}$ is a reduced Artinian ring.

\item $C_{j}=A_{X}(\eta_{j}>\cdots>\eta_{n},\mathcal{O}_{X})/\eta
_{j-1}=A_{\overline{\{\eta_{j-1}\}}}(\eta_{j}>\cdots>\eta_{n},\mathcal{O}%
_{\overline{\{\eta_{j-1}\}}})$.
\end{enumerate}
\end{lemma}

\begin{proof}
This is fairly clear: It is visibly an $\eta_{j}$-adically complete semi-local
ring with Jacobson radical $\eta_{j}$ and minimal primes all lying over
$\eta_{j-1}$. It follows that $C_{j}$ is one-dimensional. The identification
in (3) follows literally from unwinding the definition.
\end{proof}

Next, consider the normalization of $C_{j}$. We denote it by $C_{j}^{\prime}
$. This is a finite ring extension/$k$-algebra extension (since $C_{j}$ is
excellent). It is a finite product of complete discrete valuation rings, say
indexed by a variable $t$, i.e.%
\begin{equation}
C_{j}^{\prime}=%
{\textstyle\prod}
\mathcal{O}_{j,t}\qquad\text{with residue fields }\kappa_{j,t}:=\mathcal{O}%
_{j,t}/\mathfrak{m}_{j,t}\label{laact1}%
\end{equation}
and by the finiteness of normalization each $\kappa_{j,t}$ is finite over
$C_{j}/\eta_{j}$.

Consider $\operatorname*{Quot}(C_{j+1})$:\ It is the total ring of quotients
of $C_{j+1}$ as a ring and $k$-algebra. However, as this is a localization and
thus can be written as a colimit over its finitely generated $C_{j+1}%
$-submodules, it also can be given a natural structure as an $(n-j+1)$-Tate
object, or respectively as an ST module.

\begin{lemma}
\label{lem_helperquot}$C_{j}/\eta_{j}=\operatorname*{Quot}(C_{j+1})$. This is
true as rings, as $k$-algebras, Tate objects and ST modules.
\end{lemma}

\begin{proof}
The verification is immediate from the definitions, in each category.
\end{proof}

\subsubsection{Step 2: Setting up the auxiliary diagram}

The objects which we have just defined, fit into a big commutative diagram%
\begin{equation}%
\bfig\node eup(0,800)[\vdots]
\node yrings(0,400)[C_{j-1}/{\eta_{j-1}}]
\node xrings(0,0)[{\prod\mathcal{O}_{j,t}}]
\node xres(800,0)[{\prod\kappa_{j,t}}]
\node base(0,-400)[{C_j}]
\node baseres(800,-400)[{C_{j}/{\eta_{j}}}]
\arrow/{>>}/[xrings`xres;]
\arrow/{>}/[base`xrings;\gamma]
\arrow/{@{.>}@/^2em/}/[base`yrings;]
\arrow/{>>}/[base`baseres;]
\arrow/{-->}/[baseres`xres;]
\arrow/{.>}/[xrings`yrings;]
\arrow/{-->}/[yrings`eup;]
\node wxrings(800,-800)[{\prod\mathcal{O}_{j+1,t}}]
\node wxres(1600,-800)[{\prod\kappa_{j+1,t}}]
\node wbase(800,-1200)[{C_{j+1}}]
\node wbaseres(1600,-1200)[{C_{j+1}/{\eta_{j+1}}}]
\node ebot(1600,-1600)[\vdots]
\arrow/{>>}/[wxrings`wxres;]
\arrow/{>}/[wbase`wxrings;\gamma]
\arrow/{@{.>}@/^2em/}/[wbase`baseres;]
\arrow/{>>}/[wbase`wbaseres;]
\arrow/{-->}/[wbaseres`wxres;]
\arrow/{.>}/[wxrings`baseres;]
\arrow/{.>}/[ebot`wbaseres;]
\efig
\label{lFigA}%
\end{equation}
and on the upper left this diagram commences with%
\[%
\bfig\node yrings(0,400)[C_{0}]
\node xrings(0,0)[{\prod\mathcal{O}_{1,t}}]
\node xres(800,0)[{\prod\kappa_{1,t}}]
\node base(0,-400)[{C_1}]
\node baseres(800,-400)[\ddots]
\arrow/{>>}/[xrings`xres;]
\arrow/{>}/[base`xrings;\gamma]
\arrow/{@{.>}@/^2em/}/[base`yrings;]
\arrow/{>>}/[base`baseres;]
\arrow/{-->}/[baseres`xres;]
\arrow/{.>}/[xrings`yrings;]
\efig
\]
Let us quickly go through the various objects and arrows: Here $\mathcal{O}%
_{j,t}$ and $\kappa_{j,t}$ are the discrete valuation rings/rings of integers
resp. residue fields of Equation \ref{laact1}. We have allowed ourselves the
tiny abuse of notation to write \textquotedblleft$t$\textquotedblright\ to
index the factors of the products, even though for different $j$, the variable
$t$ will run through (in general) different finite indexing sets. Moreover,

\begin{enumerate}
\item (as rings, $k$-algebras) the upward dotted arrows are always the
inclusion into the total ring of quotients by Lemma \ref{lem_helperquot}.
These maps are injective. In the case of the unbent dotted arrow it is
additionally a product of the inclusions of the discrete valuation rings
$\mathcal{O}$ into their field of fractions. The maps denoted by $\gamma$ are
normalizations; the integral closure in the total ring of quotients. The
dashed upward arrows are products of finite field extensions. Each quotient
$C_{(-)}/\eta_{(-)}$ is itself a product of fields.

\item (as Tate objects, ST modules) the upward bent arrows are admissible
monics in Tate objects since they are the inclusion of an entry of an
admissible Ind-diagram into the Ind-object defined by this diagram.
Analogously, an admissible monic in ST modules for essentially the same
reason, just with the colimit carried out.
\end{enumerate}

Define both $\kappa_{0,t}$ and $\kappa_{0,t}^{\ast}$ as the field of fractions
of $\mathcal{O}_{1,t}$. Consider the left-most \textit{upward} arrow
$\prod\mathcal{O}_{1,t}\rightarrow\prod\kappa_{0,t}$ in the above Figure
\ref{lFigA}. This arrow is the product of maps $\mathcal{O}_{1,t}%
\hookrightarrow\kappa_{0,t}$, but these maps will usually not be the inclusion
of rings of integers into their field of fractions. We now define certain
rings, recursively: For $j=1,\ldots,n$ (and run through these in this order):

Denote by $\mathcal{O}_{j,t}^{\ast}$ the integral closure of $\mathcal{O}%
_{j,t}$ inside $\kappa_{j-1,t}^{\ast}$. Since the $\mathcal{O}_{j,t}$ are
complete discrete valuation rings, the $\mathcal{O}_{j,t}^{\ast}$ are also
complete discrete valuation rings, cf. Lemma \ref{comalg_lemma_reminder}%
.\ref{comalg_fact_D12} (there can only be one factor since we are inside a
field). We write $\kappa_{1,t}^{\ast}$ for their residue fields, so that
$\kappa_{1,t}^{\ast}/\kappa_{1,t}$ is a finite field extension. Now proceed to
$j+1$.

Let us quickly explain how to fit these new objects into Figure \ref{lFigA}:
For $j$, we get%
\[%
\bfig\node eup(0,800)[\vdots]
\node yrings(0,400)[{\prod\mathcal{O}^{\ast}_{j,t}}]
\node yringsright(800,400)[{\prod\mathcal{\kappa}_{j,t}^{\ast}}]
\arrow/{>>}/[yrings`yringsright;]
\node xrings(0,0)[{\prod\mathcal{O}_{j,t}}]
\node xres(800,0)[{\prod{\kappa}_{j,t}}]
\node base(0,-400)[{C_j}]
\node baseres(800,-400)[{C_{j}/{\eta_{j}}}]
\arrow/{>>}/[xrings`xres;]
\arrow/{>}/[base`xrings;\gamma]
\arrow/{@{.>}@/^2em/}/[base`eup;]
\arrow/{>>}/[base`baseres;]
\arrow/{-->}/[baseres`xres;]
\arrow/{>}/[xrings`yrings;]
\arrow/{.>}/[yrings`eup;]
\arrow/{-->}/[xres`yringsright;]
\node wxrings(800,-800)[{\prod\mathcal{O}_{j+1,t}}]
\node wxres(1600,-800)[{\prod{\kappa_{j+1,t}}}]
\node wbase(800,-1200)[{C_{j+1}}]
\node wbaseres(1600,-1200)[{C_{j+1}/{\eta_{j+1}}}]
\node ebot(1600,-1600)[\vdots]
\arrow/{>>}/[wxrings`wxres;]
\arrow/{>}/[wbase`wxrings;\gamma]
\arrow/{@{.>}@/^2em/}/[wbase`baseres;]
\arrow/{>>}/[wbase`wbaseres;]
\arrow/{-->}/[wbaseres`wxres;]
\arrow/{.>}/[wxrings`baseres;]
\arrow/{.>}/[ebot`wbaseres;]
\efig
\]
and going to $j+1$, the above defines $\mathcal{O}_{j+1,t}^{\ast}$ as in%
\[%
\bfig\node xrings(0,0)[{\prod\mathcal{O}^{\ast}_{j,t}}]
\node xres(800,0)[{\prod\kappa^{\ast}_{j,t}}]
\node base(0,-400)[{C_j}]
\node baseres(800,-400)[{\prod\mathcal{O}^{\ast}_{j+1,t}}]
\node baseresright(1600,-400)[{\prod\mathcal{{\kappa}}_{j+1,t}^{\ast}}]
\arrow/{>>}/[xrings`xres;]
\arrow/{>}/[base`xrings;]
\arrow/{.>}/[baseres`xres;]
\node wxrings(800,-800)[{\prod\mathcal{O}_{j+1,t}}]
\node wxres(1600,-800)[{\prod{\kappa}_{j+1,t}}]
\node wbase(800,-1200)[{C_{j+1}}]
\node wbaseres(1600,-1200)[{C_{j+1}/{\eta_{j+1}}}]
\node ebot(1600,-1600)[\vdots]
\arrow/{>>}/[wxrings`wxres;]
\arrow/{>}/[wbase`wxrings;]
\arrow/{@{.>}@/^2em/}/[wbase`xres;]
\arrow/{>>}/[wbase`wbaseres;]
\arrow/{>>}/[baseres`baseresright;]
\arrow/{-->}/[wbaseres`wxres;]
\arrow/{-->}/[wxres`baseresright;]
\arrow/{>}/[wxrings`baseres;]
\arrow/{.>}/[ebot`wbaseres;]
\efig
\]

This finishes the recursive definition along $j$.

\subsubsection{Step 3: A single field factor}

If we choose a field factor $\kappa_{0,t}$ of $C_{0}$, we get a corresponding
idempotent $e$, and cutting out the respective field factor from the above
Figure, induces a canonical choice of an index $t$ in each row and only these
factors will remain after applying $e$. For the rest of the proof, we work
exclusively with this chosen factor and define%
\[
\mathcal{O}_{j}:=\mathcal{O}_{j,t}^{\ast}\qquad\text{and}\qquad k_{j}%
:=\kappa_{j,t}^{\ast}\text{,}%
\]
so that $k_{j}$ is the residue field of the complete discrete valuation ring
$\mathcal{O}_{j}$. Using this new name, we see that we have finite ring
extensions $C_{j}\rightarrow\mathcal{O}_{j}$. We arrive at the diagram%
\begin{equation}%
\bfig\node x(0,1200)[K]
\node y(0,900)[\mathcal{O}_{1}]
\node yy(-300,600)[C_{1}]
\node z(300,900)[k_1]
\node w(300,600)[\mathcal{O}_{2}]
\node ww(0,300)[C_{2}]
\node u(600,600)[k_2]
\node v(600,300)[\vdots]
\arrow/>/[yy`y;]
\arrow/>/[ww`w;]
\arrow/{^{(}->}/[y`x;]
\arrow/{->>}/[y`z;]
\arrow/{^{(}->}/[w`z;]
\arrow/{->>}/[w`u;]
\arrow[v`u;]
\efig
\label{lFigB}%
\end{equation}
While it is outside the general pattern, it can easily be shown that we also
have a finite ring map $C_{0}\rightarrow K$; in fact this is the projection on
a direct factor of the ring $C_{0}$. Since $K$ is an $n$-local field, the
$k_{j}$ are $(n-j)$-local fields and $\mathcal{O}_{i+1}$ their first rings of integers.

\begin{keypoint}
\label{marker_KeyPoint}There is more structure:

\begin{enumerate}
\item (as Tate objects) Now $k_{0}:=K$, as a factor of $C_{0}$, is an $n$-Tate
object and inductively $\mathcal{O}_{j+1}$ and its maximal ideal
$\mathfrak{m}\subset\mathcal{O}_{j+1}$ are Tate lattices in $k_{j}$, and the
quotient $\mathcal{O}_{j+1}/\mathfrak{m}=k_{j+1}$ is an $(n-1)$-Tate object.
So all the $\mathcal{O}_{j}$ are objects in $\mathsf{Pro}^{a}(\left.
(n-j)\text{-}\mathsf{Tate}\right.  )$, and by sandwiching%
\begin{equation}
\mathfrak{m}^{N}\mathcal{O}_{j}\subseteq\eta_{j}\mathcal{O}_{j}\subseteq
\mathfrak{m}\mathcal{O}_{j}\label{lda2}%
\end{equation}
the morphism $C_{j}\rightarrow\mathcal{O}_{j}$ turns out to come from a
morphism of Pro-diagrams and thus the $C_{j}\rightarrow\mathcal{O}_{j}$ are
all morphisms in $\mathsf{Pro}^{a}(\left.  (n-j)\text{-}\mathsf{Tate}\right.
) $ as well.

\item (ST modules) Moreover, if $k$ is perfect, $k_{0}=K$, as a factor of
$C_{0}$, is an ST $k$-module. This ST module structure on $C_{0}$ is precisely
the one employed by Yekutieli, see \cite[\S 6]{MR3317764} for a survey, or
\cite[Definition 3.2.1]{MR1213064} and \cite[Prop. 3.2.4]{MR1213064} for
details. This renders all $k_{j}$ and $\mathcal{O}_{j}$ ST modules by the
sub-space and quotient topologies. By Equation \ref{lda2} and \cite[Prop.
1.2.20]{MR1213064} the morphism $C_{j}\rightarrow\mathcal{O}_{j} $ is a
morphism of ST\ modules.
\end{enumerate}
\end{keypoint}

\subsubsection{Step 4: Coordinatization}

Next, we work by induction, starting from $j=n$ again and working downward:

\textit{Induction Hypothesis: }Assume we have constructed and fixed an
isomorphism%
\[
\xi_{j}:k_{j}[[t_{j}]]\overset{\sim}{\longrightarrow}\mathcal{O}_{j}\text{,}%
\]
simultaneously in the categories of rings, $k$-algebras, $\mathsf{Pro}%
^{a}(\left.  (n-j)\text{-}\mathsf{Tate}\right.  )$, ST\ modules, along with a
commutative square%
\[%
\bfig\node ai(0,0)[{\underset{i_{j}}{\underleftarrow{\lim}} \, C_{j}/\eta
_{j}^{i_{j}}}]
\node bi(800,0)[{C_{j}/{\eta}_{j},}]
\node ci(0,500)[{\underset{i_{j}}{\underleftarrow{\lim}} \, k_{j}%
[[t_{j}]]/t_{j}^{i_{j}}}]
\node di(800,500)[k_{j}]
\arrow/{>>}/[ai`bi;]
\arrow/{>>}/[ci`di;]
\arrow/{>}/[ai`ci;]
\arrow/{>}/[bi`di;]
\efig
\]
where the right-hand side arrows are the quotient maps (in all categories in
question), and the upward arrows are

\begin{itemize}
\item (in rings resp. $k$-algebras) finite extensions,

\item (in ST\ modules) morphisms of ST\ modules,

\item (in Tate objects) on the left, a morphism of $\mathsf{Pro}^{a}(\left.
(n-j)\text{-}\mathsf{Tate}\right.  )$ objects, on the right in $\left.
(n-j)\text{-}\mathsf{Tate}\right.  $.\medskip
\end{itemize}

Let us now perform the induction: We start with $j:=n$. The finiteness of the
diagonal ring morphisms in Figure \ref{lFigB} yields the lower commutative
square in%
\[%
\bfig\node ai(0,0)[{\underset{i_{n}}{\underleftarrow{\lim}} \, C_{n}/\eta
_{n}^{i_{n}}}]
\node bi(800,0)[{C_{n}/{\eta}_{n}}.]
\node ci(0,500)[{  \underset{i_{n}}{\underleftarrow{\lim}} \, \mathcal{O}%
_{n}/\mathfrak{m}_{n}^{i_{n}} }]
\node di(800,500)[k_{n}]
\node ci2(0,1000)[{\underset{i_{n}}{\underleftarrow{\lim}} \, k_{n}%
[[t_{n}]]/t_{n}^{i_{n}}}]
\node di2(800,1000)[k_{n}]
\arrow/{>}/[ci2`ci;{{\xi}_{n}}]
\arrow/{=}/[di2`di;]
\arrow/{>>}/[ai`bi;]
\arrow/{>>}/[ci`di;]
\arrow/{>>}/[ci2`di2;]
\arrow/{>}/[ai`ci;]
\arrow/{>}/[bi`di;]
\efig
\]
By Cohen's Structure Theorem, we can find an isomorphism $\xi_{n}$ such that
we may attach the upper commutative square to this diagram. The claims about
the ST module morphisms, resp. $\mathsf{Pro}^{a}(\left.  \left.  0\right.
\text{-}\mathsf{Tate}\right.  )$, resp. $\left.  \left.  0\right.
\text{-}\mathsf{Tate}\right.  $, are all immediate.

Now, we establish the induction step: Suppose the case $j+1$ has been dealt
with, and we want to prove the induction hypothesis for $j$. The finiteness of
the diagonal ring morphisms in Figure \ref{lFigB} yields the lower commutative
square in%
\[%
\bfig\node ai(0,0)[{\underset{i_{j}}{\underleftarrow{\lim}} \, C_{j}/\eta
_{j}^{i_{j}}}]
\node bi(800,0)[{C_{j}/{\eta}_{j}},]
\node ci(0,500)[{  \underset{i_{j}}{\underleftarrow{\lim}} \, \mathcal{O}%
_{j}/\mathfrak{m}_{j}^{i_{j}} }]
\node di(800,500)[k_{j}]
\node ci2(0,1000)[{\underset{i_{j}}{\underleftarrow{\lim}} \, k_{j}%
[[t_{j}]]/t_{j}^{i_{j}}}]
\node di2(800,1000)[k_{j}]
\arrow/{>}/[ci2`ci;{{\xi}_{j}}]
\arrow/{=}/[di2`di;]
\arrow/{>>}/[ai`bi;]
\arrow/{>>}/[ci`di;]
\arrow/{>>}/[ci2`di2;]
\arrow/{>}/[ai`ci;]
\arrow/{>}/[bi`di;]
\efig
\]
where the upward arrows are finite ring morphisms. They also define morphisms
of Pro-objects as well as ST\ modules, by the\ Key\ Point
\ref{marker_KeyPoint}. Since $\mathcal{O}_{j}$ is an equicharacteristic
complete discrete valuation ring with residue field $k_{j}$, Cohen's Structure
Theorem allows us to pick a coefficient field isomorphic to $k_{j}$ in
$\mathcal{O}_{j}$, write $[-]_{\star}:k_{j}\hookrightarrow\mathcal{O}_{j}$,
and thus get an isomorphism of rings%
\begin{align*}
\xi_{j}:k_{j}[[t_{j}]]  & \longrightarrow\mathcal{O}_{j}\\
\sum_{s}a_{s}t_{j}^{s}  & \longmapsto\text{evaluate }\sum_{s}[a_{s}]_{\star
}t_{j}^{s}%
\end{align*}
with $a_{s}\in k_{j}$ and $t_{j}$ some (arbitrary) uniformizer of
$\mathcal{O}_{j}$. If $k$ is perfect, we can assume to have picked the
coefficient field as a sub-$k$-algebra and so that $\xi_{j}$ is a $k$-algebra
isomorphism. Otherwise, we must content ourselves with a ring isomorphism.
Rewrite this morphism as%
\begin{align}
\xi_{j}:\underset{i_{j}\geq1}{\underleftarrow{\lim}}k_{j}[[t_{j}%
]]/(t_{j}^{i_{j}})  & \longrightarrow\underset{i_{j}\geq1}{\underleftarrow
{\lim}}\mathcal{O}_{j}/\mathfrak{m}_{j}^{i_{j}}\label{lixa1}\\
\sum_{s}a_{s}t_{j}^{s}  & \longmapsto\text{evaluate }\sum_{s}[a_{s}]_{\star
}t_{j}^{s}\text{.}\nonumber
\end{align}
Now, if we can produce an entry-wise isomorphism between the Pro-diagrams
defined by either side of the morphism, and these are objects in a category
$\mathcal{C}$, this defines an isomorphism in $\mathsf{Pro}^{a}(\mathcal{C})$.

However, via $\xi_{j+1}$ this can be achieved%
\[
\xi_{j+1}:k_{j+1}[[t_{j+1}]]\overset{\sim}{\longrightarrow}\mathcal{O}%
_{j+1}\qquad\text{and}\qquad\operatorname*{Frac}\mathcal{O}_{j+1}%
=k_{j}\text{,}%
\]
and since by our induction hypothesis $\xi_{j+1}$ is an isomorphism in
$\mathsf{Pro}^{a}(\left.  (n-j-1)\text{-}\mathsf{Tate}\right.  )$, via the
field of fractions (resp. the corresponding colimit), this induces an
isomorphism%
\[
k_{j+1}((t_{j+1}))\overset{\sim}{\longrightarrow}k_{j}\qquad\text{in}%
\qquad\mathsf{Ind}^{a}\mathsf{Pro}^{a}(\left.  (n-j-1)\text{-}\mathsf{Tate}%
\right.  )\text{,}%
\]
and in fact in $(\left.  (n-j)\text{-}\mathsf{Tate}\right.  )$. Using this,
the evaluation $[a_{s}]_{\star}$ in Equation \ref{lixa1} becomes entry-wise an
isomorphism of $(\left.  (n-j)\text{-}\mathsf{Tate}\right.  )$-objects. It
follows that $\xi_{j}$, as defined in Equation \ref{lixa1}, is an isomorphism
in $\mathsf{Pro}^{a}(\left.  (n-j)\text{-}\mathsf{Tate}\right.  )$. For
ST\ modules, argue analogously (that is: carrying out the Pro-limit and
equipping it with the limit topology, respectively the colimit topology for
the colimit, the same argument shows that $\xi_{j}$ is an isomorphism of ST\ modules).

Finally, once the entire induction is done, we obtain a Tate object and
ST\ module isomorphism between $K$ and the multiple Laurent series. If $K$ is
perfect, this produces a parametrization of the $n$-local field and thus gives
an alternative proof that $K$ is a TLF (see Definition \ref{def_TLF}).
Finally, since%
\[
C_{0}=A(\triangle,\mathcal{O}_{X})\text{,}%
\]
this proves all our claims.

\subsection{Consequences}

\begin{theorem}
\label{Thm_ComparisonOfCubicalAlgebras}Suppose $X$ is a purely $n$-dimensional
reduced scheme of finite type over a field $k$ and $\triangle=\{(\eta
_{0}>\cdots>\eta_{n})\}$ a saturated flag.

\begin{enumerate}
\item (\cite[Theorem 5]{bgwTateModule}) There is a canonical isomorphism of
$n$-fold cubical algebras%
\[
E^{\operatorname*{Tate}}(\mathcal{O}_{X\triangle})\overset{\sim}%
{\longrightarrow}E_{\triangle}^{\operatorname*{Beil}}\text{.}%
\]

\item Suppose $k$ is perfect. Then for each field factor $K$ in $\mathcal{O}%
_{X\triangle}=\prod K$, cut out by the idempotent $e\in E_{\triangle
}^{\operatorname*{Beil}}$, there are canonical isomorphisms of $n$-fold
cubical algebras%
\[
eE_{\triangle}^{\operatorname*{Beil}}e\overset{\sim}{\longrightarrow
}E^{\operatorname*{Tate}}(K)\overset{\sim}{\longrightarrow}%
E^{\operatorname*{Yek}}(K)\text{.}%
\]

\item Suppose $k$ is perfect.\ Define $\triangle^{(i)}=(\eta_{i}>\eta
_{i+1}>\cdots>\eta_{n})$. Then $K$ admits a presentation%
\begin{align}
K  & =\underset{L_{1}}{\underrightarrow{\operatorname*{colim}}}\underset
{L_{1}^{\prime}}{\underleftarrow{\lim}}\frac{L_{1}}{L_{1}^{\prime}}\nonumber\\
& =\underset{L_{1}}{\underrightarrow{\operatorname*{colim}}}\underset
{L_{1}^{\prime}}{\underleftarrow{\lim}}\underset{L_{2}}{\underrightarrow
{\operatorname*{colim}}}\underset{L_{2}^{\prime}}{\underleftarrow{\lim}}%
\frac{L_{2}}{L_{2}^{\prime}}\nonumber\\
& \qquad\vdots\label{laum1}\\
& =\underset{L_{1}}{\underrightarrow{\operatorname*{colim}}}\underset
{L_{1}^{\prime}}{\underleftarrow{\lim}}\underset{L_{2}}{\underrightarrow
{\operatorname*{colim}}}\cdots\underset{L_{n}}{\underrightarrow
{\operatorname*{colim}}}\underset{L_{n}^{\prime}}{\underleftarrow{\lim}}%
\frac{L_{n}}{L_{n}^{\prime}}\nonumber
\end{align}
where, recursively, $L_{i+1}^{\prime}\hookrightarrow L_{i+1}$ are Yekutieli
lattices in $K$ (for $i=0$) resp. $L_{i}/L_{i}^{\prime}$ (for $1\leq i<n$).
But presenting $K$ as a direct summand of $A(\triangle,\mathcal{O}_{X})$, say
$K=eA(\triangle,\mathcal{O}_{K})$ with $e$ the idempotent, there is also such
a presentation,%
\begin{align}
eA(\triangle,\mathcal{O}_{K})  & =e\,\underset{L_{1}}{\underrightarrow
{\operatorname*{colim}}}\underset{L_{1}^{\prime}}{\underleftarrow{\lim}}%
\frac{L_{1}}{L_{1}^{\prime}}\nonumber\\
& =e\,\underset{L_{1}}{\underrightarrow{\operatorname*{colim}}}\underset
{L_{1}^{\prime}}{\underleftarrow{\lim}}\underset{L_{2}}{\underrightarrow
{\operatorname*{colim}}}\underset{L_{2}^{\prime}}{\underleftarrow{\lim}}%
\frac{L_{2}}{L_{2}^{\prime}}\nonumber\\
& \qquad\vdots\label{laum2}\\
& =e\,\underset{L_{1}}{\underrightarrow{\operatorname*{colim}}}\underset
{L_{1}^{\prime}}{\underleftarrow{\lim}}\underset{L_{2}}{\underrightarrow
{\operatorname*{colim}}}\cdots\underset{L_{n}}{\underrightarrow
{\operatorname*{colim}}}\underset{L_{n}^{\prime}}{\underleftarrow{\lim}}%
\frac{L_{n}}{L_{n}^{\prime}}\text{,}\nonumber
\end{align}
where $L_{i+1}^{\prime}\hookrightarrow L_{i+1}$ are Beilinson lattices for the
flag $\triangle^{(i)}$ in $\mathcal{O}_{\eta_{0}}$ (for $i=0$) resp.
$L_{i}/L_{i}^{\prime}$ (for $1\leq i<n$). Under an isomorphism%
\[
K\overset{\sim}{\longrightarrow}eA(\triangle,\mathcal{O}_{K})\text{,}%
\]
these presentations sandwich each other, i.e. levelwise (i.e. in each row of
Equations \ref{laum1} along with the same-numbered row in Equations
\ref{laum2}), the Yekutieli and Beilinson lattices pairwise sandwich each
other. And in fact, so they do with all Tate lattices.
\end{enumerate}
\end{theorem}

\begin{proof}
(1) See \cite[Theorem 5]{bgwTateModule}.\newline(2) We write%
\begin{equation}
A(\triangle,\mathcal{O}_{X})=\prod K_{j}\text{,}\label{lWims0}%
\end{equation}
where $K_{j}$ are the $n$-local field factors. Our $K$ is one of these
factors. By Theorem \ref{Thm_COA_StructureTheorem3} there is an isomorphism%
\begin{equation}
\xi:K\longrightarrow k^{\prime}((t_{1}))((t_{2}))\cdots((t_{n}))\text{,}%
\label{lWims1}%
\end{equation}
simultaneously as $k$-algebras (since we assume that $k$ is perfect), and
$n$-Tate objects in finite-dimensional $k$-vector spaces, and ST\ modules. By
the first part of the theorem,%
\[
E^{\operatorname*{Tate}}(\mathcal{O}_{X\triangle})\overset{\sim}%
{\longrightarrow}E_{\triangle}^{\operatorname*{Beil}}%
\]
and if $e$ denotes the idempotent cutting out the field factor in question,%
\[
E^{\operatorname*{Tate}}(K)=eE^{\operatorname*{Tate}}(\mathcal{O}_{X\triangle
})e\overset{\sim}{\longrightarrow}eE_{\triangle}^{\operatorname*{Beil}%
}e\text{.}%
\]
On the other hand, since $\xi$ is also an isomorphism of $n$-Tate objects, it
clearly preserves endomorphism algebras, and therefore
\[
E^{\operatorname*{Tate}}(K)\cong E^{\operatorname*{Yek}}(K)
\]
by Theorem \ref{Theorem_LaurentYekIsTate}.\newline(3) Before we prove this, we
should explain that this follows from a very general principle: If
$\mathcal{C}$ is any idempotent complete exact category and an object
$X\in\mathsf{Tate}^{el}(\mathcal{C})$ can be presented as%
\[
X:=\underset{L_{i}}{\underrightarrow{\operatorname*{colim}}}\underset{L_{j}%
}{\underleftarrow{\lim}}\,\frac{L_{i}}{L_{j}}\text{,}%
\]
where $L_{i}\hookrightarrow L_{j}$ (for $i\leq j$) are Tate lattices in it,
then for every Tate lattice $L$, which need not be among these in the
presentation, there exist indices $i_{\vee},i_{\wedge}$ such that%
\[
L_{i_{\vee}}\hookrightarrow L\hookrightarrow L_{i_{\wedge}}\text{,}%
\]
i.e. arbitrary Tate lattices can be sandwiched by the lattices from the
collection $\{L_{i}\}_{i\in I}$ (\textit{Details: }The relevant underpinning
result is \cite[Theorem 6.7]{TateObjectsExactCats}. In fact, we have already
used exactly this kind of argument in the proof of Lemma \ref{lemma:Ocfin} and
we refer the reader to this proof for a complete discussion).

Obviously, as this property holds true for arbitrary idempotent complete exact
categories $\mathcal{C}$, it means that we can (inductively)\ also apply it to
objects in $n$-Tate categories. That is, if we have an object of the shape%
\[
X=\underset{L_{1}}{\underrightarrow{\operatorname*{colim}}}\underset
{L_{1}^{\prime}}{\underleftarrow{\lim}}\underset{(n-1)\text{-Tate object}%
}{\underbrace{\underset{L_{2}}{\underrightarrow{\operatorname*{colim}}}%
\cdots\underset{L_{n}}{\underrightarrow{\operatorname*{colim}}}\underset
{L_{n}^{\prime}}{\underleftarrow{\lim}}\frac{L_{n}}{L_{n}^{\prime}}}}%
\]
in an $n$-Tate category, where the $n$-Tate object is presented by quotients
$L_{1}/L_{1}^{\prime}$, which are $(n-1)$-Tate objects, and each $L_{1}%
/L_{1}^{\prime}$ by quotients $L_{2}/L_{2}^{\prime}$, which are $(n-2)$-Tate
objects etc., then levelwise, i.e. for the rightmost colimit-limit pair in
each of the following rows%
\begin{align*}
X  & =\underset{L_{1}}{\underrightarrow{\operatorname*{colim}}}\underset
{L_{1}^{\prime}}{\underleftarrow{\lim}}\frac{L_{1}}{L_{1}^{\prime}}\\
& =\underset{L_{1}}{\underrightarrow{\operatorname*{colim}}}\underset
{L_{1}^{\prime}}{\underleftarrow{\lim}}\underset{L_{2}}{\underrightarrow
{\operatorname*{colim}}}\underset{L_{2}^{\prime}}{\underleftarrow{\lim}}%
\frac{L_{2}}{L_{2}^{\prime}}\\
& \qquad\vdots\\
& =\underset{L_{1}}{\underrightarrow{\operatorname*{colim}}}\underset
{L_{1}^{\prime}}{\underleftarrow{\lim}}\underset{L_{2}}{\underrightarrow
{\operatorname*{colim}}}\cdots\underset{L_{n}}{\underrightarrow
{\operatorname*{colim}}}\underset{L_{n}^{\prime}}{\underleftarrow{\lim}}%
\frac{L_{n}}{L_{n}^{\prime}}\text{,}%
\end{align*}
each Tate lattice in the $i$-th row is sandwiched among lattices taken from
these systems $\{L_{i}\}$. By Corollary \ref{cor_Obvious} the left-hand side
in Equation \ref{lWims0} has the presentation%
\begin{equation}
A(\eta_{0}>\cdots>\eta_{n},\mathcal{O}_{X})=\underset{L_{1}}{\underrightarrow
{\operatorname*{colim}}}\underset{L_{1}^{\prime}}{\underleftarrow{\lim}}%
\cdots\underset{L_{n}}{\underrightarrow{\operatorname*{colim}}}\underset
{L_{n}^{\prime}}{\underleftarrow{\lim}}{}\frac{L_{n}}{L_{n}^{\prime}}%
\text{,}\label{laa1}%
\end{equation}
where the lattices $\{L_{i},L_{i}^{\prime}\}$ are Beilinson lattices of the
various levels, so these Beilinson lattices define Tate lattices in
$A(\eta_{0}>\cdots>\eta_{n},\mathcal{O}_{X})$. The TLF $K$ on the right-hand
side in Equation \ref{lWims0} also has such a presentation as an $n$-Tate
object%
\begin{equation}
K=\underset{L_{1}}{\underrightarrow{\operatorname*{colim}}}\underset
{L_{1}^{\prime}}{\underleftarrow{\lim}}\cdots\underset{L_{n}}{\underrightarrow
{\operatorname*{colim}}}\underset{L_{n}^{\prime}}{\underleftarrow{\lim}}%
{}\frac{L_{n}}{L_{n}^{\prime}}\text{,}\label{laa2}%
\end{equation}
where the lattices $\{L_{i},L_{i}^{\prime}\}$ are Yekutieli lattices, so these
Yekutieli lattices define Tate lattices in $K$. Since our isomorphism $\xi$ is
an isomorphism of TLFs, these Yekutieli lattices are precisely the same as
those in the TLF\ factor cut out from the ad\`{e}les $A(\eta_{0}>\cdots
>\eta_{n},\mathcal{O}_{X})$. Now we may run the above argument about levelwise
sandwiching lattices in either way: Either, using the presentation in Equation
\ref{laa1}, we deduce that all Tate lattices are sandwiched by Beilinson
lattices, but the Yekutieli lattices are such Tate lattices -- or using the
presentation in Equation \ref{laa2}, we deduce that all Tate lattices are
sandwiched by Yekutieli lattices, but Beilinson lattices are such Tate lattices.
\end{proof}

We can use Theorem \ref{Thm_COA_StructureTheorem3} to obtain a formulation `in coordinates':

\begin{definition}
Suppose $(A,\{I_{i}^{\pm}\}_{i=1,\ldots,n})$ is a Beilinson $n$-fold cubical
algebra. A \emph{system of good idempotents} consists of elements $P_{i}%
^{+}\in A$ with $i=1,\ldots,n$ such that the following conditions are met:

\begin{itemize}
\item $[P_{i}^{+},P_{j}^{+}]=0$,\qquad(pairwise commutativity)

\item $P_{i}^{+2}=P_{i}^{+}$,

\item $P_{i}^{+}A\subseteq I_{i}^{+}$,

\item $P_{i}^{-}A\subseteq I_{i}^{-}\qquad$(and we define $P_{i}%
^{-}:=\mathbf{1}_{A}-P_{i}^{+}$).
\end{itemize}
\end{definition}

This definition originates from \cite[Def. 14]{MR3207578}.

\begin{proposition}
Let $X/k$ be a reduced finite type scheme of pure dimension $n$ over a perfect
field $k$. If $\triangle$ is a saturated flag of points and $K$ a field factor
in%
\begin{equation}
\mathcal{O}_{X\triangle}=\prod_{m}K_{m}\text{,}\label{lwza2}%
\end{equation}
then an isomorphism%
\begin{equation}
K\simeq\kappa((t_{n}))((t_{n-1}))\cdots((t_{1}))\qquad\text{with}\qquad
\lbrack\kappa:k]<\infty\label{lwza1}%
\end{equation}
as in Theorem \ref{Thm_COA_StructureTheorem3} can be chosen so that for $f\in
E^{\operatorname*{Beil}}(K)$ we have the following characterization of the ideals:

\begin{enumerate}
\item $f\in I_{i}^{+}$ holds iff for all choices of $e_{1},\ldots,e_{i-1}%
\in\mathbf{Z}$ there exists some $e_{i}\in\mathbf{Z}$ such that instead of
needing to run over the $i$-th colimit in%
\[
\operatorname*{im}(f)\subseteq\left\{  \underset{e_{1}}{\underrightarrow
{\operatorname*{colim}}}\underset{j_{1}}{\underleftarrow{\lim}}\cdots
\widehat{\underset{e_{i}}{\underrightarrow{\operatorname*{colim}}}}%
\cdots\underset{e_{n}}{\underrightarrow{\operatorname*{colim}}}\underset
{j_{n}}{\underleftarrow{\lim}}\,\sum_{\alpha_{1}=-e_{1},\ldots,\alpha
_{n}=-e_{n}}^{j_{1}-1,\ldots,j_{n}-1}a_{\alpha_{1}\ldots\alpha_{n}}%
t_{1}^{\alpha_{1}}\cdots t_{n}^{\alpha_{n}}\right\}  \text{,}%
\]
it can, as indicated by the omission symbol $\widehat{(-)}$, be replaced by
this index $e_{i}$.

\item $f\in I_{i}^{-}$ holds iff for all $e_{1},\ldots,e_{i-1}\in\mathbf{Z} $
there exists $e_{i}\in\mathbf{Z}$ so that the $i$-th colimit can be replaced,
as in%
\[
\left\{  \underset{e_{1}}{\underrightarrow{\operatorname*{colim}}}%
\underset{j_{1}}{\underleftarrow{\lim}}\cdots\widehat{\underset{e_{i}%
}{\underrightarrow{\operatorname*{colim}}}}\cdots\underset{e_{n}%
}{\underrightarrow{\operatorname*{colim}}}\underset{j_{n}}{\underleftarrow
{\lim}}\,\sum_{\alpha_{1}=-e_{1},\ldots,\alpha_{n}=-e_{n}}^{j_{1}%
-1,\ldots,j_{n}-1}a_{\alpha_{1}\ldots\alpha_{n}}t_{1}^{\alpha_{1}}\cdots
t_{n}^{\alpha_{n}}\right\}  \subseteq\ker(f)\text{,}%
\]
by the index $e_{i}$.

\item Fix \emph{such} isomorphisms for all field factors $K_{m}$ in Equation
\ref{lwza2}. Denote by $\kappa_{m}$ the last residue field of $K_{m}$. If we
define the $\kappa_{m}$-linear maps%
\[
\left.  ^{m}P_{i}^{+}\right.  \,\sum a_{\alpha_{1}\ldots\alpha_{n}}%
t_{1}^{\alpha_{1}}\cdots t_{n}^{\alpha_{n}}=\sum_{\alpha_{i}\geq0}%
a_{\alpha_{1}\ldots\alpha_{n}}t_{1}^{\alpha_{1}}\cdots t_{n}^{\alpha_{n}%
}\qquad\text{(for }1\leq i\leq n\text{)}%
\]
on the right-hand side in Equation \ref{lwza1} for each field factor $K_{m}$,
then the aforementioned isomorphisms equip $\mathcal{O}_{X\triangle}$ with a
system of good idempotents.%
\begin{align*}
P_{i}^{+}:\mathcal{O}_{X\triangle}  & \longrightarrow\mathcal{O}_{X\triangle
}\\%
{\textstyle\prod\nolimits_{m=1}^{w}}
K_{m}  & \longrightarrow%
{\textstyle\prod\nolimits_{m=1}^{w}}
K_{m}\\
(x_{1},\ldots,x_{w})  & \longmapsto(\left.  ^{1}P_{i}^{+}\right.  x_{1}%
,\ldots,\left.  ^{w}P_{i}^{+}\right.  x_{w})\text{.}%
\end{align*}

\end{enumerate}
\end{proposition}

We stress that (3) would not be true for a randomly chosen field isomorphism
in Equation \ref{lwza1}.

\begin{proof}
(1) + (2): This is just unravelling properties that we have already
established by now. By Lemma \ref{lemma_indproreform2} we know that $f\in
I_{i\triangle}^{+}(K,K)$ holds if and only if $f$ admits a factorization%
\begin{equation}
\underset{L_{1}}{\underrightarrow{\operatorname*{colim}}}\underset
{L_{1}^{\prime}}{\underleftarrow{\lim}}\cdots\underset{L_{n}}{\underrightarrow
{\operatorname*{colim}}}\underset{L_{n}^{\prime}}{\underleftarrow{\lim}}%
\frac{L_{n}}{L_{n}^{\prime}}\longrightarrow\underset{N_{1}}{\underrightarrow
{\operatorname*{colim}}}\underset{N_{1}^{\prime}}{\underleftarrow{\lim}}%
\cdots\underset{N_{i}}{\underrightarrow{\widehat{\operatorname*{colim}}}%
}\cdots\underset{N_{n}}{\underrightarrow{\operatorname*{colim}}}%
\underset{N_{n}^{\prime}}{\underleftarrow{\lim}}\frac{N_{n}}{N_{n}^{\prime}%
}\text{,}\label{lcwb1}%
\end{equation}
where the $L_{(-)},L_{(-)}^{\prime},N_{(-)},N_{(-)}^{\prime}$ run over
suitable Beilinson lattices. This means that instead of the colimit over
$N_{i}$, the image factors through a fixed $N_{i}$ (allowed to depend on
$N_{1},N_{1}^{\prime},\ldots,N_{i-1},N_{i-1}^{\prime}$). In Theorem
\ref{Thm_COA_StructureTheorem3} we can pick the isomorphism in such a way that
it stems from an isomorphism of the underlying $n$-Tate objects. So this
isomorphism sends these Beilinson lattices to Tate lattices of $\kappa
((t_{n}))\cdots((t_{1}))$ with its standard $n$-Tate object structure. For
this Tate object structure, see Example
\ref{example_KatoIndProOfLaurentSeries}, i.e. slightly rewritten%
\begin{align*}
& \kappa((t_{n}))((t_{n-1}))\ldots((t_{1}))\\
& \qquad\qquad=\underset{e_{1}}{\underrightarrow{\operatorname*{colim}}%
}\underset{j_{1}}{\underleftarrow{\lim}}\cdots\underset{e_{n}}%
{\underrightarrow{\operatorname*{colim}}}\underset{j_{n}}{\underleftarrow
{\lim}}\frac{1}{t_{1}^{e_{1}}\cdots t_{n}^{e_{n}}}\kappa\lbrack t_{1}%
,\ldots,t_{n}]/(t_{1}^{j_{1}},\ldots,t_{n}^{j_{n}})\\
& \qquad\qquad=\underset{e_{1}}{\underrightarrow{\operatorname*{colim}}%
}\underset{j_{1}}{\underleftarrow{\lim}}\cdots\underset{e_{n}}%
{\underrightarrow{\operatorname*{colim}}}\underset{j_{n}}{\underleftarrow
{\lim}}\,\sum_{\alpha_{1}=-e_{1},\ldots,\alpha_{n}=-e_{n}}^{j_{1}%
-1,\ldots,j_{n}-1}a_{\alpha_{1}\ldots\alpha_{n}}t_{1}^{\alpha_{1}}\cdots
t_{n}^{\alpha_{n}}\text{.}%
\end{align*}
Now, as the image factors through some fixed $N_{i}$ in Equation \ref{lcwb1},
this is equivalent to factoring over some fixed $e_{i}\in\mathbf{Z}$. Stated
along with its dependencies on the other indices this becomes: For all
$e_{1},\ldots,e_{i-1}\in\mathbf{Z}$, there exists $e_{i}\in\mathbf{Z}$, so
that%
\[
\alpha_{i}<e_{i}\Rightarrow a_{\alpha_{1}\ldots\alpha_{n}}=0\text{.}%
\]
It is clear that we can run this argument backwards as well. The rest can be
done in an analogous fashion.

(3) For each fixed $m$, on $K_{m}$ we see that the $\left.  ^{m}P_{i}\right.
$ are pairwise orthogonal, therefore commuting, idempotents. On $\mathcal{O}%
_{X\triangle}$ we deduce that all $\left.  ^{m}P_{i}\right.  $ are again
pairwise orthogonal and then use that the sum of pairwise orthogonal
idempotents is again an idempotent. To check $P_{i}^{+}A\subseteq I_{i}^{+}$
and $P_{i}^{-}A\subseteq I_{i}^{-}$, one can just use $e_{i}:=0$ in (1) resp. (2).
\end{proof}

\section{\label{sect_DiffTypesOfLattices}Different types of lattices}

Suppose we look at some flag of points $\triangle=\{(\eta_{0}>\cdots>\eta
_{n})\}$ in a scheme $X$, say reduced, pure dimensional, and of finite type
over a perfect field $k$. In Theorem \ref{Thm_ComparisonOfCubicalAlgebras} we
have seen that a higher local field factor $K$ of the ad\`{e}les
$\mathcal{O}_{\triangle}\underset{\operatorname*{def}}{=}A(\triangle
,\mathcal{O}_{X})$ may be presented as%
\[
K=\underset{L_{1}}{\underrightarrow{\operatorname*{colim}}}\underset
{L_{1}^{\prime}}{\underleftarrow{\lim}}\underset{L_{2}}{\underrightarrow
{\operatorname*{colim}}}\cdots\underset{L_{n}}{\underrightarrow
{\operatorname*{colim}}}\underset{L_{n}^{\prime}}{\underleftarrow{\lim}%
}\,\frac{L_{n}}{L_{n}^{\prime}}%
\]
(e.g. in the category of $n$-Tate objects or as ST\ modules), where one may
either let the $L_{i},L_{i}^{\prime}$ run through Beilinson, Tate or Yekutieli
lattices. We had also seen that this implies that all these three families of
lattices sandwich each other, see Theorem
\ref{Thm_ComparisonOfCubicalAlgebras} for the precise statement. One may ask
for a much stronger property: Could it be true that there is an
(order-preserving) bijection between all these sets of lattices?

Indeed, at first sight, this looks promising: Using the presentation where all
$L_{i},L_{i}^{\prime}$ are Beilinson lattices, we have%
\begin{equation}
K=\underset{L_{1}}{\underrightarrow{\operatorname*{colim}}}\left(
\underset{L_{1}^{\prime}}{\underleftarrow{\lim}}\underset{\left.
(n-1)\text{-}\mathsf{Tate}(\mathsf{Vect}_{f})\right.  }{\underbrace
{\cdots\underset{L_{n}}{\underrightarrow{\operatorname*{colim}}}%
\underset{L_{n}^{\prime}}{\underleftarrow{\lim}}}{}}\frac{L_{n}}{L_{n}%
^{\prime}}\right)  \label{lMis1}%
\end{equation}
and thus for each Beilinson lattice $L_{1}$ we get a Pro-subobject of the
$n$-Tate object $K$. The quotient by the latter has the shape%
\[
\underset{L_{1}}{\underrightarrow{\operatorname*{colim}}}\widehat
{\underset{L_{1}^{\prime}}{\underleftarrow{\lim}}}\underset{\left.
(n-1)\text{-}\mathsf{Tate}(\mathsf{Vect}_{f})\right.  }{\underbrace
{\underset{L_{2}}{\underrightarrow{\operatorname*{colim}}}\cdots
\underset{L_{n}}{\underrightarrow{\operatorname*{colim}}}\underset
{L_{n}^{\prime}}{\underleftarrow{\lim}}{}}}\frac{L_{n}}{L_{n}^{\prime}%
}\text{,}%
\]
where $\widehat{(-)}$ denotes omission, and this is visibly an Ind-quotient in
the outer-most Tate category. Thus, rewriting the bracket in Equation
\ref{lMis1} as $L_{\triangle^{\prime}}$ (recall that this notation was defined
to mean $A_{X}(\triangle^{\prime},L)$ in Definition \ref{def_Mis1}), we have%
\[
L_{\triangle^{\prime}}\subseteq\mathcal{O}_{\triangle}%
\]
and this defines a Tate lattice in the $n$-Tate object $\mathcal{O}%
_{\triangle}$. Hence, there is a mechanism to associate Tate lattices to
Beilinson lattices. Suggestively, albeit somewhat vaguely, we could write%
\begin{equation}
\text{Beilinson lattices}\qquad\rightsquigarrow\qquad\text{Tate lattices.}%
\label{lMis2}%
\end{equation}
Moreover, the $\mathcal{O}_{\eta_{1}}$-module structure of the Beilinson
lattice induces%
\begin{align}
\mathcal{O}_{\eta_{1}}\otimes L  & \longrightarrow L\label{ltc1}\\
(\mathcal{O}_{\eta_{1}})_{\triangle^{\prime}}\otimes L_{\triangle^{\prime}}  &
\longrightarrow L_{\triangle^{\prime}}\text{.}\nonumber
\end{align}
Since the maximal ideals of $\mathcal{O}_{\triangle^{\prime}}$ lie over
$\eta_{1}$, we have $(\mathcal{O}_{\eta_{1}})_{\triangle^{\prime}}%
=\mathcal{O}_{\triangle^{\prime}}$. This makes $L_{\triangle^{\prime}}$ a
finitely generated $\mathcal{O}_{\triangle^{\prime}}$-module. By Theorem
\ref{TATE_StructureOfLocalAdelesProp} the normalization $(-)^{\prime}$ of
$\mathcal{O}_{\triangle^{\prime}}$ satisfies $(\mathcal{O}_{\triangle^{\prime
}})^{\prime}=\prod\mathcal{O}_{i}\subseteq\prod K_{i}=\mathcal{O}_{\triangle}%
$. Let $e$ be the idempotent cutting out $K_{i}$ from $\mathcal{O}_{\triangle
}$, and then also $\mathcal{O}_{i}$ from $(\mathcal{O}_{\triangle^{\prime}%
})^{\prime}$. Inside $\mathcal{O}_{\triangle}$, we can take $\mathcal{O}%
_{\triangle}$-spans of elements; in particular, $e(\mathcal{O}_{i}\cdot
L_{\triangle^{\prime}})$ defines a finitely generated $\mathcal{O}_{i}%
$-submodule of $K_{i}$. As $L$ was a Beilinson lattice, we have%
\[
\mathcal{O}_{\eta_{0}}\cdot L=\mathcal{O}_{\eta_{0}}%
\]
and as in Equation \ref{ltc1} this implies%
\[
(\mathcal{O}_{\eta_{0}})_{\triangle^{\prime}}\cdot L_{\triangle^{\prime}%
}=(\mathcal{O}_{\eta_{0}})_{\triangle^{\prime}}\text{,}%
\]
but the maximal ideals of the $(\mathcal{O}_{\eta_{1}})_{\triangle^{\prime}}%
$-module structure of $L_{\triangle^{\prime}}$ all lie over $\eta_{1}$, so as
the localization at $\eta_{0}$ inverts this, it follows that $(\mathcal{O}%
_{\eta_{0}})_{\triangle^{\prime}}=\mathcal{O}_{\triangle}$ and $(\mathcal{O}%
_{\eta_{0}})_{\triangle^{\prime}}=\mathcal{O}_{\triangle}$. Thus,
$\mathcal{O}_{\triangle}\cdot L_{\triangle^{\prime}}=\mathcal{O}_{\triangle}$
and therefore%
\[
e\mathcal{O}_{\triangle}=e(\mathcal{O}_{\triangle}\cdot L_{\triangle^{\prime}%
})\subseteq e(\mathcal{O}_{i}\cdot L_{\triangle^{\prime}})\subseteq
e\mathcal{O}_{\triangle}\text{.}%
\]
It follows that $\mathcal{O}_{i}\cdot L_{\triangle^{\prime}}\subseteq K_{i}$
is a Yekutieli lattice. It is not hard to show that such Yekutieli lattices
again define Tate lattices in $K_{i}$, using a similar argument as around
Equation \ref{lMis1}. This yields a further mechanism to produce Tate
lattices, this time yielding%
\[
\text{Beilinson lattices}\qquad\rightsquigarrow\qquad\text{Yekutieli
lattices}\qquad\rightsquigarrow\qquad\text{Tate lattices.}%
\]
Note that this is a different mechanism as in line \ref{lMis2}. We ask: Does
every\ Tate lattice arise this way?\medskip

It does not, and it is indeed very easy to find examples. For example, if $L$
is a Beilinson lattice, it is by definition an $\mathcal{O}_{\eta_{1}}%
$-module. As a result,%
\[
\mathcal{O}_{\triangle^{\prime}}\otimes L_{\triangle^{\prime}}\longrightarrow
L_{\triangle^{\prime}}%
\]
defines an $\mathcal{O}_{\triangle^{\prime}}$-module structure on
$L_{\triangle^{\prime}}$. But Tate lattices have no reason to carry any module
structure at all. For example, let $x_{1},\ldots,x_{r}$ an arbitrary family of
elements in $\mathcal{O}_{\triangle}$, some `noise'. Then if $L\subseteq
\mathcal{O}_{\triangle}$ is a Tate lattice, so is $L+k\otimes\{x_{1}%
,\ldots,x_{r}\}\subseteq\mathcal{O}_{\triangle}$ (if $R$ is a ring and $M$ an
$R$-module, we write $R\otimes\{v_{1},\ldots\}$ to denote the $R$-submodule of
$M$ which is spanned by elements $v_{1},\ldots$). This is true for the simple
reason that adding or quotienting out some finite-dimensional vector space
will not affect being a Pro- or Ind-object inside $\mathsf{Tate}%
^{el}(\mathsf{Vect}_{f})$. This shows that a general Tate lattice need not
come from a Beilinson or Yekutieli lattice. The rest of this section will be
devoted to discussing a more sophisticated example, where a Tate and Yekutieli
lattice does carry (the natural!) module structure, but still does not come
from a Beilinson lattice.\medskip

Consider the affine $2$-space $\mathbf{A}^{2}=\operatorname*{Spec}k[s,t]$ and
the singleton flag $\triangle:=\{((0)>(s^{2}-t^{3})>(s,t))\}$. For the sake of
brevity, we employ the shorthand%
\[
A_{j}:=A(\eta_{j}>\cdots>\eta_{2},\mathcal{O}_{\mathbf{A}^{2}})\in
\mathsf{Rings}\text{,}%
\]
(we had already used this notation earlier; cf. Definition
\ref{def_AjRingsShorthand}) and we regard these only as commutative rings for
the moment. We compute%
\begin{align*}
A_{2}  & =k[[s,t]]\\
A_{1}  & =\underset{j}{\underleftarrow{\lim}}\,k[[s,t]]\left[  \left(
k[s,t]_{(s,t)}-(s^{2}-t^{3})\right)  ^{-1}\right]  /(s^{2}-t^{3})^{j}\text{.}%
\end{align*}
To understand $A_{1}$ as a ring, note that $k[[s,t]]$ is a $2$-dimensional
regular local domain. Already inverting only $t$ removes the maximal ideal, so
that $k[[s,t]][t^{-1}]$ is a $1$-dimensional regular domain $-$ since
$k[[s,t]]$ is regular, it is factorial, and so all height one primes are
principal. Therefore, $k[[s,t]][t^{-1}]$ is actually a principal ideal domain.
Hence, $k[[s,t]]\left[  \left(  k[s,t]_{(s,t)}-(s^{2}-t^{3})\right)
^{-1}\right]  $ is a localization thereof, and thus itself a principal ideal
domain. The ideal $(s^{2}-t^{3})$ is then necessarily maximal, thus completing
at this ideal yields a regular complete local ring of dimension one, i.e. a
discrete valuation ring. Hence, by Cohen's Structure Theorem (cf. Proposition
\ref{Prop_CohenStructureTheorem}) there exists an isomorphism $A_{1}%
\simeq\varkappa\lbrack\lbrack w]]$ with%
\[
\varkappa:=A_{1}/(s^{2}-t^{3})=k[[s,t]]/(s^{2}-t^{3})\left[  \overline{\left(
k[s,t]_{(s,t)}-(s^{2}-t^{3})\right)  ^{-1}}\right]  \text{,}%
\]
where the overline denotes that we refer to the images of these elements after
taking the quotient by $(s^{2}-t^{3})$. Thus, $\varkappa=\operatorname*{Frac}%
k[[s,t]]/(s^{2}-t^{3})$. Next,%
\[
A_{0}=\underset{j}{\underleftarrow{\lim}}\,A_{1}\left[  \left(  k[s,t]_{(s^{2}%
-t^{3})}-(0)\right)  ^{-1}\right]  /(0)^{j}\text{,}%
\]
so this is just the field of fractions of $A_{1}$. We therefore could draw a
diagram (except for the $k[[u]]$ entry, which will be constructed only below)%
\[%
\begin{array}
[c]{ccccc}%
A_{0} &  &  &  & \\
\uparrow &  &  &  & \\
A_{1} & \longrightarrow & A_{1}/(s^{2}-t^{3}) &  & \\
&  & \uparrow &  & \\
&  & k[[u]] &  & \\
&  & \uparrow &  & \\
&  & A_{2}/(s^{2}-t^{3}) & \longrightarrow & A_{2}/(s,t)\text{.}%
\end{array}
\]
The upper-right diagonal entries are fields, the lower-left diagonal entries
are one-dimensional local domains, the upward arrows are localizations, and
the rightward arrows quotients by the the respective maximal ideals. Note that
$A_{2}/(s^{2}-t^{3})\simeq k[[s,t]]/(s^{2}-t^{3})$ is the completed local ring
of the standard cusp singularity. In particular, it is not a normal ring. The
well-known integral closure inside the field of fractions is $k[[u]]$ via the
inclusion $t\mapsto u^{2}$, $s\mapsto u^{3}$. In particular, $\varkappa
:=A_{1}/(s^{2}-t^{3})\simeq k((u))$ since $\frac{s}{t}=\frac{u^{3}}{u^{2}}=u$
and $t$ is already a unit in $A_{1}$ as we had discussed above. In particular,
after these isomorphisms we may rephrase the previous diagram in the shape%
\[%
\begin{array}
[c]{ccccc}%
k((u))((w)) &  &  &  & \\
\uparrow &  &  &  & \\
k((u))[[w]] & \longrightarrow & k((u)) &  & \\
&  & \uparrow &  & \\
&  & k[[u]] &  & \\
&  & \uparrow &  & \\
&  & k[[s,t]]/(s^{2}-t^{3}) & \longrightarrow & k\text{.}%
\end{array}
\]
If we follow Beilinson's definition of a lattice, Definition
\ref{Definition_AdeleOperatorIdeals}, the lattices in $\mathcal{O}%
_{(0)}=k(s,t)$ are finitely generated $k[s,t]_{(s^{2}-t^{3})}$-submodules
$L\subseteq k(s,t) $ so that $k(s,t)\cdot L=k(s,t)$. A\ quotient of such, say
$L_{1}^{\prime}\subseteq L_{1}$, would be, for example,%
\[
\frac{L_{1}}{L_{1}^{\prime}}=\frac{k[s,t]_{(s^{2}-t^{3})}}{(s^{2}-t^{3}%
)^{N}\cdot k[s,t]_{(s^{2}-t^{3})}}\qquad\left(  \frac{L_{1}}{L_{1}^{\prime}%
}\right)  _{\triangle^{\prime}}=\frac{k((u))[[w]]}{w^{N}\cdot k((u))[[w]]}%
\text{,}%
\]
where $\triangle^{\prime}=((s^{2}-t^{3})>(s,t))$ and $N\geq0$ some integer.
Now, the Beilinson lattices inside $L_{1}/L_{1}^{\prime}$ are $k[s,t]_{(s,t)}%
$-modules, for example,%
\begin{align*}
\left(  t^{p}\cdot k[s,t]_{(s,t)}\right)  _{\triangle^{\prime\prime}}%
\equiv(u^{2p}\cdot k[u^{2},u^{3}]_{(u)})_{\triangle^{\prime\prime}}\equiv
u^{2p}\cdot k[[u,w]]  & \subset\left(  \frac{L_{1}}{L_{1}^{\prime}}\right)
_{\triangle^{\prime}}\\
(s^{p}\cdot k[s,t]_{(s,t)})_{\triangle^{\prime\prime}}\equiv(u^{3p}\cdot
k[u^{2},u^{3}]_{(u)})_{\triangle^{\prime\prime}}\equiv u^{3p}\cdot k[[u,w]]  &
\subset\left(  \frac{L_{1}}{L_{1}^{\prime}}\right)  _{\triangle^{\prime}%
}\text{.}%
\end{align*}
Here the symbol \textquotedblleft$\equiv$\textquotedblright\ really just means
equality, but is chosen to stress that we are working in the quotient $\left(
L_{1}/L_{1}^{\prime}\right)  _{\triangle^{\prime}}=L_{1\triangle^{\prime}%
}/L_{1\triangle^{\prime}}^{\prime}$.

Any Beilinson lattice $\mathcal{L}\subseteq L_{1}/L_{1}^{\prime}$ is generated
by polynomials in the variables $s,t$, and thus after applying $(-)_{\triangle
^{\prime}}$ is generated from elements of the shape $\sum_{i,j\geq0}%
a_{ij}u^{2i+3j}$ only. So we see that for $N=1$, there exists no Beilinson
lattice $\mathcal{L}\subseteq L_{1}/L_{1}^{\prime}$ so that $\mathcal{L}%
_{\triangle^{\prime\prime}}\equiv u\cdot k[[u,w]]\equiv u\cdot
k[[u]]\left\langle 1,w,\ldots,w^{N-1}\right\rangle $ (these agree in $\left(
L_{1}/L_{1}^{\prime}\right)  _{\triangle^{\prime}}$ since $w^{N}\equiv0$;
again writing \textquotedblleft$\equiv$\textquotedblright\ instead of equality
is meant to emphasize this notationally). In particular, $u\cdot k[[u,w]]$ is
a Tate lattice, an $(\mathcal{O}_{\mathbf{A}^{2}})_{\triangle^{\prime\prime}}%
$-module, yet cannot be of the shape $\mathcal{L}_{\triangle^{\prime\prime}}$
for a Beilinson lattice.

In summary, we have inequalities%
\[
\text{Beilinson lattices\quad}\neq\text{\quad Yekutieli lattices\quad}%
\neq\text{\quad Tate lattices,}%
\]
with a slight abuse of language since they each live in different categories
and objects.

\appendix

\section{\label{section_Appendix_ResultsFromComAlg}Results from commutative
algebra}

For the convenience of the reader, we list various facts from commutative
algebra which we need in various proofs:

\begin{fact}
Suppose $R$ is a Noetherian ring.\label{comalg_lemma_reminder}

\begin{enumerate}
\item \label{comalg_fact_D1}(\cite[Thm. 7.2 (a)]{MR1322960}) For an ideal $I$,
and a finitely generated $R$-module $M$, $\widehat{M_{I}}\cong M\otimes
_{R}\widehat{R_{I}}$.

\item \label{comalg_fact_D2}(\cite[Lemma 2.4]{MR1322960}) For a multiplicative
set $S$ and $M$ an $R$-module, we have $M[S^{-1}]\cong M\otimes_{R}R[S^{-1}]$.

\item \label{comalg_fact_D3}(\cite[Prop. 10.15 (iv)]{MR0242802}) For an ideal
$I$ and $\widehat{R_{I}}$ the $I$-adic completion. Then $I\widehat{R_{I}}$ is
contained in the Jacobson radical of $\widehat{R_{I}}$.

\item \label{comalg_fact_D4}(\cite[Cor. 2.16]{MR1322960}) Every Artinian ring
$R$ is isomorphic to $\prod R_{P}$ (i.e. a product of Artinian local rings),
where $P$ runs through the finitely many maximal primes of $R$.

\item \label{comalg_fact_D6}(\cite[Thm. 7.2(3)]{MR1011461}) Let $R$ be a ring
and $M$ an $R$-module. Then $M$ is faithfully flat over $R$ iff $M\neq
\mathfrak{m}M$ for every maximal ideal of $R$.\newline(\cite[Thm.
7.5(ii)]{MR1011461}) If $M$ is a faithfully flat $R$-algebra and $I$ an ideal
of $R$, $IM\cap R=I$.

\item \label{comalg_fact_D7}A reduced Artinian local ring is a field (for an
Artinian ring the maximal ideal is nilpotent, so if there are no non-trivial
nilpotent elements, we must have $\mathfrak{m}=0$).

\item \label{comalg_fact_D8}(\cite[Cor. 7.5]{MR1322960}) (Lifting of
Idempotents) Suppose $R$ is a Noetherian ring which is complete with respect
to an ideal $I$. Then any system of pairwise orthogonal idempotents
$\overline{e_{1}},\ldots,\overline{e_{r}}\in R/I$ lifts uniquely to pairwise
orthogonal idempotents $e_{1},\ldots,e_{r}\in R$.

\item \label{comalg_fact_D9}(\cite[Cor. 2.1.13]{MR2266432}) Suppose $R$ is a
reduced ring, $Q_{1},\ldots,Q_{r}$ its minimal primes, and $R^{\prime}$ the
integral closure in its total ring of quotients $\operatorname*{Quot}(R)$.
Then $R^{\prime}\cong\prod_{i=1}^{r}(R/Q_{i})^{\prime}$, where $(R/Q_{i}%
)^{\prime}$ denotes the integral closure in the field of fractions of
$R/Q_{i}$.

\item \label{comalg_fact_D10}(\cite[Thm. 21.10]{MR1125071}) Let $R$ be a ring
and $e$ an idempotent. Then $\operatorname*{rad}(eR)=e\operatorname*{rad}R$,
where $\operatorname*{rad}R$ denotes the Jacobson radical of $R$. Thus,
$eR/e\operatorname*{rad}R\cong\overline{e}(R/\operatorname*{rad}R)$, where
$\overline{e}$ denotes the image of $e$ in $R/\operatorname*{rad}R$.

\item \label{comalg_fact_D11}Suppose $R\rightarrow S$ is a faithfully flat
morphism. Let $P$ be a prime in $R$. Then $P$ is of the shape $Q\cap S$ for a
prime ideal $Q$ in $S$ minimal over $PS$. Conversely, for every prime ideal
$Q$ minimal over $PS$ we have $Q\cap S=P$.

\item \label{comalg_fact_D12}Let $(R,\mathfrak{m})$ be a Noetherian complete
local ring and $R\rightarrow S$ a finite extension. Then $S$ is semi-local and
decomposes as a finite product of complete local rings, $S\cong\prod
\widehat{S_{\mathfrak{m}^{\prime}}}$, where $\mathfrak{m}^{\prime}$ runs
through the finitely many maximal ideals of $S$.

\item \label{comalg_fact_D14}(\cite[Prop. 11.1]{MR1322960}) Suppose
$(R,\mathfrak{m})$ is a $1$-dimensional regular local ring. Then it is a
discrete valuation ring.

\item \label{comalg_fact_D15}Let $R$ be a reduced excellent ring and $I$ an
ideal. Then $\widehat{R_{I}}$ is also reduced.

\item \label{comalg_fact_D16}(\cite[33.I, Thm. 79]{MR575344}) Let $R$ be an
excellent ring, $I$ an ideal. Then the canonical map $R\rightarrow
\widehat{R_{I}}$ is regular.

\item \label{comalg_fact_D17}(\cite[33.B, Lemma 2]{MR575344}) Let
$R\rightarrow S$ be a regular, faithfully flat ring homomorphism. Then $R$ is
reduced iff $S$ is reduced.

\item \label{comalg_fact_D18}(\cite[Thm. 6.5]{MR0241408}) Let $R$ be a reduced
Noetherian local ring with geometrically regular formal fibers (e.g. an
excellent reduced local ring). Then there is a canonical bijection between the
maximal ideals of the normalization $R^{\prime}$ and the minimal primes of the
completion $\widehat{R}$.
\end{enumerate}
\end{fact}

\begin{proof}
All given references also give a full proof. As additional remarks: For
\ref{comalg_fact_D12} we refer to \cite[Prop. 4.3.2]{MR2266432}: Use that
$S/\mathfrak{m}S$ has finitely many minimal primes $\mathfrak{m}^{\prime}$ and
therefore by the Chinese Remainder Theorem $S/\mathfrak{m}S\cong
\prod(S/\mathfrak{m}S)_{\mathfrak{m}^{\prime}}$. Then use lifting of
idempotents. For \ref{comalg_fact_D15} just combine \ref{comalg_fact_D16} with
\ref{comalg_fact_D17} and the faithful flatness of completion.
\end{proof}

\bibliographystyle{amsalpha}
\bibliography{ollinewbib}

\def\polhk#1{\setbox0=\hbox{#1}{\ooalign{\hidewidth
  \lower1.5ex\hbox{`}\hidewidth\crcr\unhbox0}}} \def\cprime{$'$}
  \def\cprime{$'$} \def\cprime{$'$} \def\cprime{$'$}
\providecommand{\bysame}{\leavevmode\hbox to3em{\hrulefill}\thinspace}
\providecommand{\MR}{\relax\ifhmode\unskip\space\fi MR }
\providecommand{\MRhref}[2]{%
  \href{http://www.ams.org/mathscinet-getitem?mr=#1}{#2}
}
\providecommand{\href}[2]{#2}
\begin{thebibliography}{BGW15b}

\bibitem[Abr07]{MR2357687}
V.~Abrashkin, \emph{An analogue of the field-of-norms functor and of the
  {G}rothendieck conjecture}, J. Algebraic Geom. \textbf{16} (2007), no.~4,
  671--730. \MR{2357687 (2009b:11207)}

\bibitem[AM69]{MR0242802}
M.~F. Atiyah and I.~G. Macdonald, \emph{Introduction to commutative algebra},
  Addison-Wesley Publishing Co., Reading, Mass.-London-Don Mills, Ont., 1969.
  \MR{0242802 (39 \#4129)}

\bibitem[Be{\u\i}80]{MR565095}
A.~A. Be{\u\i}linson, \emph{Residues and ad\`eles}, Funktsional. Anal. i
  Prilozhen. \textbf{14} (1980), no.~1, 44--45. \MR{565095 (81f:14010)}

\bibitem[Be{\u\i}87]{MR923134}
\bysame, \emph{How to glue perverse sheaves}, {$K$}-theory, arithmetic and
  geometry ({M}oscow, 1984--1986), Lecture Notes in Math., vol. 1289, Springer,
  Berlin, 1987, pp.~42--51. \MR{923134 (89b:14028)}

\bibitem[BGW15a]{bgwTateModule}
O.~Braunling, M.~Groechenig, and J.~Wolfson, \emph{{Operator ideals in Tate
  objects}}, arXiv:1508.07880 [math.AG] (2015).

\bibitem[BGW15b]{bgwRelativeTateObjects}
\bysame, \emph{{Relative Tate objects and boundary maps in the $K$-theory of
  coherent sheaves}}, arXiv:1511.05941 (2015).

\bibitem[BGW16]{TateObjectsExactCats}
\bysame, \emph{{Tate Objects in Exact Categories}}, Moscow Mathematical Journal
  (to appear), 2016.

\bibitem[Bra14]{MR3207578}
O.~Braunling, \emph{Ad\`ele residue symbol and {T}ate's central extension for
  multiloop {L}ie algebras}, Algebra Number Theory \textbf{8} (2014), no.~1,
  19--52. \MR{3207578}

\bibitem[C{\'a}m13]{MR3161556}
A.~C{\'a}mara, \emph{Functional analysis on two-dimensional local fields},
  Kodai Math. J. \textbf{36} (2013), no.~3, 536--578. \MR{3161556}

\bibitem[C{\'a}m14]{MR3227342}
\bysame, \emph{Locally convex structures on higher local fields}, J. Number
  Theory \textbf{143} (2014), 185--213. \MR{3227342}

\bibitem[C{\'a}m15]{camaratopology}
\bysame, \emph{Topology on rational points over higher local fields}, Revista
  de la Real Academia de Ciencias Exactas, Físicas y Naturales. Serie A.
  Matemáticas. To appear. (2015).

\bibitem[Coh46]{MR0016094}
I.~S. Cohen, \emph{On the structure and ideal theory of complete local rings},
  Trans. Amer. Math. Soc. \textbf{59} (1946), 54--106. \MR{0016094 (7,509h)}

\bibitem[CPT15]{MR3366858}
T.~Chinburg, G.~Pappas, and M.~J. Taylor, \emph{{Higher adeles and non-abelian
  {R}iemann-{R}och}}, Adv. Math. \textbf{281} (2015), 928--1024. \MR{3366858}

\bibitem[Die67]{MR0241408}
J.~Dieudonn{\'e}, \emph{Topics in local algebra}, Edited and supplemented by
  Mario Borelli. Notre Dame Mathematical Lectures, No. 10, University of Notre
  Dame Press, Notre Dame, Ind., 1967. \MR{0241408 (39 \#2748)}

\bibitem[Eis95]{MR1322960}
D.~Eisenbud, \emph{Commutative algebra}, Graduate Texts in Mathematics, vol.
  150, Springer-Verlag, New York, 1995, With a view toward algebraic geometry.
  \MR{1322960 (97a:13001)}

\bibitem[Fes01]{MR1850194}
I.~B. Fesenko, \emph{Sequential topologies and quotients of {M}ilnor
  {$K$}-groups of higher local fields}, Algebra i Analiz \textbf{13} (2001),
  no.~3, 198--221. \MR{1850194 (2002e:11172a)}

\bibitem[Fes03]{MR2046602}
\bysame, \emph{Analysis on arithmetic schemes. {I}}, Doc. Math. (2003),
  no.~Extra Vol., 261--284 (electronic), Kazuya Kato's fiftieth birthday.
  \MR{2046602 (2005a:11186)}

\bibitem[Fes06]{MR2276855}
\bysame, \emph{Measure, integration and elements of harmonic analysis on
  generalized loop spaces}, Proceedings of the {S}t. {P}etersburg
  {M}athematical {S}ociety. {V}ol. {XII} (Providence, RI), Amer. Math. Soc.
  Transl. Ser. 2, vol. 219, Amer. Math. Soc., 2006, pp.~149--165. \MR{2276855
  (2007j:11166)}

\bibitem[Fes10]{MR2658047}
\bysame, \emph{Analysis on arithmetic schemes. {II}}, J. K-Theory \textbf{5}
  (2010), no.~3, 437--557. \MR{2658047 (2011k:14019)}

\bibitem[Fes15]{FesRR}
\bysame, \emph{{Geometric adeles and the Riemann-Roch theorem for one-cycles on
  surfaces}}, {Moscow Mathematical Journal, Volume 15, Issue 3, July-September
  2015 pp. 435-453.} (2015).

\bibitem[FK00]{MR1804915}
I.~B. Fesenko and M.~Kurihara (eds.), \emph{Invitation to higher local fields},
  Geometry \& Topology Monographs, vol.~3, Geometry \& Topology Publications,
  Coventry, 2000, Papers from the conference held in M{\"u}nster, August
  29--September 5, 1999. \MR{1804915 (2001h:11005)}

\bibitem[Fra65]{MR0180954}
S.~P. Franklin, \emph{Spaces in which sequences suffice}, Fund. Math.
  \textbf{57} (1965), 107--115. \MR{0180954 (31 \#5184)}

\bibitem[Fra67]{MR0222832}
\bysame, \emph{Spaces in which sequences suffice. {II}}, Fund. Math.
  \textbf{61} (1967), 51--56. \MR{0222832 (36 \#5882)}

\bibitem[FV02]{MR1915966}
I.~B. Fesenko and S.~V. Vostokov, \emph{Local fields and their extensions},
  second ed., Translations of Mathematical Monographs, vol. 121, American
  Mathematical Society, Providence, RI, 2002, With a foreword by I. R.
  Shafarevich. \MR{1915966 (2003c:11150)}

\bibitem[Gor08]{MR2489487}
S.~O. Gorchinski{\u\i}, \emph{Adelic resolution for homology sheaves}, Izv.
  Ross. Akad. Nauk Ser. Mat. \textbf{72} (2008), no.~6, 133--202. \MR{2489487
  (2010a:14017)}

\bibitem[HS06]{MR2266432}
C.~Huneke and I.~Swanson, \emph{Integral closure of ideals, rings, and
  modules}, London Mathematical Society Lecture Note Series, vol. 336,
  Cambridge University Press, Cambridge, 2006. \MR{2266432 (2008m:13013)}

\bibitem[Hub91a]{MR1105583}
A.~Huber, \emph{Ad\`ele f\"ur {S}chemata und {Z}ariski-{K}ohomologie},
  Schriftenreihe des {M}athematischen {I}nstituts der {U}niversit\"at
  {M}\"unster, 3.\ {S}erie, {H}eft 3, Schriftenreihe Math. Inst. Univ.
  M\"unster 3. Ser., vol.~3, Univ. M\"unster, M\"unster, 1991, p.~86.
  \MR{1105583 (92h:14014)}

\bibitem[Hub91b]{MR1138291}
\bysame, \emph{On the {P}arshin-{B}e\u\i linson ad\`eles for schemes}, Abh.
  Math. Sem. Univ. Hamburg \textbf{61} (1991), 249--273. \MR{1138291
  (92k:14024)}

\bibitem[Kap01a]{MR1800352}
M.~M. Kapranov, \emph{Double affine {H}ecke algebras and 2-dimensional local
  fields}, J. Amer. Math. Soc. \textbf{14} (2001), no.~1, 239--262
  (electronic). \MR{1800352 (2001k:20007)}

\bibitem[Kap01b]{KapranovSemiInfinite}
\bysame, \emph{{Semi-infinite symmetric powers}}, {arXiv:math/0107089v1
  [math.QA]}, 2001.

\bibitem[Kat78]{MR517332}
K.~Kato, \emph{A generalization of local class field theory by using
  {$K$}-groups. {II}}, Proc. Japan Acad. Ser. A Math. Sci. \textbf{54} (1978),
  no.~8, 250--255. \MR{517332 (80c:12021)}

\bibitem[Kat83]{MR726423}
\bysame, \emph{Class field theory and algebraic {$K$}-theory}, Algebraic
  geometry ({T}okyo/{K}yoto, 1982), Lecture Notes in Math., vol. 1016,
  Springer, Berlin, 1983, pp.~109--126. \MR{726423 (85f:11089)}

\bibitem[Kat00]{MR1804933}
\bysame, \emph{Existence theorem for higher local fields}, Invitation to higher
  local fields ({M}\"unster, 1999), Geom. Topol. Monogr., vol.~3, Geom. Topol.
  Publ., Coventry, 2000, pp.~165--195. \MR{1804933 (2002e:11173)}

\bibitem[Kel90]{MR1052551}
B.~Keller, \emph{Chain complexes and stable categories}, Manuscripta Math.
  \textbf{67} (1990), no.~4, 379--417. \MR{1052551 (91h:18006)}

\bibitem[Lam91]{MR1125071}
T.~Y. Lam, \emph{A first course in noncommutative rings}, Graduate Texts in
  Mathematics, vol. 131, Springer-Verlag, New York, 1991. \MR{1125071
  (92f:16001)}

\bibitem[Lef42]{MR0007093}
S.~Lefschetz, \emph{Algebraic {T}opology}, American Mathematical Society
  Colloquium Publications, v. 27, American Mathematical Society, New York,
  1942. \MR{0007093 (4,84f)}

\bibitem[Mat80]{MR575344}
H.~Matsumura, \emph{Commutative algebra}, second ed., Mathematics Lecture Note
  Series, vol.~56, Benjamin/Cummings Publishing Co., Inc., Reading, Mass.,
  1980. \MR{575344 (82i:13003)}

\bibitem[Mat89]{MR1011461}
\bysame, \emph{Commutative ring theory}, second ed., Cambridge Studies in
  Advanced Mathematics, vol.~8, Cambridge University Press, Cambridge, 1989,
  Translated from the Japanese by M. Reid. \MR{1011461 (90i:13001)}

\bibitem[Mor12]{MorrowHLF}
M.~Morrow, \emph{{Constructing higher dimensional local fields}}, unpublished,
  2012.

\bibitem[MZ95]{MR1363290}
A.~I. Madunts and I.~B. Zhukov, \emph{Multidimensional complete fields:
  topology and other basic constructions}, Proceedings of the {S}t.\
  {P}etersburg {M}athematical {S}ociety, {V}ol.\ {III}, Amer. Math. Soc.
  Transl. Ser. 2, vol. 166, Amer. Math. Soc., Providence, RI, 1995, pp.~1--34.
  \MR{1363290 (97b:12009)}

\bibitem[OP11]{MR2866188}
D.~V. Osipov and A.~N. Parshin, \emph{Harmonic analysis on local fields and
  adelic spaces. {II}}, Izv. Ross. Akad. Nauk Ser. Mat. \textbf{75} (2011),
  no.~4, 91--164. \MR{2866188 (2012h:11165)}

\bibitem[Osi07]{MR2314612}
D.~V. Osipov, \emph{Adeles on {$n$}-dimensional schemes and categories
  {$C_n$}}, Internat. J. Math. \textbf{18} (2007), no.~3, 269--279. \MR{2314612
  (2008b:14005)}

\bibitem[Osi08]{MR2388492}
\bysame, \emph{{$n$}-dimensional local fields and adeles on {$n$}-dimensional
  schemes}, Surveys in contemporary mathematics, London Math. Soc. Lecture Note
  Ser., vol. 347, Cambridge Univ. Press, Cambridge, 2008, pp.~131--164.
  \MR{2388492 (2009b:14044)}

\bibitem[Par75]{MR0401710}
A.~N. Parshin, \emph{Class fields and algebraic {$K$}-theory}, Uspehi Mat. Nauk
  \textbf{30} (1975), no.~1 (181), 253--254. \MR{0401710 (53 \#5537)}

\bibitem[Par76]{MR0419458}
\bysame, \emph{On the arithmetic of two-dimensional schemes. {I}.
  {D}istributions and residues}, Izv. Akad. Nauk SSSR Ser. Mat. \textbf{40}
  (1976), no.~4, 736--773, 949. \MR{0419458 (54 \#7479)}

\bibitem[Par78]{MR514485}
\bysame, \emph{Abelian coverings of arithmetic schemes}, Dokl. Akad. Nauk SSSR
  \textbf{243} (1978), no.~4, 855--858. \MR{514485 (80b:14014)}

\bibitem[Par84]{MR752939}
\bysame, \emph{Local class field theory}, Trudy Mat. Inst. Steklov.
  \textbf{165} (1984), 143--170, Algebraic geometry and its applications.
  \MR{752939 (85m:11086)}

\bibitem[Pre11]{MR2872533}
L.~Previdi, \emph{Locally compact objects in exact categories}, Internat. J.
  Math. \textbf{22} (2011), no.~12, 1787--1821. \MR{2872533}

\bibitem[Sch33]{MR1512831}
F.~K. Schmidt, \emph{Mehrfach perfekte {K}\"orper}, Math. Ann. \textbf{108}
  (1933), no.~1, 1--25. \MR{1512831}

\bibitem[Sch99]{MR1779315}
J.-P. Schneiders, \emph{Quasi-abelian categories and sheaves}, M\'em. Soc.
  Math. Fr. (N.S.) (1999), no.~76, vi+134. \MR{1779315 (2001i:18023)}

\bibitem[Yek92]{MR1213064}
A.~Yekutieli, \emph{An explicit construction of the {G}rothendieck residue
  complex}, Ast\'erisque (1992), no.~208, 127, With an appendix by Pramathanath
  Sastry. \MR{1213064 (94e:14026)}

\bibitem[Yek95]{MR1352568}
\bysame, \emph{Traces and differential operators over {B}e\u\i linson
  completion algebras}, Compositio Math. \textbf{99} (1995), no.~1, 59--97.
  \MR{1352568 (96g:14014)}

\bibitem[Yek15]{MR3317764}
\bysame, \emph{Local {B}eilinson-{T}ate operators}, Algebra Number Theory
  \textbf{9} (2015), no.~1, 173--224. \MR{3317764}

\bibitem[Zhu00]{MR1804916}
I.~Zhukov, \emph{Higher dimensional local fields}, Invitation to higher local
  fields ({M}\"unster, 1999), Geom. Topol. Monogr., vol.~3, Geom. Topol. Publ.,
  Coventry, 2000, pp.~5--18 (electronic). \MR{1804916 (2001k:11245)}

\end{thebibliography}

\end{document}